\documentclass[11pt,a4]{amsart}
\usepackage{amsthm,amssymb, amsmath,epsfig,color, comment,float, graphics, bm}
\usepackage[colorlinks=true,allcolors=blue]{hyperref}
\usepackage[ruled,vlined]{algorithm2e}
\usepackage{float}
\usepackage{lipsum}
\usepackage{setspace}
\usepackage{caption}
\usepackage{subcaption}
\usepackage{multirow}
 
\newcommand{\blue}{\color{darkblue}}

\newcommand{\black}{\color{black}}

\definecolor{darkblue}{rgb}{.1, 0.1,.8}
\definecolor{darkgreen}{rgb}{0,0.8,0.2}
\definecolor{darkred}{rgb}{.8, .1,.1}

\newtheorem{lemma}{Lemma}[section]

\newtheorem{theorem}[lemma]{Theorem}

\newtheorem{proposition}[lemma]{Proposition}
\newtheorem{definition}[lemma]{Definition}
\newtheorem{corollary}[lemma]{Corollary}
\newtheorem{example}[lemma]{Example}
\newtheorem{exercise}[lemma]{Exercise}
\newtheorem{remark}[lemma]{Remark}
\newtheorem{fig}[lemma]{Figure}
\newtheorem{tab}[lemma]{Table}

\newcommand{\bfQ}{{\bf Q}}

\newcommand{\bth}{\begin{theorem}}
\newcommand{\ethe}{\end{theorem}}

\newcommand{\bre}{\begin{remark}\em }
\newcommand{\ere}{\end{remark}}

\newcommand{\ble}{\begin{lemma}}
\newcommand{\ele}{\end{lemma}}

\newcommand{\pp}{point process}
\newcommand{\bde}{\begin{definition}}
\newcommand{\ede}{\end{definition}}
\newcommand{\bco}{\begin{corollary}}
\newcommand{\eco}{\end{corollary}}

\newcommand{\bpr}{\begin{proposition}}
\newcommand{\epr}{\end{proposition}}

\newcommand{\bexer}{\begin{exercise}}
\newcommand{\eexer}{\end{exercise}}

\newcommand{\bexam}{\begin{example}}
\newcommand{\eexam}{\end{example}}

\newcommand{\bfi}{\begin{fig}}
\newcommand{\efi}{\end{fig}}

\newcommand{\btab}{\begin{tab}}
\newcommand{\etab}{\end{tab}}

\newcommand{\rv}{random variable}

\newcommand{\sign}{{\rm sign}}

\newcommand{\var}{{\rm Var}}

\newcommand{\cov}{{\rm Cov}}

\newcommand{\bfTh}{\mbox{$\pmb{\Theta}$}}

\newcommand{\beao}{\begin{eqnarray*}}
\newcommand{\eeao}{\end{eqnarray*}\noindent}

\newcommand{\beam}{\begin{eqnarray}}
\newcommand{\eeam}{\end{eqnarray}\noindent}

\newcommand{\beqq}{\begin{equation}}
\newcommand{\eeqq}{\end{equation}\noindent}

\newcommand{\bce}{\begin{center}}
\newcommand{\ece}{\end{center}}
\newcommand{\bfzero}{{\bf 0}}
\newcommand{\barr}{\begin{array}}
\newcommand{\earr}{\end{array}}

\newcommand{\std}{\stackrel{d}{\rightarrow}}

\newcommand{\eqd}{\stackrel{d}{=}}

\newcommand{\vague}{\stackrel{\lower0.2ex\hbox{$\scriptscriptstyle
                    \it{v} $}}{\rightarrow}}
\newcommand{\weak}{\stackrel{\lower0.2ex\hbox{$\scriptscriptstyle
                    \it{w} $}}{\rightarrow}}
\newcommand{\what}{\stackrel{\lower0.2ex\hbox{$\scriptscriptstyle
                    \it{\hat{w}} $}}{\rightarrow}}

\newcommand{\bdis}{\begin{displaymath}}
\newcommand{\edis}{\end{displaymath}\noindent}

\renewcommand{\P}{\mathbb P}
\newcommand{\R}{\mathbb{R}}
\newcommand{\bfX}{\mathbf{X}}

\newcommand{\nto}{n\to\infty}

\newcommand{\wt}{\widetilde}

\newcommand{\vep}{\epsilon}

\newcommand{\fct}{function}

\newcommand{\seq}{sequence}

\newcommand{\bfx}{{\bf x}}
\newcommand{\bfB}{{\bf B}}

\newcommand{\bfy}{{\bf y}}
\newcommand{\bfA}{{\bf A}}
\newcommand{\bfG}{{\bf G}}

\newcommand{\bfZ}{{\bf Z}}

\newcommand{\bfa}{{\bf a}}

\def\1{\ensuremath{\mathrm{1}\hspace{-.35em} \mathrm{1}}} 

\def\E{{\mathbb E}}
\def\Cov{{\rm Cov}}
\def\Var{{\rm Var}}

\def\P{{\mathbb{P}}}
\def\R{\mathbb{R}}
\def\Z{\mathbb{Z}}

\def\a{{\alpha}}

\renewcommand{\le}{\ensuremath{\leqslant}}
\renewcommand{\ge}{\ensuremath{\geqslant}}

\newcommand{\introo}[2]{{\left]{#1,\,#2\,}\right[\kern1pt}}

\newcommand{\intrfo}[2]{{\left[{#1,\,#2}\right[\kern1pt}}

\begin{document}
\title{On the asymptotics of extremal $\ell^p-$blocks cluster inference }
\date{}
\today

\thanks{
}
\author[G. Buritic\'a]{Gloria Buritic\'a}
\address{MIA Paris-Saclay, Universit\'e Paris Saclay, AgroParisTech, INRAE, 22 place de l’Agronomie,
91123, Palaiseau, France}
\email{gloria.buritica@agroparistech.fr}

\author[O. Wintenberger]{Olivier Wintenberger}
\address{LPSM, Sorbonne Universit\'es\\
UPMC Universit\'e Paris 06\\ 
F-75005, Paris\\ France}
\address{Institut Pauli CNRS, Vienna University, Oskar Morgensern Platz 1, Wien 1090, Austria}
\email{olivier.wintenberger@sorbonne-universite.fr}

\begin{abstract}
Extremes occur in  stationary regularly varying time series as short periods with several large observations, known as extremal blocks. 
We study cluster statistics summarizing the behavior of functions acting on these extremal blocks. Examples of cluster statistics are the extremal index, cluster size probabilities, and other cluster indices.
The purpose of our work is twofold. First, we state the asymptotic normality of block estimators for cluster inference based on consecutive observations with large $\ell^p-$norms, for $p > 0$. 
The case $p=\a$, where $\a > 0$ is the tail index of the time series, has specific nice properties thus we analyze the asymptotics of block estimators when approximating $\a$ using the Hill estimator. 
Second, we 
verify the conditions we require on classical models such as linear models and solutions of  stochastic recurrence equations. 
Regarding linear models, we prove that the asymptotic variance of classical index cluster-based estimators is null as first conjectured in \cite{hsing:1993}.
We illustrate our findings on simulations. 
\end{abstract}
\maketitle

\section{Introduction}
We study stationary heavy-tailed time series $(\bfX_t)$ in $(\mathbb{R}^d,|\cdot|)$, with regularly varying distributions, and tail index $\alpha > 0$; cf. \cite{basrak:segers:2009}, a formal definition is conferred to  Section~\ref{sec:rv}.   
In this framework, extremal observations cluster:
an extreme value triggers a short period with numerous large observations.
This behavior is known to perturb classical inference procedures tailored for independent observations like high quantile inference; see \cite{embrechts:kluppelberg:mikosch:1997}.
This clustering effect can be summarized by the {\it extremal index}, initially introduced in \cite{leadbetter:1983} and \cite{leadbetter:lindgren:rootzen:1983}. We can interpret it as the inverse of the mean number of consecutive exceedances above a high threshold in a short period of time. 
In this article, we aim to infer statistics of the clustering effect by letting functionals act on consecutive observations with extremal behavior. For example, we can recover the extremal index from this setting and also other important indices of the extremes of the series. 
%

\par

For extremal cluster inference, we consider a sample $\bfX_{[1,n]}$ together with a sequence $(b_n)$, and we define the sample of disjoint blocks $(\mathcal{B}_j)_{j=1,\dots,m_n}$ as blocks of consecutive observations: 
\beam \label{eq:blocks}
\mathcal{B}_j &:=& (\bfX_{(j-1)b_n+1},\dots,\bfX_{jb_n}) \;= \; \bfX_{(j-1)b_n +[ 1,b_n]},
\eeam 
such that $b_n \to \infty$, $m_n = \lfloor n/b_n \rfloor \to \infty$, as $n \to  \infty$. 
 Following the $p-$clusters theory developed in \cite{buritica:mikosch:wintenberger:2021},
 the extremal behavior of the series is modeled by the conditional behavior of a block $\mathcal{B}_j$ given that its ${\ell}^p-$norm is large:
\beam\label{limit:cluster:process:tcl:intro}
    \P( \mathcal{B}_1/x_{b_n} \in A \, | \, \|\mathcal{B}_1\|_p > x_{b_n}) &\xrightarrow[]{w}& \P( Y\bfQ^{(p)} \in A\,), \qquad n \to  \infty,
\eeam
such that $Y$ is independent of the $p-$cluster $\bfQ^{(p)} \in {\ell}^p$, $\P(Y > y) = y^{-\alpha},$ for $y > 1$, and $\|\bfQ^{(p)}\|_p = 1$ a.s., for $p\in (0,\infty]$. The weak convergence holds for a family of shift-invariant continuity sets $A \subset \ell^p$, and $(x_n)$ is a suitable sequence satisfying $\P(\|\mathcal{B}_1\|_p > x_{b_n}) \to 0$, as $n \to \infty$.
The limit distribution in the right-hand side of \eqref{limit:cluster:process:tcl:intro} summarizes the extremal behavior of regularly varying $\ell^p-$blocks.
{
In a nutshell, the $(\alpha)-$Pareto component $Y$ models the magnitude of the rescaled extremal block $\|\mathcal{B}_1\|_p/x_{b_n}$, and the $p-$cluster $\bfQ^{(p)}$  describes the extremal dependencies from the normalized block $\mathcal B_1/\|\mathcal{B}_1\|_p$ when its $\ell^p-$norm reaches extreme levels.
In this way, $\bfQ^{(p)}$ allows us to describe 
the appearance of consecutive extremes in regularly varying time series.
}
\par
In this article we study the inference of $p-$cluster statistics of the form
\beam\label{eq:statistic:1}\label{eq:statistic}
\quad f^{\bfQ}(p) &=&  \E[f( Y\bfQ^{(p)})]\,,
\eeam
for suitable $\ell^p-$continuity functions $f:\ell^p \to \mathbb{R}$ which are invariant to the shift operator of sequences. 
{Examples of shift-invariant functions are 
$p-$norms, and also coordinate-wise averages.
These functions are typically good summaries of the extremal clustering behavior of blocks.
To stress the relation of the functional $f$ with the $p-$cluster statistic in~\eqref{eq:statistic:1} it is convenient to write $f = f(p)$.}
To infer these statistics, we use the disjoint block estimators proposed in \cite{buritica:mikosch:wintenberger:2021} defined as
\beam\label{eq:estimator:cluster:1}\label{eq:estimator:cluster}
\widehat{f^\bfQ}(p)   &:=& \frac{1}{k_n} \sum_{t=1}^{m_n} f({\mathcal B}_{t}/\|{\mathcal B}\|_{p,({k_n+1})})\1(\|{\mathcal B}_{t}\|_p >  \|{\mathcal B}\|_{p,({k_n+1})}),
\eeam 
where $ 
\|{\mathcal B}\|_{p,(1)} \ge  \|{\mathcal B}\|_{p,(2)} \ge  \dots \ge  \|{\mathcal B}\|_{p,(m_n)},$
denotes the sequence of order statistics of the $\ell^p-$norms of blocks defined in \eqref{eq:blocks}, and {$(k_n)$ is an integer sequence determining the number of extremal blocks that we select for inference, and $k_n \to \infty$, as $n \to \infty$. }
\par 
{ The case $p=\a$ is particularly relevant for two reasons. 
The first reason is that under mixing and anti-clustering conditions, choosing $(a_n)$ satisfying $n\,\P(|\bfX_1|>a_n)\to 1$, as $\nto$, \cite{buritica:meyer:mikosch:wintenberger:2021} prove that 
\beam\label{eq:ppconv}
\sum_{t=1}^n\vep_{a_n^{-1}\bfX_t} &\std& \sum_{i=1}^\infty \sum_{j=-\infty}^\infty\vep_{\Gamma_i^{-1/\a}\bfQ_{ij}}\,, \qquad \nto\,,
\eeam
where $\Gamma_1\le \Gamma_2\le \cdots $ are the points of an homogeneous Poisson process, $(\bfQ_{ij})_j$ are independent copies of the $\a-$cluster process $\bfQ:=\bfQ^{(\a)}$ independent of $(\Gamma_i)$. The extremal cluster dependencies of the series are fully modeled using the spectral cluster process $\bfQ$, and from it one can recover the distribution of $\bfQ^{(p)}$, for every $p \in (0,\infty]$,
using the change-of-norms equation given by \cite{buritica:mikosch:wintenberger:2021} and recalled in \eqref{eq:change:of:norms:2}. 
{In practice this means $\a-$cluster inference allows us to infer any statistic in the form of \eqref{eq:statistic}.
{Moreover, the choice $p=\alpha$ is also ideal for selecting extremal blocks in \eqref{eq:estimator:cluster:1} as it is less susceptible to local serial dependencies; see \cite{buritica:mikosch:wintenberger:2021} for a broader discussion.}
}
\par
The main goal of this article is to establish the asymptotic normality of the block estimators from Equation~\eqref{eq:estimator:cluster:1}, tailored for cluster inference. We state moment, mixing and bias assumptions yielding the existence of a sequence $(k_n)$, satisfying $k_n \to \infty$, $m_n/k_n \to \infty$ such that 
\beam\label{eq:central:limit:2} 
\sqrt{ k_n}\big(\, \widehat{f^\bfQ}(p)  - f^\bfQ(p) \, \big) &\xrightarrow[]{d}&  \mathcal{N}\big(0, \Var(f(Y\bfQ^{(p)}) )\, \big)\,, \quad n \to  \infty,
\eeam 
and the limit is a centered Gaussian distribution. 
Our inference methodology can be viewed as a Peaks Over Threshold over order statistics of blocks. 
For $p=\a$, fixing $k$ and letting first $\nto$ in \eqref{eq:estimator:cluster} implies {heuristically} from \eqref{eq:ppconv} that
\beao
\widehat{f^\bfQ}(\a)   
&\approx& \frac{1}{k} \sum_{i=1}^{k} f((\Gamma_i/\Gamma_{k+1})^{-1/\a}({\bfQ}_{it})_t)\,.
\eeao
Then, the simple expression of the asymptotic variance in \eqref{eq:central:limit:2} follows as $(\Gamma_i/\Gamma_{k+1})\stackrel{d}{=}(U_{k,i})$ where $U_{k,1}<\cdots<U_{k,k}$ are the ordered statistics of iid uniformly distributed $U_j$, $1\le j\le k$, and $U^{-1/\a}\stackrel{d}{=}Y$. This heuristic is extended to any $p-$cluster ${\bfQ}^{(p)}$, $p\in(0, \infty]$, via the change-of-norms in \eqref{eq:change:of:norms:2}, { and rigorously proved in this article using the theory of the tail empirical process for time series; cf. \cite{kulik:soulier:2020}.
}
\par
In general, for $p-$cluster inference, the function $f = f(p)$ can involve the tail-index $\alpha $ in its expression, meaning $f(p)  = f_\a(p)$. Furthermore, we already mentioned that the choice of $p=\alpha$ has 
{the advantage of being robust to serial dependencies and to fully characterize clusters of extreme values (cf. \cite{buritica:mikosch:wintenberger:2021}),} 
thus to implement these procedures we replace $\alpha$ with an estimate $\widehat \alpha$.
We then show the asymptotic normality of the $p-$cluster estimator $\widehat f^{\bfQ}_{\widehat \a}(p)$ when we let $1/\widehat \a$ equal the classical Hill estimator, and
we extend the analysis to cover ${\widehat \a}-$cluster inference.
Furthermore, we conduct simulations to illustrate that $\ell^{\widehat \a}-$block estimators  
are competitive both in terms of bias and variance for finite sample sizes. 
 \black 
\par 
Our asymptotic results highlight how introducing $\ell^p-$norm block order statistics in \eqref{eq:estimator:cluster:1}, instead of order statistics of the sample $(|\bfX_t|)$ as in \cite{drees:neblung:2021,cissokho:kulik:2021}, can lead to a better asymptotic variance for cluster inference. We give examples of variance reduction in the case of linear models with short-range dependence, for inference of classical indices. In our examples, the asymptotic variance of linear models $\Var(f(Y\bfQ^{(p)}) )$ is null because of the deterministic properties of the spectral $p-$cluster process of linear models. For linear models, the advantage of replacing thresholds with block maxima records was previously investigated in \cite{hsing:1993}. Existing works \cite{drees:rootzen:2010,drees:neblung:2021,cissokho:kulik:2021,kulik:soulier:2020} following \cite{hsing:1993} focus on cluster of exceedances inference such that $p=\infty$. Our asymptotic result comforts and extends the heuristics presented in \cite{hsing:1993} for $p=\infty$ and linear models to the case $p < \infty$ and general models. 
To prove the asymptotic normality of block estimators, we rely on the theory of empirical processes \cite{vandervaart:wellner:1996}, but adapted to block estimators. 
For this purpose, we build on the modern overview in \cite{kulik:soulier:2020}. 
To handle the asymptotics of extremal $\ell^p-$blocks, we build on the large deviation principles
studied in \cite{buritica:mikosch:wintenberger:2021}, and appeal to the $p-$cluster processes theory therein. 
\par 
\par

\par
{
The article is organized as follows. Preliminaries on mixing coefficients, regular variation, and the $p-$cluster theory of stationary time series are compiled in Section~\ref{sec:preliminaries:2}.
In Section~\ref{sec:main:result} we present our main result in Theorem~\ref{thm:main}, stating the asymptotic normality of the block estimators introduced in Equation \eqref{eq:estimator:cluster:1}. We work under mixing, moment, and bias conditions on the series that we also present in Section~\ref{sec:main:result}.
Section~\ref{sec:examples:cluster} studies examples of extremal cluster inference such as estimation of the extremal index, the cluster size probabilities, and the cluster index for sums. We conclude by verifying our conditions on classical models such as linear processes and stochastic recurrence equations in Section~\ref{sec:examples}. 
In the case of linear models with short-range dependence, Theorem~\ref{eq:thm:linear:tcl} states that the $\ell^p-$block estimators of all the aforementioned quantities have null-asymptotic variance. Thereby, they are super-efficient for cluster inference of important indices as conjectured by \cite{hsing:1993} for $p=\infty$. We illustrate the finite-sample performances of our estimators in Section~\ref{sec:numerical:experiments}. All proofs are deferred to Apendices~\ref{sec:consistency},  \ref{sec:asymptotic:normality}, \ref{sec:proof:thm:clt},
\ref{sec:profs:app:A}, \ref{sec:appendix:C}, \ref{sec:appendix:E}, and \ref{section:models}.
}


\subsection{Notation}
We consider stationary time series $(\bfX_t)$ taking values in $\mathbb{R}^d$, that we endow with a norm $|\cdot|$. Let $p\in (0,\infty]$, and $(\bfx_t) \in (\mathbb{R}^d)^\mathbb{Z}$. Define the $p-$modulus function $\|\cdot\|_p: (\mathbb{R}^d)^{\mathbb{Z}} \to [0,\infty]$ as
\beao\label{eq:norm} 
 \|(\bfx_t)\|^p_p &:=& \sum_{t \in \mathbb{Z}}|\bfx_t|^p\,,
\eeao 
and define the sequential space $\ell^p$ as 
$
\ell^p := \{ (\bfx_t) \in (\mathbb{R}^d)^{\mathbb{Z}} : \|(\bfx_t)\|_p^p < \infty\}\,,$
with the convention that, for $p=\infty$, the space $\ell^\infty$ refers to sequences with finite supremum norm. 
For any $p \in (0,\infty]$, the $p-$modulus functions induce a distance $d_p$ in $\ell^p$, and for $p \in [1,\infty]$, it defines a norm. Abusing notation, we call them all $\ell^p-$norms. Let $\tilde{\ell}^p = \ell^p/\sim$ be the shift-invariant quotient space where $(\bfx_t) \sim (\bfy_t)$ if and only if there exists $k \in \mathbb{Z}$ such that $\bfx_{t-k} = \bfy_t$, $t \in \mathbb{Z}$. We also consider the metric space $(\tilde{\ell}^p,\tilde{d}_p)$ such that for $[\bfx] , [\bfy] \in \tilde{\ell}^p$,
\beao 
  \tilde{d}_p([\bfx], [\bfy]) &:=& \inf_{k \in \mathbb{Z}}\{ d_p( \bfx_{t-k}, \bfy_t) , (\bfx_t) \in [\bfx], (\bfy_t) \in [\bfy] \},
\eeao 
and without loss of generality, we write an element $[\bfx]$ in $\tilde{\ell}^p$ also as $(\bfx_t)$. Further details on the shift-invariant spaces are deferred to \cite{buritica:mikosch:wintenberger:2021,basrak:planinic:soulier:2018}. 
\par 
The operator norm for $d\times d$ matrices, ${\bfA} \in \mathbb{R}^{d\times d}$, is defined as $|{\bfA}|_{op}:= \sup_{|\bfx| = 1}|\bfA \bfx|.$ The truncation operations of $(\bfx_t)$ at the level $\epsilon$, for $\epsilon > 0$, are defined by \beao
(\overline{\bfx}_t^\epsilon)&:=&(\bfx_t \1_{|\bfx_t|\le\epsilon})\,, \qquad \qquad (\underline{\bfx_t}_\epsilon )\; :=\; (\bfx_t \1_{|\bfx_t|>\epsilon})\,.
\eeao 
The notation $a \land b$ denotes the minimum between two constants $a, b \in \mathbb{Z}$, and $a \lor b$ denotes its maximum. We write $\log^+(x) := \log(x)\lor 0$, for $x \in (0,\infty)$. We sometimes write $\bfx$ for the sequence $\bfx := (\bfx_t) \in (\mathbb{R}^d)^{\mathbb{Z}}$.
Furthermore, for $a,b, \in \mathbb{R}$, and $a \le b$, we write as $\bfx_{[a,b]}$ the vector $(\bfx_t)_{t=a,\dots,b}$ taking values in $(\mathbb{R}^d)^{b-a+1}$. 
We write $\bfx_{[a,b]} \in \tilde{\ell}^p$, which means we take the natural embedding of $\bfx_{[a,b]}$ in $\tilde{\ell}^p$ defined by assigning zeros to undefined coefficients. It will be convenient
to write $\mathcal{G}_{+}(\tilde{\ell}^p)$ 
for the continuous non-negative \fct s on $(\tilde{\ell}^p,\tilde{d}_p)$ which vanish in a neighborhood of the origin.
\par 

\section{Preliminaries}\label{sec:preliminaries:2}
\subsection{Mixing coefficients}\label{definition:mixing}
Let $(\bfX_t)$ be an $\mathbb{R}^d-$valued strictly stationary time series defined over a probability space $( (\mathbb{R}^d)^{\mathbb{Z}}, \mathcal{A}, \P)$. The properties of stationary sequences are usually studied through mixing coefficients. Denote the past and future $\sigma-$algebras by 
\beao
\mathcal{F}_{t\le 0} \quad :=\quad   \sigma( (\bfX_t)_{t\le 0}), \qquad \mathcal{F}_{t\ge h} \quad:=\quad \sigma( (\bfX_t)_{t \ge h}),\qquad h\ge 1\,,
\eeao
respectively. We recall the definition of the mixing coefficients $(\beta_h)$ below
\beao
    \beta_h
        &:=& d_{TV}\big( \, \P_{ \mathcal{F}_{t \le 0} \otimes \mathcal{F}_{t \ge h} }\, , \,  \P_{ \mathcal{F}_{t \le 0}} \otimes \P_{\mathcal{F}_{t \ge h} } \,\big),
\eeao 
where $d_{TV}(\cdot, \cdot)$ is the total variation distance between two probability measures:
$( (\mathbb{R}^d)^{\mathbb{Z}}, \mathcal{A}, \P_1)$, $( (\mathbb{R}^d)^{\mathbb{Z}}, \mathcal{A}, \P_2)$, and
$
\P_1\otimes\P_2(A \times B) := \P_1(A)\P_2(B),
$
for $A, B \in \mathcal{A}$. For a summary on mixing conditions see \cite{bradley:2005,dedecker:doukhan:lang:rafael:louhichi:prieur:2007,rio:2017}. 
%


\bre\label{remark:markov:beta}
A detailed interpretation of the  $\beta-$mixing coefficients $(\beta_t)$ in terms of the total variation distance can be found in Chapter 1.2 in \cite{dedecker:doukhan:lang:rafael:louhichi:prieur:2007}. These mixing coefficients are well adapted while working with Markov processes. Indeed, a strictly stationary Harris recurrent Markov chain $(\bfX_t)$, satisfies $\beta_t \to 0$ exponentially fast as $t \to  \infty$; see Theorem 3.5 in \cite{bradley:2005}.  
\ere

\par 

\subsection{Regular variation}\label{sec:rv} 
We consider stationary time series $(\bfX_t)$ taking values in $(\mathbb{R}^d,|\cdot|)$ and regularly varying with tail index $\alpha > 0$: all its finite-dimensional vectors are multivariate regularly varying of the same index. In this case we write $(\bfX_t)$ satisfies $\bf RV_\alpha$. Borrowing the ideas in \cite{basrak:segers:2009}, $(\bfX_t)$ satisfies $\bf RV_\alpha$ if and only if, for all $h \ge 0$, there exists a vector $(\bfTh_t)_{|t| \le h}$, taking values in $(\mathbb{R}^{d})^{2h+1}$ such that 
\beam\label{eq:rv} 
 \P(x^{-1}(\bfX_t)_{|t| \le h} \in \cdot \,|\, |\bfX_0| > x) &\xrightarrow[]{d}& \P(Y (\bfTh_t)_{|t| \le h} \in \cdot), \quad x \to \infty,
\eeam 
where $Y$ is independent of $(\bfTh_t)_{|t| \le h}$ and $\P(Y > y ) = y^{-\alpha}, y > 1$. We call the sequence $(\bfTh_t)$, taking values in $(\mathbb{R}^d)^{\mathbb{Z}}$, the spectral tail process. 
\par 
The time series $(\bfTh_t)$ does not inherit the stationarity property of the series. Instead, the time-change formula of \cite{basrak:segers:2009} holds: for any $s,t \in \mathbb{Z}, s \le 0 \le t$ and for any measurable bounded function $f:(\mathbb{R}^d)^{t-s+1} \to \mathbb{R}$,
\begin{align}\label{eq:time:change:formula:1}
    \E[f(\bfTh_{s-i}, \dots, \bfTh_{t-i})\1(|\bfTh_{-i}| \not = 0)] \,= \, \E[|\bfTh_i|^\alpha \, f(\bfTh_{s}/|\bfTh_i|, \dots, \bfTh_{t}/|\bfTh_i|)].
\end{align}
\subsection{$\ell^p-$cluster processes}\label{sec:p:cluster}
Let $(\bfX_t)$ be a stationary time series  satisfying $\bf RV_\alpha$. For $p \in (0,\infty]$, we say the series admits a $p-$cluster process $\bfQ^{(p)} \in \tilde{\ell}^p$ if there exists a sequence $(x_n)$, satisfying
\beam\label{eq:cp:7}\label{eq:constant:cp:1:tcl}
\P(\|\bfX_{[1,n]}\|_p > x_n) &\sim& n\,c(p)\P(|\bfX_1| > x_n), \quad n \to \infty, 
\eeam 
with $c(p) \in (0,\infty)$,  $n\P(|\bfX_1| > x_n) \to 0$, and
\beam\label{limit:cluster:process}\label{limit:cluster:process:tcl}
    \P(\bfX_{[1,n]}/x_n \, \in \cdot \, | \, \|\bfX_{[1,n]}\|_p > x_n) &\xrightarrow[]{w}& \P(Y\bfQ^{(p)} \in \cdot \, ), \quad n \to  \infty,
\eeam
where $Y$ is independent of $\bfQ^{(p)} \in \tilde{\ell}^p$, $\P(Y > y) = y^{-\alpha},$ for $y > 1$, $\|\bfQ^{(p)}\|_p = 1$ a.s., and the limit in \eqref{limit:cluster:process} holds in $(\tilde{\ell}^p, \tilde{d}_p)$. We study below the anti-clustering and vanishing-small values conditions denoted {\bf AC}, ${\bf CS}_p$, respectively, which guarantee the existence of $\ell^p-$clusters. We rephrase next the Theorem 2.1. of \cite{buritica:mikosch:wintenberger:2021}.
\begin{proposition}\label{prop:existence:cluster:process}
Let $(\bfX_t)$ be a stationary time series  satisfying $\bf RV_\alpha$. Let $(x_n)$ be a sequence such that $n\,\P(|\bfX_1|> x_n) \to 0$, as $n \to \infty$, and $p > 0$. 
For all $\epsilon > 0$, $\delta > 0$,  assume \\[1pt]
    \item[{\bf AC}\,:] $\lim_{s\to \infty} \limsup_{n \to \infty} \P(\|\bfX_{[s,n]}\|_\infty > \epsilon \, x_n \,|\, |\bfX_1| > \epsilon \, x_n \,)=0,$ \\[1pt]
    \item[${\bf CS}_p$:] $ \lim_{\epsilon \downarrow 0}\limsup_{n \to  \infty } \frac{\P( \| \overline{\bfX_{[1,n]}/x_n}^\epsilon \|^p_p   > \delta)}{n \P(|\bfX_1| > x_n)} \;=\; 0.$\\[4pt]
Then, if $p \ge \alpha$, Equation $\eqref{eq:cp:7}$ holds with $c(p) = \E[\|\bfQ^{(\alpha)}\|_p^\alpha]$
 and $c(\infty) \le c(p) \le c(\alpha) = 1$, 
then $(\bfX_t)$ admits a $p-$cluster process $\bfQ^{(p)}$ in the sense of \eqref{limit:cluster:process}. 
If $p < \alpha$, existence of the $p-$cluster process holds if $\E[\|\bfQ^{(\alpha)}\|_p^\alpha] < \infty$.
\end{proposition}
We see from Proposition~\ref{prop:existence:cluster:process} assuming {\bf AC} and ${\bf CS}_\alpha$ implies the time series $(\bfX_t)$ admits an $\a-$cluster $\bfQ^{(\alpha)}$, where $\alpha > 0$, denotes the tail index. In this case, appealing to Proposition 3.1. in \cite{buritica:mikosch:wintenberger:2021}, we have
\beam\label{eq:alphacluster}  
           \bfQ \quad := \quad \bfQ^{(\alpha)} &\eqd& \bfTh/\|\bfTh\|_\alpha, \quad \in \tilde{\ell}^\alpha\,,
\eeam 
where $(\bfTh_t)$ is the spectral tail process from Equation~\eqref{eq:rv}.
Moreover, if ${\bf CS}_p$, ${\bf CS}_{p^\prime}$, and $ \E[\|\bfQ\|_{p}^\alpha] + \E[\|\bfQ\|_{p^\prime}^\alpha] <  \infty $ also hold, then the $p, p^\prime-$clusters exist and are related by the change-of-norms formula below 
\beam\label{eq:change:of:norms}
 \P(\bfQ^{(p)}\in\cdot)
&=& c(p)^{-1}\E[\|\bfQ\|_{p}^\alpha\1(\bfQ/\|\bfQ\|_{p}\in\cdot )]\\
&=& \dfrac{c(p')}{c(p)}\E[\|\bfQ^{(p')}\|_{p}^\alpha\1(\bfQ^{(p')}/\|\bfQ^{(p')}\|_{p}\in\cdot )].\label{eq:change:of:norms:2}
\eeam
Since $\|\bfQ^{(p)}\|_p=1$ a.s. for any $p\in(0,\infty]$, then $c(\alpha)=1$, and $\E[\|\bfQ^{(p')}\|_{p}^\alpha]=c(p)/c(p')$, where $c(p), c(p^\prime)$, are as in Equation~\eqref{eq:cp:7}.
In the following we denote by $\bfQ$ the $\a-$cluster as in  \eqref{eq:alphacluster}.

\begin{remark}\label{remark:cs:m0}
Using the {monotonicity} of norms, we see ${\bf CS}_p$ implies ${\bf CS}_{p^\prime}$, for $p^\prime>p>0$. 
If $p > \alpha$,  condition ${\bf CS}_p$ is always satisfied for sequences $(x_n)$ such that $n\P(|\bfX_1| > x_n) \to 0$, as $n \to \infty$. 
If $\alpha/2 < p \le \alpha$,  
then for all $\kappa > 0$,
${\bf CS}_p$ holds for short-range dependence models and sequences $(x_n)$ 
satisfying $n/x_n^{p\land (\alpha - \kappa')} \to 0$, as $n \to \infty$ (see remarks 5.1. and 5.2. in \cite{buritica:mikosch:wintenberger:2021}).
We verify this condition on classical models in Section~\ref{sec:examples}.
\end{remark}

\section{Asymptotics of $p-$cluster block estimators}\label{sec:main:result}

\subsection{Block estimators}
{
Let $(\bfX_t)$ be an $\mathbb{R}^d-$valued stationary time series satisfying $\bf RV_\alpha$.
For $p > 0$ fixed, assume the conditions of Proposition~\ref{prop:existence:cluster:process} hold for $p$, thus the series admits a $p-$cluster process $\bfQ^{(p)} \in \tilde{\ell}^p$, and \eqref{limit:cluster:process:tcl} holds for a sequence of high levels $(x_n)$ satisfying $\P(\|\bfX_{[1,n]}\|_p > x_n) \to 0$, as $n \to \infty$. 
Recall the $\ell^p-$block estimator in \eqref{eq:estimator:cluster} is tuned with the block lengths $(b_n)$, and the number $(k_n)$ of extremal blocks. 
The total number of disjoint blocks in a sample is denoted $(m_n)$ with $m_n = \lfloor n/b_n \rfloor $. 
To study the asymptotics of block estimators we assume the following relation between $(k_n)$, $(x_n)$ and $(b_n)$: 
\beam \label{eq:asymp:k}\label{eq:k:intro}
k \quad:=\quad  k_n(p)&=& \;\big\lfloor m_n\P(\|\mathcal B_1\|_p >  x_{b_n})\big\rfloor,
\eeam
holds.  
Similarly  we consider a sequence $(x'_n)$ satisfying \eqref{limit:cluster:process:tcl} for $p = \infty$, and we assume the sequence $(k^\prime_n)$ satisfies the following relation 
\beam
k^\prime \; := \; k_n^\prime &=& \lfloor  n\, \P(|\bfX_1| > x^\prime_{b_n})\rfloor , \quad n \to \infty. \label{eq:k:5}\label{eq:k:4}
\eeam
In what follows, the sequences $(x_n)$, $(b_n)$, $(m_n)$, $(k_n)$, $(x_n^\prime),$  $(k_n^\prime)$ that appear in this article satisfy conditions \eqref{eq:asymp:k}, \eqref{eq:k:5} above.

\subsection{Tail-index estimation}{ 
To estimate the tail-index $\alpha$ of the {regularly varying series,
we consider the Hill estimator:
\beam\label{eq:hill:estimator}  
\frac{1}{\widehat{\alpha}^n} \;  := \;  \frac{1}{\widehat{\alpha}^n(k_n^\prime)}  &:=& \frac{1}{k_n^\prime} \, \sum_{t=1}^{n} \log^{+}(|\bfX_t|/|\bfX|_{(k_n^\prime+1)}),
\eeam 
where $|\bfX|_{(1)} \ge |\bfX|_{(2)} \ge \cdots \ge  |\bfX|_{(n)}$,
and $k^\prime = k^\prime_n$ is a tuning sequence for \eqref{eq:hill:estimator}.
To study the asymptotic properties of the Hill estimator
we write it as a cluster statistic.  
Consider the functional $h : \tilde{\ell}^{\infty} \to [0,\infty)$ given by

\beam \label{eq:func:hill}
h(\bfx) &=& \sum_{t \in \mathbb{Z}}\log(|\bfx_t|)\1(|\bfx_t| > 1).
\eeam 
It is easy to see $h(\bfx) =  h(\bfx)\1(\|\bfx\|_\alpha > 1)$ and
\beao 
{h^{\bfQ}(\alpha) }
&=& \int_{0}^\infty
\sum_{t \in \mathbb{Z}}  \E\big[
	\log(y|\bfQ_t|) 
	\1( y|\bfQ_t| > 1 )  \big] d(-y^{-\a}) \; = \; \frac{1}{\alpha}.
\eeao 
We also introduce the counts of exceedances functional $e: \tilde{\ell}^{\infty} \to [0,\infty)$ given by
\beam \label{eq:func:exceedances}
e(\bfx) &:=& \sum_{j\in\mathbb{Z}} \1(|\bfx_t| > 1 ),
\eeam 
which also satisfies
$e(\bfx) = e(\bfx)\1(\|\bfx\|_\alpha > 1)$ and 
\beao   
{ e^{\bfQ}(\alpha) }
&=& \int_{0}^\infty
\sum_{t \in \mathbb{Z}} 
\E\big[ 
	\1( y|\bfQ_t| > 1 )  \big]d(-y^{-\a})  \; = \; 1.
\eeao

\par
{


\par 
{
\subsection{Consistency of block estimators}\label{subsec:consis:intro}

Let $f: \tilde{\ell}^p \to \mathbb{R}$ be a functional defining the $p-$cluster statistic in \eqref{eq:statistic}.
To stress its relation with $p$ it will be convenient to write $f = f(p)$.}
In addition, we define the functional $1(p)$ as }
\beam \label{eq:1:functional}
1(p)(\bfx) &:=& \1(\|\bfx\|_p > 1),
\eeam 
which satisfies $1^\bfQ(p) =  \P(Y \|\bfQ^{(p)}\|_p > 1) = 1$. 
Recall also the functionals $h(\a), e(\a) $ in \eqref{eq:func:hill} and \eqref{eq:func:exceedances} defining the Hill and exceedances estimators, respectively. 
\par 
In numerous examples, the functional $f(p)$ might also depend on $\a$ in its expression, meaning $f(p) = f_\a(p)$.
Hence, it will  also be useful to consider the family of functions $f_q(p): \tilde{\ell}^p \to \mathbb{R}$, indexed by $q$, for $q$ in a neighborhood of $\a$. 
\par 
We start by stating the consistency of 
$\ell^p-$block estimators
$\widehat{f_{\widehat \alpha}^{\bfQ}}(p)$
in \eqref{eq:estimator:cluster} where we replace 
all appearances of $\alpha$ by $\widehat \alpha$ in the function $f_\alpha(p)$. 
This is the purpose of Lemma~\ref{lem:con} for which we require the assumptions below. \\[2mm] 

\par 
\noindent 
   {\bf C}: 
   {  Let $f = f_\a(p) \in \mathcal{G}_+(\tilde{\ell}^p)$. 
    Assume there exist $\epsilon, \delta >0$ such that 
\beam\label{eq:C}\\
\nonumber
\limsup_{n \to \infty}
\frac{
\E\Big[
\sup_{  \tiny
\substack{ q \in [\alpha-\epsilon,\alpha + \epsilon]\\ 
u \in [1-\epsilon,1 + \epsilon]
}
}
\big(
f_{q} (\bfX_{[1,n]}/(ux_{n}) )\big)^{1+\delta}\1(\|\bfX_{[1,n]}\|_{p} > ux_{n}) \Big]
}
{\P(\|\bfX_{[1,n]}\|_p > x_{n})}< \infty,
\eeam 
and $\sup_{ \tiny
\substack{ q \in [\alpha-\epsilon, \alpha + \epsilon]\\
u \in [1-\epsilon, 1+\epsilon],
}
} \E[f_{q}(Y\bfQ/u) ] < \infty$. 
\par 
{
If $p = \alpha$, assume for every $q \in [\a - \epsilon, \a+\epsilon]$ that $ c(q) = \E[\|\bfQ\|_q^\a] < \infty$ 
and Equation~\eqref{eq:C} holds with  
$p$ replaced by $q$ in the numerator. }
\\[2mm] 
  {\bf S}:
   Let $f= f_\alpha(p) \in \mathcal{G}_+(\tilde{\ell}^p)$. Assume there exists $\epsilon > 0$ such that,  for all $q \in [\a - \epsilon, \a + \epsilon]$, for all $\bfx \in \tilde{\ell}^p \setminus\{0\}$, $f_q$ admits the Taylor development:
    \beam 
    f_q(\bfx) &=& f_\a(\bfx) + (q-\alpha) \tfrac{\partial f_q}{\partial q}|_{q=\a}(\bfx)+ {\frac{1}{2}}(q-\a)^2 \tfrac{\partial^2 f_q}{\partial q^2}|_{q=\xi}(\bfx),
    \eeam 
    where $\xi = \xi(\bfx)$ is a real value in $[\alpha-q,\alpha+q]$, and here $\partial f_q/\partial q$ and $\partial^2 f_q/\partial q^2$ denote the first and second-order derivatives of the function $q \mapsto f_q$. 
    }\\[2mm]
   
{
In the following, if we assume that the consistency and smoothness assumptions: ${\bf C,S}$ hold, then we let $\epsilon,\delta > 0$ be such that both statements stay true.}
The proof of the next Lemma is deferred to Appendix~\ref{sec:consistency}.

\begin{lemma}\label{lem:con}
 { 
Let $(\bfX_t)$ be a  stationary time series satisfying $\bf RV_\alpha$.
Assume the conditions of Proposition~\ref{prop:existence:cluster:process} hold and the series admits a $p-$cluster process $\bfQ^{(p)} \in \tilde{\ell}^p$.
Consider the mixing-coefficients $(\beta_t)$ and assume there exists a sequence $(\ell_n)$ 
satisfying $\ell_n \to  \infty$, as $n \to  \infty$,  and
    \beao
        \lim_{n \to  \infty } m_n \beta_{\ell_n}/k_n & = & 
        \lim_{n \to  \infty} \ell_n/b_n \quad = \quad   0\,.
    \eeao 
Let $f = f_\a(p) :\tilde \ell^{p} \to \mathbb{R}$
and assume ${
\bf S}$ holds. 
Assume $f_\a(p)$ and $
\partial f_q/\partial q|_{q = 
\alpha}(p)$ satisfy {\bf C}, 
for fixed $\epsilon,\delta > 0$, and assume $\widehat \alpha_n \xrightarrow[]{\P} \alpha$, 
as $n \to \infty$.  Then, 
\beao 
\widehat{f^\bfQ_{\widehat \a}}(p)  &\xrightarrow[]{\P}& f_{
\alpha
}^{\bfQ}(p), \qquad n \to  \infty.
\eeao 
Furthermore, if $f_
\alpha(\alpha), 1(\a)$, and $
\partial f_q/\partial q|_{q = 
\alpha}(\alpha)$ all satisfy ${
\bf C}$ for $p = 
\alpha$, then the relation $\widehat{f_{\widehat \a}^\bfQ}(\widehat \a) \xrightarrow[]{\P} f_\a^{\bfQ}(\alpha)$ holds as well as $n \to \infty$.
}
\end{lemma}
}}}

\subsection{Assumptions for asymptotic normality}\label{sec:assumption}

Our main result is presented in Theorem~\ref{thm:main} 
stating the asymptotic normality of the $\ell^p-$block estimators in \eqref{eq:estimator:cluster} under the Lindeberg, bias, moment and mixing assumptions below. 
\\[2mm]

\par 

\noindent 
     {\bf L}: Let $f= f_\a(p) \in \mathcal{G}_+(\tilde{\ell}^p)$ such that $u \mapsto f((\bfx_t)/u)$ is non-increasing, and there exists $\delta > 0$ such that, for all $u > 0$,  the following Lindeberg-type condition holds
    \beam\label{eq:cond:lind:1} 
        \limsup_{n \to \infty}\frac{\E[\big(f(\bfX_{[1,n]}/(u\,x_{n}))\big)^{2+\delta}  \1(\|\bfX_{[1,n]}/x_{n} \|_p > u) ]}{\P(\|\bfX_{[1,n]}\|_p >x_n)} &<& \infty.
    \eeam
    {Assume  $\E[\big(f(Y \bfQ^{(p)})\big)^{2+\delta}] < \infty$.
    \par 
    If $p = \alpha$ assume there exists $\epsilon > 0$ such that
    { \eqref{eq:cond:lind:1} holds replacing $p$ by $\alpha-\epsilon$ and}
    $\sup_{ \tiny q \in [\a - \epsilon, \a + \epsilon]} \E[\big(f(Y \bfQ^{(q)})\big)^{2+\delta}] < \infty$.} \\[2mm]
    {${\bf B}(k_n):$} Consider $f$ satisfying $\bf L$, and assume {there exists $\epsilon > 0$} such that the bias condition:
\begin{align}\label{eq:bias2}
   \lim_{n \to \infty} \sqrt{k_n} 
   \sup_{
   \substack{ u \in [1-\epsilon,1+\epsilon]} } \big|\, \frac{ \E[f(        \mathcal{B}_1/(u\, x_{b_n}) )\1(\|\mathcal{B}_1/x_{b_n} \|_p > u) ]}{
   \P(\|\mathcal{B}_1\|_p > x_{b_n})
   } 
   &-  u^{-\a}\, f^{\bfQ}(p)\, \big|
    \nonumber \\
    & \quad =\quad  0,
\end{align}
    holds where $f^\bfQ(p)$  is as in \eqref{eq:statistic}. \\[2mm]
    {
{{\bf B}$_\alpha(k_n):$} Consider $f$ satisfying $\bf L$, and assume {there exists $\epsilon > 0$} such that
\begin{align}
\label{eq:bias1}
   \lim_{n \to \infty}  \sqrt{k_n} 
   \sup_{
   \substack{ u \in [1-\epsilon,1+\epsilon], \\
   q \in [\alpha - \epsilon,  \alpha + \epsilon]} } \big|\, \frac{ \E[f(        \mathcal{B}_1/(u\, x_{b_n}) )\1(\|\mathcal{B}_1/x_{b_n}\|_q > u) ]}{
   \P(\|\mathcal{B}_1\|_q > x_{b_n})
   } &- u^{-\a}\, f^{\bfQ}(q)\, \big|\, \nonumber \\
    & \quad = \quad 0.
\end{align}
     {\bf M}: Assume {there exists $\epsilon > 0$} such that the moment condition below holds
    \beao 
\E\big[ \|\bfQ\|_{\a-\vep}^{2\a}\big]<\infty\,.
\eeao
    ${\bf MX_\beta}$: Consider $f$ satisfying $\bf L$, and let $\delta$ be such that \eqref{eq:cond:lind:1} hold. Assume that the mixing coefficients $(\beta_t)$ satisfy  for some sequence $(\ell_n)$ satisfying $\ell_n \to  \infty$, and $\ell_n/b_n \to 0$, $ m_n \beta_{\ell_n}/k_n \to 0$, as $n \to \infty$, 
    and
    \beam\label{eq:covariance:beta:mixing}
    \lim_{n \to  \infty}  \sum_{t=2}^{m_n} (m_n\beta_{ tb_n}/k_n)^{\tfrac{\delta}{2+\delta}} \; = \; 0.
    \eeam 
     {If $f$ is bounded, note it is enough to assume $\sum_{t=1}^{m_n} m_n\beta_{ tb_n}/k_n  \to 0$, $n \to \infty$. }\\[2mm]
\par 
   In the remainder of the article, if we assume that certain of the consistency and smoothness assumptions: ${\bf C, S}$ hold or if the Lindeberg, bias, moment, and mixing assumptions: ${\bf L, B, M, MX}_\beta$ hold, then we let $\epsilon, \delta > 0$ be such that all previous statements stay true.
 }

 \subsection{Asymptotic normality of block estimators}\label{sec:main:res}
   
We state next
the asymptotic normality of the $\ell^p-$block estimators in \eqref{eq:estimator:cluster:1}.
We extend the result to $\ell^{\widehat \a}-$block estimators where $1/\widehat \alpha$ is the Hill estimator.
We also cover the implementation of block estimators for functionals $f_\a(p)$ where we plug-in $\widehat \alpha$ at the place of $\alpha$.
  We defer the proof of Theorem~\ref{thm:main} to Section~\ref{sec:asymptotic:normality}.


{
\begin{theorem}\label{thm:main}

Let $(\bfX_t)$ be a  stationary time series satisfying $\bf RV_\alpha$. 
Assume the conditions of Proposition~\ref{prop:existence:cluster:process} hold and the series admits a $p-$cluster process $\bfQ^{(p)} \in \tilde{\ell}^p$. 
Consider $f=f(p) :\tilde{\ell}^p \to \mathbb{R}$ satisfying {\bf L},
and if $f(p) =f_\alpha(p)$ assume  {\bf M} and {\bf S} hold, and ${f^\prime_\alpha = \partial f_q/\partial q|_{q=\alpha} }$, and ${f^{\prime \prime}_\alpha = \partial^2 f_q/\partial q^2|_{q=\alpha}}$ satisfy ${\bf C}$, for fixed values $\epsilon, \delta > 0$.
Assume
$f(p)$ and $1(p)$ satisfy $\bfB(k_n)$, and $h, e$, in 
\eqref{eq:func:hill}, \eqref{eq:func:exceedances} satisfy  {\bf L} and $\bfB(k_n^\prime)$.
Moreover, assume ${\bf MX_\beta}$ holds and 
 $m_n \beta_{b_n} \to 0$, as $ n \to \infty$.\\
Then, if $k_n/k_n^\prime \to \kappa $, with $\kappa \ge 0$, we have
\beam\label{eq:central:limit:1} \label{eq:central:limit} 
&&\sqrt{k_n}\,
\Big( \widehat{f_{\widehat \a }^{\bfQ}}(p)  - f^\bfQ_{\a}(p) \Big)
\xrightarrow[]{d} 
\mathcal{N}\big(0, \sigma^2(\kappa)\big), \qquad n \to \infty\,,\eeam 
where the expression of $\sigma^2(\kappa)\ge 0$ is given in \eqref{eq:sigma:2} and  $\sigma^2(0)= \var(f_\alpha(Y\bfQ^{(p)}))$.
Furthermore,
if $p= \a$, assume ${\bf M}$ holds and {\bf L} holds, and assume $f(\alpha)$ and $1(\alpha)$ satisfy $\bfB_\alpha(k_n)$.  Then, \eqref{eq:central:limit} continues to hold replacing $p$ by $\widehat \a$ in the block estimator. 
\end{theorem}

}


In classical examples (cf. Section~\ref{sec:examples}), for all $\kappa' > 0$, a sequence 
$x_{b_n} \sim  b_n^{\kappa' +1/(p \land \alpha)} $, typically satisfies conditions $\bf AC$ and {\bf CS}$_p$, as $n \to \infty$.
This is a consequence of Remark~\ref{remark:cs:m0}.  
In this case, for any $\kappa'' >0$, the sequence
\beam\label{eq:k:mo}
k_n \blue &\sim&    n\, b_n^{-\kappa'' -{\alpha}/{(p\land \alpha) }},
\eeam
 satisfies the relation in Equation \eqref{eq:k:intro}. Here we use an application of Potter's bound together with Equation \eqref{eq:constant:cp:1:tcl}. 
We study this choice of sequence $(k_n)$ in more detail in Section~\ref{sec:examples}. 
\par 

The choice of $k_n, k_n^\prime$ are also subject to the bias conditions ${\bf B}(k_n)$ and $\bfB(k^\prime_n)$. 
Actually, it is common practice to choose $k_n^\prime$ larger than $k_n$, and the numerical results from Section~\ref{sec:numerical:experiments} support this practice. 
When we use fewer blocks $k_n$ for $p-$cluster inference, compared to the number of records $k_n^\prime$ we use to tune the Hill estimator, precisely if $k_n/k_n^\prime \to 0$, as $n \to \infty$, the variance term simplifies to 
\beao
\sqrt{k_n}\Big( \widehat{f_{\widehat \alpha}^{\bfQ}}(p)  - f_\a^\bfQ(p) \Big) &\xrightarrow[]{d}&  \mathcal{N}\big(0, \Var\big(f_\alpha(Y\bfQ^{(p)}) \big) \, \big)\,, \quad n \to \infty.
\eeao
This expression also holds when the functional $f$ doesn't include $\a$ in its expression. 
\par

\begin{remark}\label{remark:bias}
To plug-in $\widehat \a$ in the place of $\a$ for $p-$cluster inference we require the bias condition ${\bf B}(k_n^\prime)$ for $h, e$.  To do so, the Hill estimator is seen as a block estimator as in  \eqref{eq:estimator:cluster} that evaluates the block functional $h$ on $(\mathcal{B}_t/x^\prime_{b_n})$, $t=1,\dots, m_n$, replacing the high threshold level $(x^\prime_{b_n})$ by $k^\prime-$order statistic from the sample $(|\bfX_t|)$ where
\beao
k^\prime_n &\sim& n\,\P(|\bfX_1| > x_{b_n}^\prime ), \quad n \to \infty. 
\eeao
Then the bias condition  ${\bf B}(k_n^\prime)$ can be rewritten as
\beao 
\lim_{n \to \infty} \sqrt{k^\prime_n} \sup_{u \in [1-\epsilon, 1+ \epsilon]} \Big| \frac{ \E[\log(|\bfX_1|/(u\,x^\prime_{b_n}))\1(|\bfX_1| > (u\, x^\prime_{b_n}))]}{\P(|\bfX_1| > x^\prime_{b_n})} - u^{-\a} \frac{1}{\a}  \Big| &=& 0,
\eeao 
which is no longer a condition on blocks, but on $|\bfX_1|$.
This type of condition was considered in   \cite{drees:janssen:neblung:2021,kulik:cissokho:2022}. Notice that the dependence of the threshold $x^\prime_{b_n}$ with $b_n$ is an artifact of our notation for ${\bf B}$.

\end{remark}


{
\begin{remark}
In Section~\ref{sec:examples} we promote choosing $(k_n)$ as in \eqref{eq:k:mo}, which implies, for $\kappa{''}>0$,
\beao 
m_n/k_n  &\sim&   b_n^{- 1  + \frac{\alpha}{p \land \alpha}+ \kappa'' }, \quad n \to \infty.
\eeao 
In this way, the mixing condition ${\bf MX}_\beta$ depends solely on the mixing coefficients of the series $(\beta_t)$ and its tail index $\alpha$, and not on the choice of block lengths $(b_n)$.  
We can thus interpret condition $m_n \beta_{b_n} \to 0$, as $n\to \infty$, in Theorem~\ref{thm:main} as a restriction
 on the sequence of block lengths $(b_n)$. 
\end{remark}
}

\section{Cluster statistics}\label{sec:examples:cluster}

In view of Theorem~\ref{thm:main}, we derive asymptotic normality of the classical cluster index estimators in extreme value theory. 

\subsection{The extremal index} \label{example:extremal:index}
Let $(\bfX_t)$ be a stationary time series in $(\mathbb{R}^d,|\cdot|)$ satisfying $\bf RV_\alpha$. The extremal index $\theta_{|\bfX|}$ of the series $(|\bfX_t|)$ is a measure of serial clustering introduced in \cite{leadbetter:1983} and \cite{leadbetter:lindgren:rootzen:1983}. We recall the extremal index estimator proposed in \cite{buritica:meyer:mikosch:wintenberger:2021}, based on extremal $\ell^\alpha-$blocks.
\par 
\begin{corollary}\label{cor:ei}
Consider $f_\a:\tilde{\ell}^\a  \to \mathbb{R}$ to be the function $\bfx \mapsto \|\bfx \|_\infty^\a/\|\bfx \|_\a^\a $. Assume the conditions of 
Theorem~\ref{thm:main} hold for $p=\alpha$, and $k/k^\prime \to 0$, as $n \to \infty$.   Let $\theta_{|\bfX|} = \E[\|\bf\bfQ\|_\infty^\alpha]$, hence we deduce an estimator 
\beam\label{eq:estimator:Ei}
\widehat \theta_{|\bfX|} &:=& \dfrac{1}{k_n}   \sum_{t=1}^{m_n} \dfrac{\|\mathcal{B}_t\|^{\widehat \alpha}_\infty}{\|\mathcal{B}_t\|^{\widehat \alpha}_{\widehat \a}}\1(\|\mathcal{B}_t\|_{\widehat \a} > \|\mathcal{B}\|_{\widehat \a,(k_n+1)}),
\eeam 
such that
\beao 
\sqrt{k_n}\big(\, \widehat \theta_{|\bfX|} - \theta_{|\bfX|} \, \big) &\xrightarrow[]{d}&  \, \mathcal{N}(0, \Var(\|\bfQ\|_\infty^{\a})), \quad n \to \infty\,.
\eeao 
\end{corollary}

\begin{proof} 
The proof of Corollary~\ref{cor:ei} follows directly from an application of Theorem \ref{thm:main} to
 the function $\tilde f_\a(\bfx) = f_\alpha(\bfx)\1(\|\bfx\|_{\a} > 1)$ satisfying
$\tilde f_\a  \in \mathcal{G}_+(\tilde{\ell}^{\a})$ and $\tilde f_\alpha$ is a bounded a.s. continuous function satisfying $\bf L$. 
\end{proof}
\par 
For comparison, we also review the block estimator based on extremal $\ell^\infty-$blocks proposed in \cite{hsing:1991}: 
\beam \label{eq:blocks:ei}
\widehat{\theta}_{|\bfX|}^B \;:=\; \dfrac{1}{k_n}   \sum_{t=1}^{m_n} \1( \|\mathcal{B}_t\|_\infty > |\bfX|_{(k_n+1)})\,. 
\eeam 
Direct computations from Example 10.4.2 in \cite{kulik:soulier:2020} yield
\beao  \sqrt{k_n}( \widehat{\theta}_{|\bfX|}^B - \theta_{|\bfX|}) &\xrightarrow[]{d}& \mathcal{N}(0,\sigma^2_\theta ), \quad n \to \infty,
\eeao 
where $\sigma^2_\theta \in [0,\infty)$, and
\beam\label{eq:sigma} 
\sigma^2_\theta &:=& \theta^2_{|\bfX|} \sum_{j\in\mathbb{Z}}\E[|\bfTh_j|^\a \land 1] - \theta_{|\bfX|} \nonumber \\
&=& \theta^2_{|\bfX|} \sum_{j\in  \mathbb{Z}}\sum_{t \in \mathbb{Z}}   \E[|\bfQ_{j+t}|^\alpha \land |\bfQ_{t}|^\alpha ] - \theta_{|\bfX|}.
\eeam 
The last equality follows appealing to the time-change formula in \eqref{eq:time:change:formula:1} and Equation \eqref{eq:alphacluster}. As a result, we can compare the asymptotic variances of $\widehat \theta_{|\bfX|}$ and $\widehat \theta^B_{|\bfX|}$ in the cases where $\bfQ$ is known. This is the topic of Section~\ref{sec:examples}.

\begin{remark}
An alternative $\a-$cluster estimator of the extremal index corresponds
 to the block functional $\tilde f(\bfx) = \1(\|\bfx\|_\infty > 1)$.
A similar asymptotic normality 
result applies but with an asymptotic variance $\Var ( \tilde f ( Y \bfQ) )$ larger 
than $\Var(f_\a( Y \bfQ) )$. It motivates the use of $\hat \theta_{|\bfX|}$ 
although it requires the estimation of $\a$. 
The latter is harmless using Hill's estimator choosing $k'$ sufficiently large with respect to $k$, 
which is often the case in practice.
\end{remark}
\begin{remark}\label{remark:check:S:ei}
To check condition $\bf S$ for the functional $f_\alpha(\bfx) = \|\bfx\|_\infty^\alpha/\|\bfx\|_\a^\a$, note the Taylor expansions:
\beao 
\lefteqn{ \|\bfx\|_q^q - \|\bfx\|_\alpha^\alpha }\\
&= &  (q - \a) \sum_{t \in \mathbb{Z}}|\bfx_t|^\alpha \log(|\bfx_t|) + \frac{1}{2} (q-\a)^2 \sum_{t \in \mathbb{Z}} |\bfx_t|^{q^\prime} \log^2(|\bfx_t|),\\
\lefteqn{\|\bfx\|_\infty^q - \|\bfx\|_\infty^\alpha }\\
&= & (q- \a) \, \|\bfx\|_\infty^\a \log(\|\bfx\|_\infty) + \frac{1}{2} (q- \a)^2 \, \|\bfx\|_\infty^{q^{\prime\prime} } \log^2(\|\bfx\|_\infty),
\eeao 
hold for some $q^\prime,q^{\prime\prime} \in [\alpha \land q , \alpha \lor q ]$. Hence, $f_q$ satisfies 
\beao 
\lefteqn{ f_q(\bfx) - f_\a(\bfx)}\\
 &=& (q - \a) \underbrace{ \frac{\|\bfx\|_\infty^\alpha}{\|\bfx\|_\a^\a} \sum_{t \in \mathbb{Z} }\frac{|\bfx_t|^\a}{\|\bfx\|_\a^\a} \log(\|\bfx\|_\infty/|\bfx_t|)}_{= \partial f_q/\partial q |_{q=\a }} \\
 && + \frac{1}{2} (q-\alpha)^2  
 \underbrace{ \frac{\|\bfx\|_\infty^{q^\prime} }{\|\bfx\|_{q^\prime}^{q^\prime}} \sum_{t \in \mathbb{Z} }
\sum_{j \in \mathbb{Z}} \frac{|\bfx_t|^{q^\prime}}{\|\bfx\|_{q^\prime}^{q^\prime}}\frac{|\bfx_j|^{q^\prime}}{\|\bfx\|_{q^\prime}^{q^\prime}} \log(\|\bfx\|_\infty/|\bfx_j|)\log(|\bfx_j|/|\bfx_t|).}_{ = \partial^2 f_q/\partial q^2|_{q=q^\prime} }
\eeao 
As mentioned, this expansion is helpful to verify condition {\bf S} on the models from Section~\ref{sec:examples}. 
	
\end{remark}

\subsection{The cluster index for sums}\label{example:cluster:index}
Let $(\bfX_t)$ be a stationary time series with values in $(\mathbb{R}^d,|\cdot|)$ satisfying $\bf RV_\alpha$. We recall {that when $\alpha < 2$} \cite{mikosch:wintenberger:2014} coined the constant $c(1)$ in \eqref{eq:constant:cp:1:tcl} as the cluster index for sums. We review a cluster-based estimator of it, introduced in \cite{buritica:mikosch:wintenberger:2021}, based on extremal $\ell^\a-$blocks.
\begin{corollary}\label{cor:c1}
Consider $f_\alpha:\tilde{\ell}^\a  \to \mathbb{R}$ to be the function $\bfx \mapsto \|\bfx \|_1^\a/\|\bfx \|_\alpha^\a $.
Assume
the conditions of Theorem~\ref{thm:main} hold for $p=\alpha \land 1$,  and $k/k^\prime \to 0$, as $n \to \infty$.  Let $c(1) = \E[\|\bfQ\|_1^\alpha]  < \infty$, hence one deduces an estimator 
\beam\label{eq:estimator:c1} 
\widehat c(1) &:=& \dfrac{1}{k_n}   \sum_{t=1}^{m_n} \dfrac{\|\mathcal{B}_t\|_1^{\widehat \alpha} }{\|\mathcal{B}_t\|_{{\widehat \alpha}}^{\widehat \alpha}}\1(\|\mathcal{B}_t\|_{{\widehat \alpha}} > \|\mathcal{B}\|_{{\widehat \alpha},(k_n+1)}) ,
\eeam 
such that
\beao 
\sqrt{k_n}\big(\, \widehat  c(1) - c(1) \, \big) &\xrightarrow[]{d}&  \,  \mathcal{N}(0, \Var(\|\bfQ\|^\a_1))\,,\quad n \to \infty,
\eeao 
and $c(1)$ is as in \eqref{eq:constant:cp:1:tcl} with $p=1$.
\end{corollary}
\begin{proof}
The proof of Corollary \ref{cor:c1} follows directly from Theorem~\ref{thm:main} 
 the function $\tilde f_\a(\bfx) = f_\alpha(\bfx)\1(\|\bfx\|_{p} > 1)$ satisfies
$\tilde f_\a  \in \mathcal{G}_+(\tilde{\ell}^{\a})$ where $\tilde f_\alpha$ is a bounded a.s. continuous function satisfying $\bf L$. 
\end{proof}
\par 
Another sums index cluster-based estimator we can consider is the one proposed in \cite{kulik:soulier:2020} based on extremal $\ell^\infty-$blocks:
\beam\label{examples:cluster:sum:2}
\widehat{c}^B(1) \;=\; \dfrac{1}{k_nb_n}   \sum_{t=1}^{m_n} \1( \|\mathcal{B}_t\|_1 > |\bfX|_{(k+1)}).
\eeam 
Then, relying on Example 10.4.2 in \cite{kulik:soulier:2020},
\beao 
\sqrt{k_n}( \widehat{c}^B - c(1)) &\xrightarrow[]{d}& \mathcal{N}(0,\sigma^2_{c(1)}), \quad n \to \infty,
\eeao 
for a constant $\sigma^2_{c(1)} \in [0,\infty)$ defined by
\beam\label{eq:sigma:2} 
\sigma^2_{c(1)} &=&  c(1)^2 \sum_{j\in  \mathbb{Z}}\sum_{t \in \mathbb{Z}}   \E[|\bfQ_{j+t}|^\alpha \land |\bfQ_{t}|^\alpha ] - c(1).
\eeam 
Similarly as in Example~\ref{example:extremal:index}, whenever $\bfQ$ is known, we can directly compare the asymptotic variances relative to the estimators $\widehat{c}(1)$ and $\widehat c^B(1)$. Section~\ref{sec:examples} covers this topic for classical models where the cluster process is known.
\par 
{ 
Moreover, we can use the computations in Remark~\ref{remark:check:S:ei} to verify condition {\bf S} holds. In this case it suffices to replace $\|\bfx\|_\infty$ by $\|\bfx\|_1$ in all its appearances. 
}
\subsection{The cluster sizes}\label{example:cluster:sizes}
In general, a classical approach to model serial exceedances is using point processes as in \cite{leadbetter:lindgren:rootzen:1983} and \cite{hsing:1993}. For the levels $(a_n)$, satisfying $n\P(|\bfX_1| > a_n) \to 1$, as $n \to \infty$, and for every fixed $x>0$ consider 
the \pp\ of exceedances with state space $(0,1]$: 
\beao
\eta_{n,x}(\cdot) &:=& \overline N_n\big(\{\bfy\in\tilde{\ell}^\infty:|\bfy|>x\} \times \cdot\big) \;=\;
\sum _{i=1}^n\vep_{i/n}(\cdot)\,\1(|\bfX_i|>x\,a_n)\,.
\eeao
Under mixing and anti-clustering conditions, for fixed $x>0$, we can express the limiting point process in \cite{hsing:1993} 
such as
\beao
\eta_{n,x}(\cdot)\;\std\; \eta_x(\cdot)&:=& \overline N\big(\{\bfy\in \tilde{\ell}^\infty:|\bfy|>x\} \times 
\cdot \big)\\
&=& \sum_{i=1}^\infty\sum_{j \in \mathbb{Z}}\1\big(\Gamma_i^{-1/\a}| \bfQ_{ji}|>x\big)\, \vep_{U_i}(\cdot)\,,
\eeao
where the points $(U_i)$ are iid uniformly distributed on $(0,1)$, $(\Gamma_i)$ are the points of a standard homogeneous Poisson process, and $(\bfQ_{\cdot i})$ are iid copies of the cluster process $\bfQ$. Using the independence among these three processes, one can easily rewrite the limit as 
\beam\label{eq:limit:pp}
\eta_x((0,t]) &:=& \sum_{i=1}^{N_x(t)} \xi_i\,,\qquad 0<t\le 1\,,
\eeam 
where
\begin{itemize}
\item
$N_x$ is a homogeneous Poisson process on $(0,1]$ with intensity $x^{-\a}$,
\item
for an iid \seq\ $(Y_i)$ of Pareto$(\a)-$distributed \rv s 
which is also independent of 
$( \bfQ_i)$,
\beao
\xi_i&:=&\sum_{j\in \mathbb{Z}} \1(Y_i\, | \bfQ_{ji}|>1)\,,
\eeao
\item $N_x$, $(\xi_i)$ are independent.
\end{itemize}
{
  Relying on the point process of exceedances representation in \eqref{eq:limit:pp}, the random variables $(\xi_i)$ can be interpreted as counts of serial exceedances from one cluster. Furthermore, we deduce the relation $\P( \xi_1 > 0 ) = \E[\|\bf\bfQ\|_{\infty}^\alpha] =\theta_{|\bfX|} $, and also get an expression for the cluster size probabilities
\beam\label{eq:xmas25c}
\P(\xi_1 = j) &=& {\E[|\bfQ|_{(j)}^\a-|\bfQ|_{(j+1)}^\a] } = \quad \pi_j \,,\qquad j \ge 1\,.
\eeam
The statistic $\pi_j$ can be understood as the probability of recording a cluster of length $j$. The block estimators provide natural estimators of these quantities
\beam \label{eq:cluster:sizes}
\widehat{\pi}_j &:=& \dfrac{1}{k_n}   \sum_{t=1}^{m_n} \frac{ |\mathcal{B}_t|_{(j)}^{\widehat \alpha} - |\mathcal{B}_t|_{(j+1)}^{\widehat \alpha} }{\|\mathcal{B}_t\|_{\widehat \alpha}^{\widehat \alpha}} \1(\|\mathcal{B}_t\|_{\widehat \alpha} > \|\mathcal{B}\|_{{\widehat \alpha},(k_n+1)}),
\eeam 
$|\mathcal{B}_t|_{(1)} \ge |\mathcal{B}_t|_{(2)} \ge \dots \ge  |\mathcal{B}_t|_{(m)}$ are the order statistics of $\mathcal{B}_t$, the $t-$th block.
\begin{corollary}\label{cor:cluster:lengths}
Consider the function $\pi_j:\tilde{\ell}^\alpha \to \mathbb{R}$ defined by $\pi_j(\bfx) := (|\bfx|^\alpha_{(j)} - |\bfx|^\alpha_{(j+1)})/\|\bfx\|_\alpha^\alpha$, where $|\bfx|_{(1)} \ge |\bfx|_{(2)} \ge  \dots $.
Assume
the conditions of Theorem~\ref{thm:main} hold for $p=\alpha$ and $k/k^\prime \to 0$, as $n\to \infty$. Then, for all $j \ge 1$ we have
\beam\label{eq:clt:cluster} 
\sqrt{k_n}\big(\, \widehat{\pi}_j - \pi_j \, \big) &\xrightarrow[]{d}&  \,  \mathcal{N}(0, \Var(\pi_j^{\bfQ}(\bfQ) ))\,,\quad n \to \infty\,.
\eeam 
\end{corollary}
}
\begin{proof}
The proof of Corollary \ref{cor:c1} follows directly from Theorem~\ref{thm:main} as 
the function $\tilde \pi_j(\bfx) = \pi_j(\bfx)\1(\|\bfx\|_{\a} > 1)$ satisfies
$\tilde \pi_j  \in \mathcal{G}_+(\tilde{\ell}^{\a})$ and $\tilde \pi_j$ is a bounded a.s. continuous function satisfying $\bf L$.
\end{proof}
\par 
{
Corollary~\ref{cor:cluster:lengths} provides a novel procedure for estimating cluster size probabilities based on extremal $\ell^\alpha-$blocks. As in the previous examples, the asymptotic variance can be computed as long as $\bfQ$ is known. This allows for comparison with the other cluster-based inference procedures provided in \cite{hsing:1991,ferro:2003,robert:2009}. One advantage of our methodology is that we can straightforwardly infer the asymptotic variances of cluster sizes since we express them as cluster statistics in \eqref{eq:clt:cluster}. Moreover, inference through extremal $\ell^\a-$blocks has already proven to be useful in \cite{buritica:mikosch:wintenberger:2021} for fine-tuning the hyperparameters of the estimators, see also the discussion in Section~\ref{sec:numerical:experiments}. }
\par 
{ 
As before, we can use the computations in Remark~\ref{remark:check:S:ei} to verify condition {\bf S}, it suffices to replace $\|\bfx\|_\infty$ by $(|\bfx|_{(j)} - |\bfx|_{(j+1)})$ in the equations therein. 
}


\section{Models}\label{sec:examples}
\subsection{Linear $m_0$--dependent sequences.}\label{sec:example:m0}
We consider $(\bfX_t)$ to be a $m_0$--dependent time series with values in $(\mathbb{R}^d,|\cdot|)$ satisfying $\bf RV_\alpha$.
\begin{example}\label{example:m0}
The time series $(\bfX_t)$ is a linear moving average of order $m_0\ge 1$ if it satisfies
\beam\label{eq:m0:MA} 
 \bfX_t &:=& \bfZ_{t} + \varphi_1 \bfZ_{t-1} + \dots + \varphi_{m_0} \bfZ_{t-m_0}, \quad  t \in \mathbb{Z},
\eeam 
with $\mathbb{R}^d-$variate iid innovations $(\bfZ_t)$ satisfying $\bf RV_\alpha$, and $(\varphi_j) \in \mathbb{R}^{m_0}$. 

Alternatively, the max-moving average  of order $m_0\ge 1$ satisfies
\beam\label{eq:m0:MaxA} 
 X_t &:=& \max\{Z_{t},\varphi_1 Z_{t-1}, \ldots, \varphi_{m_0} Z_{t-m_0}\}, \quad  t \in \mathbb{Z},
\eeam 
with $\mathbb{R}_+-$variate iid innovations $(Z_t)$ satisfying $\bf RV_\alpha$, and $(\varphi_j) \in \mathbb{R}^{m_0}_+$. 
Then both moving averages satisfy $\bf RV_\alpha$ with $|\bfQ|$ admitting the same deterministic expression  $(|\varphi_t|/\|(\varphi_j)\|_\alpha)$ in $\tilde{\ell}^\alpha$, see for instance Proposition 3.1. in \cite{buritica:mikosch:wintenberger:2021} and Chapter 5 of \cite{kulik:soulier:2020}.
\end{example}

\par 
The proof of the Proposition below is postponed to  Section~\ref{section:models}.

\begin{proposition}\label{lem:CS:m0}
 Consider $(\bfX_t)$ to be an $m_0$--dependent time series with values in $(\mathbb{R}^d,|\cdot|)$. Consider $p > \alpha/2 $, 
 and sequence $(k_n)$ satisfying \eqref{eq:k:mo}, and
 $m_n/k_n \to \infty$, 
 and $(k_n^\prime)$ such that  
 $k/k^\prime \to 0$ as $\nto$.
  Consider $f_\alpha(p):\tilde{\ell}^p \to \mathbb{R}$, and assume $\bf L$, {\bf S} hold. 
  Then, 
\beao 
\sqrt{k_n}\, \big( \, \widehat{f_{\widehat \a}^\bfQ}(p) - f_\a^\bfQ(p) \, \big) &\xrightarrow[]{d}& \mathcal{N}\big(\, 0, \Var(\, f_\a(Y \bfQ^{(p)})\,) \, \big), \quad n \to  \infty,
\eeao 
under the bias conditions $\bfB_\a(k)$, $\bfB(k^\prime)$, and the result extends to $\widehat \a-$cluster inference.
In particular
the $\widehat \a-$cluster based estimators from Section~\ref{sec:examples:cluster} in \eqref{eq:estimator:Ei} \eqref{eq:estimator:c1}, and \eqref{eq:cluster:sizes} are asymptotically normally distributed, and in the case of the moving averages of Example \ref{example:m0}
\beao 
\sqrt{k_n}\, \big(\, \widehat{f_{\widehat \a}^\bfQ}(\widehat \a) - f_\a^\bfQ(\a) \,\big ) &\xrightarrow[]{\P}& 0, \quad n \to  \infty.
\eeao 
\end{proposition}


\subsection{Linear processes}\hypertarget{lm}
In this section we consider stationary linear processes $(\bfX_t)$ with values in $(\mathbb{R}^d, |\cdot|)$ satisfying $\bf RV_\alpha$.

\begin{example}\label{ex:linear:model}
Consider $(\bfX_t)$ to be an $\mathbb{R}^d-$variate sequence 
satisfying 
\beam\label{eq:linear:model} 
\bfX_t &=& \sum_{t\in \mathbb{Z}} \varphi_j \bfZ_{t-j}, \quad t \in \mathbb{Z},
\eeam 
for a sequence of iid innovations $(\bfZ_t)$ satisfying $\bf RV_\alpha$, and a sequence $(\varphi_j)$ in $\mathbb{R}^\mathbb{Z}$. Moreover, assume there exists $\kappa' > 0$ such that $\|(\varphi_j)\|_{(\alpha-\kappa')\land 2} <  \infty$. 
\end{example}
In the setting of Example~\ref{ex:linear:model}, a stationary solution $(\bfX_t)$ exists and satisfies $\bf RV_\alpha$ (c.f. \cite{cline:1983,mikosch:samorodnitsky:2000,hult:samorodnitsky:2008}).  Proposition~\ref{lem:linear:process} below demonstrates conditions {\bf AC, CS}$_p
$ hold for $p > \alpha/2$, and a suitable sequence $(x_n)$ such that $n\P(|\bfX_1| > x_n) \to 0$ as $n \to \infty$. Therefore, the time series $(\bfX_t)$ admits an $\alpha-$cluster process $\bfQ$, which we can compute in terms of the filter $(\varphi_j)$, and the spectral measure of the random variable $\bfZ_0$, denoted by $\bfTh_0^\bfZ$, with $|\bfTh_0^\bfZ| = 1$ a.s.

We obtain the expression, cf. Chapter 5 of \cite{kulik:soulier:2020}, 
\beam\label{eq:Q:linear} 
\bfQ &\eqd&  (\varphi_t/\|(\varphi_j)\|_\alpha) \,  \bfTh_0^{\bfZ}, \quad \in \tilde{\ell}^\alpha.
\eeam 
Note again that the norm of the $\alpha-$cluster process, i.e., $|\bfQ|$, is deterministic in $\tilde{\ell}^\alpha$. Assuming $\|(\varphi_j)\|_p < \infty$, we can compute the indices $c(p)$ in \eqref{eq:constant:cp:1:tcl} by
\beam\label{eq:c:linear}
c(p)  &=& \E[\|\bfQ\|_p^\alpha] \; = \; \|(\varphi_j)\|^\alpha_p/\|(\varphi_j)\|^\alpha_\alpha \; < \;  \infty.
\eeam
 Classic examples of these heavy-tailed linear models are auto-regressive moving averages, i.e., ARMA processes, with iid regularly varying noise; cf. \cite{brockwell:davis:2016}.
\par 
The proposition below guarantees that the assumptions of Proposition~\ref{prop:existence:cluster:process} hold. We defer its proof to Section~\ref{sec:models:ca:cs}.

\begin{proposition}\label{lem:linear:process}
 Consider $(\bfX_t)$ to be a linear process with values in $(\mathbb{R}^d,|\cdot|)$, as in Example~\ref{ex:linear:model}. Let $p > \alpha/2$, $\kappa' > 0$,  
 and consider a sequence $(x_n)$ such that $n/x_n^{p \land (\alpha-\kappa')}\to 0$, $n \to \infty$. Then it holds for all  $\delta > 0$
\beam \label{eq:truncation:k:linear:1}
\lim_{s \to  \infty} \limsup_{n \to  \infty} \frac{\P( \mbox{$ \| \bfX_{[1,n]}/x_n - \bfX_{[1,n]}^{(s)}/x_n\|_p^p > \delta $} )}{ n \P(|\bfX_1| > x_n) } &=& 0,
\eeam
where $\bfX_t^{(s)} := \sum_{|j| \le s} \varphi_j \bfZ_{t-j}$. Thus ${\bf AC}$ and ${\bf CS}_p$ are satisfied.
\end{proposition}
We now review the mixing properties of a linear process. We recall below the statement in Theorem 2.1. in \cite{pham:tran:1985} (see Lemma 15.3.1. in \cite{kulik:soulier:2020}).
\begin{proposition}\label{prop:beta:mixing:linear}
 Consider $(\bfX_t)$ to be a causal linear process with values in $(\mathbb{R}^d,|\cdot|)$, as in Example~\ref{ex:linear:model} with $\varphi_j = 0$, for $j < 0$. Assume the distribution of $\bfZ_0$ is absolutely continuous with respect to the Lebesgue measure in $\mathbb{R}^d$, and has a density $g_Z$ satisfying
\begin{itemize}
    \item[i)] $\int|g(\bfx - \bfy) - g(\bfx)|d \bfx \,=\, O(|\bfy|) $, for all ${\bf y} \in \mathbb{R}^d$,
 \item[ii)] $\varphi_t = O(t^{-\rho})$, for $t \ge 0$, and $\rho > 2 + 1/\alpha$,
\item[iii)] 
$\sum_{j=0}^\infty\varphi_j \bfx^j \not = 0$, for all $\bfx \in \mathbb{R}^d$ with $|\bfx| < 1$,
\end{itemize}
Then, for all $0<\epsilon <\alpha$, the mixing coefficients $(\beta_t)$ satisfy
\beam\label{eq:beta:linear}
 \beta_t &=& O\Big( t^{1-\tfrac{(\rho-1)(\alpha - \epsilon)}{1+\alpha-\epsilon}}\Big)\,.
\eeam 
\end{proposition}
Combining Propositions~\ref{lem:linear:process} and 
\ref{prop:beta:mixing:linear}, we state below the asymptotic normality of the $p-$cluster based estimators for linear processes in Theorem~\ref{eq:thm:linear:tcl}. We defer its proof to Section~\ref{sec:proof:clt:linear}.

\begin{theorem}\label{eq:thm:linear:tcl}

 Consider $(\bfX_t)$ to be a causal linear process with values in $(\mathbb{R}^d,|\cdot|)$, as in Example~\ref{ex:linear:model}.
 Let $\rho > 0$, and assume the conditions of Proposition~\ref{prop:beta:mixing:linear} hold with $\varphi_t = O(t^{-\rho})$, for $t > 0$.    
  Consider $p > \alpha/2$,
    and sequence $(k_n)$ satisfying \eqref{eq:k:mo}, and
    $m_n/k_n \to \infty$, 
    and $(k_n^\prime)$ such that  
    $k/k^\prime \to 0$ as $\nto$.
  Consider
$f_\alpha(p):\tilde{\ell}^p \to \mathbb{R}$, and assume $\bf L$, {\bf S} hold. Assume  
that for $\delta > 0$ as in \eqref{eq:cond:lind:1}, \beam\label{cond:i)} \rho &>& 3 + \tfrac{2}{\alpha} + \tfrac{2}{\delta}(1+\tfrac{1}{\alpha})\,.
      \eeam
  If $f_\alpha(p)$ is bounded, condition \eqref{cond:i)} can be replaced by $\rho > 3+ 2/\alpha$.
 Then, 
  {
\beao 
\sqrt{k_n}\, \big(\, \widehat{f_{\widehat \a}^\bfQ}(p) - f_\a^\bfQ(p) \, \big) &\xrightarrow[]{d}& \mathcal{N}\big( \,0, \Var(\, f_\a(Y \bfQ^{(p)})\,)\, \big), \quad n \to  \infty,
\eeao 
under the bias conditions $\bfB_\a(k)$, $\bfB(k^\prime)$, and the result extends to $\widehat \a-$cluster inference.
In particular
the $\widehat \a-$cluster based estimators from Section~\ref{sec:examples:cluster} in \eqref{eq:estimator:Ei} \eqref{eq:estimator:c1}, and \eqref{eq:cluster:sizes},
satisfy
\beao 
\sqrt{k_n}\, \big( \, \widehat{f_{\widehat \a}^\bfQ}(\widehat \a) - f_\a^\bfQ(\a) \, \big) &\xrightarrow[]{\P}& 0, \quad n \to  \infty.
\eeao 
}
\end{theorem}
Regarding cluster inference in the case of linear models, the $\a-$cluster approach has an optimal asymptotic variance for shift-invariant functionals since we use the $\ell^\alpha-$norm order statistics. For this reason, it compares favourably with state-of-the-art block estimator. For example, for the extremal index, the super-efficient estimator in \eqref{eq:estimator:Ei} has a lower asymptotic variance than the block estimator in \eqref{eq:blocks:ei}. Indeed the asymptotic variance $\sigma^2_\theta$ of the latter, computed in \eqref{eq:sigma}, is not necessarily null. For example, for the autoregressive process of order one AR$(1)$ one has $\sigma_\theta^2 = 1-\theta_{|\bfX|} > 0$.
\par 





\subsection{Affine stochastic recurrence equation solution under Kesten's conditions}\hypertarget{sre}{}
In this section
we focus on the causal solution to the affine stochastic recurrence equation SRE under Kesten's conditions. To guarantee the existence of a solution $(\bfX_t)$, with values in $(\mathbb{R}^d, |\cdot|)$ as in \eqref{eq:sre:def} satisfying ${\bf RV}_\alpha$, we rely on Theorem 2.1. and Theorem 2.4 in \cite{basrak:davis:mikosch:2002}. For an overview, we refer to  \cite{buraczewski:damek:mikosch:2016}. In what follows, we study time series $(\bfX_t)$ as in the Example~\ref{example:sre} below. 
\begin{example}\label{example:sre}
Consider $(\bfX_t)$ to be a sequence with values in $\mathbb{R}^d$ satisfying
\beam\label{eq:sre:def}
 \bfX_t &=& \bfA_t \bfX_{t-1} + \bfB_t, \qquad t \in \mathbb{Z},
\eeam 
where $((\bfA_t,{\bfB}_t))$ is an iid sequence of non-negative random $d \times d$  matrices with generic element $\bfA$, and non-negative random vectors with generic element $\bfB$ taking values in $\mathbb{R}^d$. For the existence of a causal stationary solution, we assume 
\begin{itemize}
    \item[$i)$]   $\E[\log^{+}|\bfA|_{op}] + \E[\log^{+}|\bfB|]\;<\;\infty,$
    \item[$ii)$] under $i)$, assume the Lyapunov exponent of $(\bfA_t)$, denoted $\gamma$, satisfies
    \beao \gamma &:=& \lim_{n \to \infty} n^{-1} \log |\bfA_n \cdots \bfA_1|_{op}\; < \;0, \quad a.s. 
    \eeao
\end{itemize}
To guarantee the heavy-tailedness condition $\bf RV_\alpha$, we also assume
\begin{itemize}
    \item[$iii)$] $\bfB \not = \bfzero$ a.s.,  and $\bfA$ has no zero rows a.s.
    \item[$iv)$] there exists $\kappa > 0$ such that $\E[|\bfA|_{op}^\kappa] < 1$,
    \item[$v)$] the set $\Gamma$ from Equation \eqref{eq:Gamma} generates a dense group on $\mathbb{R}$,
    \beam\label{eq:Gamma} 
    \Gamma &=& \{\log|\bfa_n \cdots \bfa_1|_{op} : n \ge 1, \, \bfa_n \cdots \bfa_1 > 0,\nonumber \\
    && \quad \; \bfa_n,\dots, \bfa_1 \text{ are in the support of $\bfA$'s distribution } \},
    \eeam 
    \item[$vi)$] there exists $\kappa_1 > 0$ such that $\E[ (\min_{i=1,\dots,d} \sum_{t=1}^d A_{ij})^{\kappa_1} ] \ge d^{\kappa_1/2}$, and $\E[ |\bfA|_{op}^{k_1}\, \log^{+}|\bfA|_{op}] < \infty$.
    \item[$vii)$] under $i)-vi)$, there exists a unique $\alpha >0$ such that 
    \beam\label{eq:sre:alpha}
     \lim_{n \to  \infty} n^{-1} \log \E\big[ | \bfA_n \cdots \bfA_1 |_{op}^{\a}\big] &=& 0,
     \eeam 
     and $\E[|\bfB|^\alpha] < \infty$. If $d>1$ assume $\alpha$ is not an even integer.
\end{itemize}
The $\mathbb{R}^d-$variate series $(\bfX_t)$, satisfying \eqref{eq:sre:def} and $i)-vii)$, admits a causal stationary solution and satisfies ${\bf RV}_\alpha$, with $\alpha > 0$ as in Equation~\eqref{eq:sre:alpha}.
\end{example} 

The previous example is motivated by the seminal Kesten's paper \cite{kesten:1973}.  We follow Theorem 2.1. in \cite{basrak:davis:mikosch:2002} to state conditions $i)-ii)$ of Example~\ref{example:sre}. Under the conditions $i)-ii)$, the unique solution $(\bfX_t)$ of \eqref{eq:sre:def} has the a.s. causal representation
\beam\label{eq:sre:representation}
  \bfX_t &=& \sum_{i \ge 0} \bfA_{t-i+1} \dots \bfA_t \, {\bfB}_{t-i}, \qquad t \in \mathbb{Z},
\eeam 
where the first summand is ${\bfB}_t$ for $i = 0$ by convention; for an overview see \cite{buraczewski:damek:mikosch:2016}.

\par 
 One of the main reasons why the solutions to SRE as in Example~\ref{example:sre} have received strong interest, is because $(\bfX_t)$ satisfies ${\bf RV}_\alpha$ even when the innovations $((\bfA_t,\bfB_t))$ are light-tailed. This feature was first noticed in \cite{kesten:1973} where the original Kesten's assumptions were introduced. In Kesten's framework, a causal stationary solution to the SRE exists as in \eqref{eq:sre:representation}, and the extremes of the series occur due to the sums of infinitely many terms of growing length products appearing in \eqref{eq:sre:representation}; see \cite{bingham:goldie:teugels:1987} for a review. Further, the community adopted the simplified Kesten's conditions stated by Goldie in \cite{goldie:1991} for univariate SRE. These conditions also aim to capture the heavy-tailed feature under lighter-tailed innovations. In Example~\ref{example:sre}, we borrow the conditions $iii)-vii)$ established for the multivariate setting from Theorem 2.4 and Corollary 2.7. in \cite{basrak:davis:mikosch:2002}; see also \cite{basrak:davis:mikosch:2002:2}. Then, a solution $(\bfX_t)$ as in Example~\ref{example:sre} satisfies $\bf RV_\alpha$, for $\alpha >0$, and the index of regular variation $\alpha$ is the unique solution to the Equation~\eqref{eq:sre:alpha}. We are also  interested in Example~\ref{example:sre} because it models classic econometric time series such as the squared ARCH$(p)$, and the volatility of GARCH$(p,q)$ processes; see \cite{buraczewski:damek:mikosch:2016}.
 \par 
 Concerning the extremes of $(\bfX_t)$ in Example~\ref{example:sre}, the forward spectral tail process satisfies the relation
\beao 
 \bfTh_t &=& \bfA_t \cdots \bfA_1 \bfTh_0, \quad t \ge 0,
\eeao 
where $(\bfA_t)$ is an iid sequence distributed as $\bfA$; see \cite{janssen:segers:2014}. The backward spectral tail process has a cumbersome representation that we omit here; c.f.  \cite{janssen:segers:2014}. We state in Proposition~\ref{prop:sre:cs} sufficient conditions on $(\bfA, \bfB)$ yielding assumptions $\bf AC$, ${\bf CS}_p$ hold for $p > \alpha/2$, and a suitable sequence $(x_n)$ such that $n\P(|\bfX_1|> x_n) \to 0$ as $n \to \infty$. In this case the time series $(\bfX_t)$ admits an $\alpha-$cluster process $\bfQ$.
We recall the identity from Equation (8.6) of \cite{buritica:mikosch:wintenberger:2021}:
$c(p) = \E[\|\bfQ\|_p^\alpha] = \E[\|(\bfTh_t)_{t \ge 0}\|_p^\alpha -\|(\bfTh_t)_{t \ge 1}\|_p^\alpha]$, for $c(p)$ as in \eqref{eq:constant:cp:1:tcl}. Then, letting $p = \alpha/2$, a straightforward computation yields
\beao 
c(p) &\le& 2 \, \E[\|(\bfTh_t)_{t \ge 0}\|_p^{\alpha - p}] \;=\; 2\, \sum_{t \ge 0} \E[  | \bfA_t \cdots \bfA_1 \bfTh_0|^p] \\
&\le& 2\,s \sum_{t \ge 0} (\E[ | \bfA_s\cdots\bfA_1 |_{op}^p])^t,
\eeao 
and $\E[ |\bfA_s\cdots\bfA_1|_{op}^p] < 1$, for $p < \alpha$ and $s\ge 1$ fixed sufficiently large in the setting of Example~\ref{example:sre}. Hence, for $p \in (\alpha/2,\alpha)$,  $c(p) < \infty$ in \eqref{eq:constant:cp:1:tcl}, and then the series admits a $p-$cluster process $\bfQ^{(p)}$.

\par 
We state now Proposition~\ref{prop:sre:cs} which verifies conditions {\bf AC}, ${\bf CS}_p$ for the SRE equation. The proof is postponed to Section~\ref{sec:proof:sre:cs}.
\begin{proposition}\label{prop:sre:cs}
Let $(\bfX_t)$ be a stationary time series with values in $(\mathbb{R}^d,|\,\cdot\,|)$, as in Example~\ref{example:sre}. Let $p > \alpha/2$, and consider $(x_n)$ such that there exists $\kappa' > 0$ satisfying $n/x_n^{p\land (\alpha - \kappa')} \to 0$, as $n \to  \infty$. Then, $(x_n)$ satisfies conditions ${\bf AC}$ and ${\bf CS}_p$.
\end{proposition}

In the setting of SRE equations, condition ${\bf AC}$ has been shown in Theorem 4.17 in \cite{mikosch:wintenberger:2013}. In \cite{mikosch:wintenberger:2013}, the authors already considered a condition similar to ${\bf CS}_p$. Parallel to their setting, we propose a proof of Proposition~\ref{prop:sre:cs} which shows ${\bf CS}_p$ holds over uniform regions $\Lambda_n = (x_n,\infty)$ such that $n/x_n^p \to 0$, as $n \to \infty$, in the sense of \eqref{eq:sre:unif}. Thereby, our proof extends Theorem 4.17 in \cite{mikosch:wintenberger:2013} to uniform regions $\Lambda_n$ not having an upper bound.

Concerning the mixing properties of $(\bfX_t)_{t\ge 0}$ as in Example~\ref{example:sre}, we use that it is a Markov chain and that $\bfX_0$ has the stationary distribution. As mentioned in Remark~\ref{remark:markov:beta}, we can then use Markov chain's theory to compute its mixing coefficients; cf. \cite{meyn:tweedie:1993}. We review Theorem 2.8. in \cite{basrak:davis:mikosch:2002}, yielding an exponential decay of the mixing-coefficients $(\beta_t)$ of the series.
For a general treatment see Chapter 4.2 in \cite{buraczewski:damek:mikosch:2016}.

\begin{proposition}\label{prop:mixing}
Consider a time series $(\bfX_t)$ with values in $(\mathbb{R}^d,|\cdot|)$, as in Example~\ref{example:sre}. Assume  there exists a Borel measure $\mu$ on $(\mathbb{R}^d,|\cdot|)$, such that the Markov chain $(\bfX_t)_{t\ge0}$ is $\mu-$irreducible, i.e., for all $C \subset \mathbb{R}^d$ with $\mu(C) > 0$,
    \beam\label{eq:irreducible} 
     \sum_{t = 0}^\infty  \P( \bfX_t \in C \,|\, \bfX_0 = \bfx) > 0, \quad \bfx \in \mathbb{R}.
    \eeam 
Then $(\bfX_t)$ has mixing coefficients $(\beta_t)$ satisfying $\beta_t = O(\rho^t)$ for some $\rho \in (0,1)$, and we say it is strongly mixing with geometric rate. Moreover, $(\bfX_t)_{t\ge 0}$ is irreducible with respect to the Lebesgue measure if $(\bfA, \bfB)$ admits a density.
\end{proposition}
We can now state the asymptotic normality of cluster-based estimator for SRE solutions in Theorem~\ref{thm:sre:clt} below. The proof is postponed to Section \ref{proof:sre:clt}.
\begin{theorem}\label{thm:sre:clt}
Consider $(\bfX_t)$ to be the causal solution to the SRE in \eqref{eq:sre:def} with values in $(\mathbb{R}^d, |\,\cdot\,|)$, as in Example~\ref{example:sre}.
Assume the conditions of Proposition~\ref{prop:mixing} hold. Consider $p > \alpha/2$, 
and sequence $(k_n)$ satisfying \eqref{eq:k:mo}, and
 $m_n/k_n \to \infty$, 
 and $(k_n^\prime)$ such that  
 $k/k^\prime \to 0$ as $\nto$.
  Consider
  $f_\alpha(p):\tilde{\ell}^p \to \mathbb{R}$, and assume $\bf L$, {\bf S} hold. 
Then, 
\beao 
\sqrt{k_n}\, \big( \, \widehat{f_{\widehat \a}^\bfQ}(p) - f_\a^\bfQ(p) \,\big) &\xrightarrow[]{d}& \mathcal{N}\big( \, 0, \Var(\, f_\a(Y \bfQ^{(p)})\,) \, \big), \quad n \to  \infty,
\eeao 
under the bias conditions $\bfB_\a(k)$, $\bfB(k^\prime)$, and the result extends to $\widehat \a-$cluster inference.
In particular, 
the $\widehat \a-$cluster based estimators from Section~\ref{sec:examples:cluster} in \eqref{eq:estimator:Ei} \eqref{eq:estimator:c1}, and \eqref{eq:cluster:sizes}, are asymptotically normally distributed. 
\end{theorem}
\bre
In this example, the asymptotic variances of the $\a-$cluster based estimators from Section~\ref{sec:examples:cluster} in \eqref{eq:estimator:Ei} \eqref{eq:estimator:c1}, and \eqref{eq:cluster:sizes} are non-null. The limiting variances in Theorem \ref{thm:sre:clt} are difficult to compare with the existing ones in the literature because of the complexity of the distribution of $\bfQ^{(p)}$. However, we provide  simple $\ell^\a-$block estimators of the asymptotic variances in Section \ref{sec:numerical:experiments}.
\ere

\section{Numerical experiments}\label{sec:numerical:experiments}
This section aims to illustrate the finite-sample performance of the $\widehat{\alpha}-$cluster estimators on time series $(\bfX_t)$ with tail-index $\alpha>0$. In all the models we consider in Section~\ref{sec:examples}, we work under the assumption that the tuning parameters of the ${\alpha}-$cluster satisfy \eqref{eq:k:mo}. We take $\kappa^\prime = 1$ in \eqref{eq:k:mo} which yields $b = \sqrt{n/k}$. In this case, the implementation of our estimators can be written solely as a function of 
{
$k$ and $k^\prime$. Recall $k = k_n $ must satisfy $k \to \infty$,  $m/k \to \infty$ with $m = [n/b]$, $n/k^\prime  \to \infty$, and $k/k^\prime \to 0$ as $n\to \infty$.}
Numerical comparisons of our $\widehat \a-$cluster based approach with other existing estimators for the extremal index and the cluster index are at the advantage of our approach; see \cite{buritica:meyer:mikosch:wintenberger:2021} and \cite{buritica:mikosch:wintenberger:2021}. The code of all numerical experiments is available at: \url{https://github.com/GBuritica/cluster_functionals.git}.

\subsection{Tuning the Hill estimator}\label{subsec:rep:alpha}
 
 We recommend choosing the tuning sequence of the tail-index and of the cluster estimators as $(k^\prime_n)$, $(k_n)$, respectively, such that $k/k^\prime \to 0$. 
 Roughly speaking, the cluster statistics capture the block extremal behavior whereas the tail-index recovers an extremal property of margins.
 In this section, we illustrate that the $\alpha$ cluster-based estimators perform well in simulation  when we use the Hill estimator $1/\widehat{\alpha}(k^\prime)$ with $k^\prime$ larger than $k$.
 To illustrate this point, we simulate $500$ samples $(X_t)_{t=1,\dots,n}$ of an AR$(\varphi)$ model with absolute value student$(\alpha)$ noise for $n = 12\,000$, $\alpha = 1$ and $\varphi  \in \{ 0.5, 0.7\}$, and for samples of Example~\ref{example:sre}. We estimate the extremal index $\widehat{\theta}_X(k)$ as in \eqref{eq:estimator:Ei} where we replace $\widehat \alpha$ by $\widehat{\alpha}(k^\prime)$. Recall that for an AR$(\varphi)$ model the asymptotic variances of the extremal index estimator are asymptotically null when $k/k^\prime \to 0$ as $n\to \infty$.  We see in Figures~\ref{fig:0.78000}, \ref{fig:0.58000} and \ref{fig:sre:k1k2} that in practice we have to choose $k$ small to reduce the bias of the estimator.  Moreover, the estimation procedure is robust with respect to $k^\prime$ therefore we recommend taking $k^\prime$ large to reduce variance. Similar results were found for $n = 3\,000$, $n=5\,000$, and $n = 8\, 000$ and these are available upon request. 
To conclude, we see in  Figures~\ref{fig:0.78000}, \ref{fig:0.58000} and \ref{fig:sre:k1k2} that standard deviations are small, and thus the error of cluster inference is mainly due to bias. 
For this reason, we recommend choosing $k$ small and $k^\prime$ larger in all settings.

\begin{figure}
     \centering
    \begin{subfigure}[b]{0.45\textwidth}
         \centering
    \includegraphics[width=\textwidth]{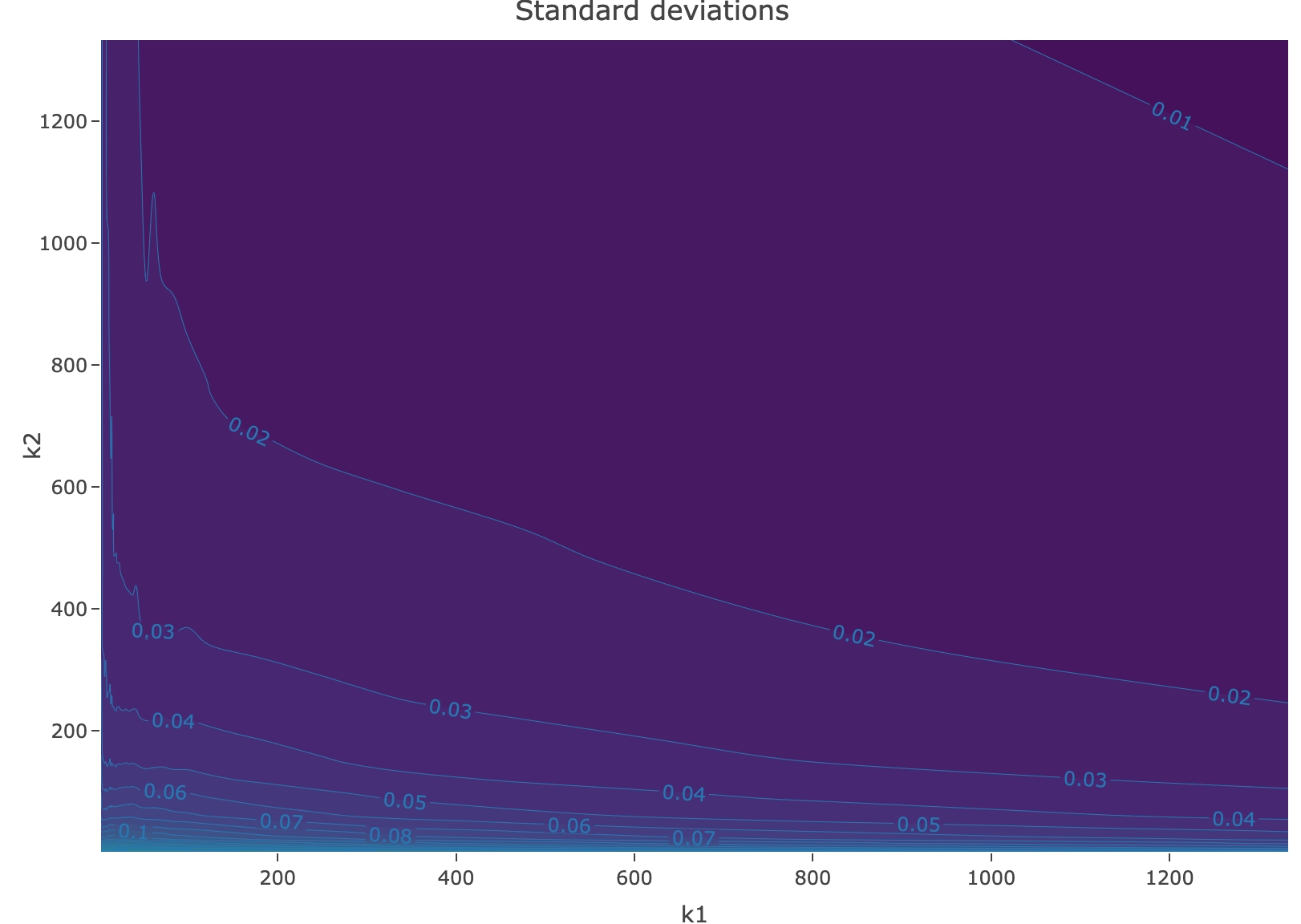}
     \end{subfigure}
     \hfill
     \begin{subfigure}[b]{0.43\textwidth}
     \includegraphics[width=\textwidth]{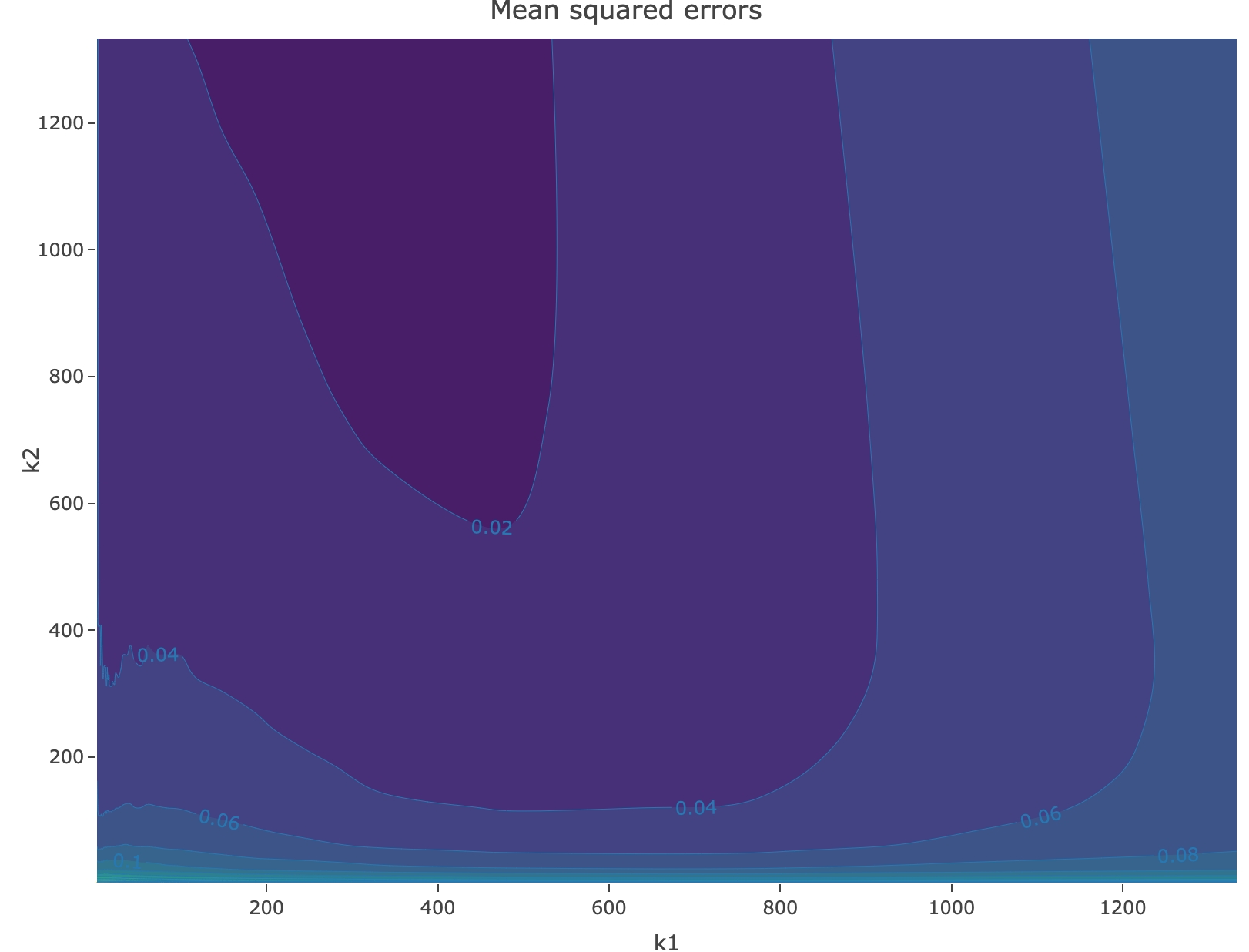}
     \end{subfigure}
     \caption{ Heatmap with contour curves of standard deviations and mean squared errors for estimates of the extremal index $ k1 = k \mapsto \widehat{\theta}_{X}(k)$ in \eqref{eq:estimator:Ei}, using the {\blue Hill} estimator $k2 = k^\prime \mapsto \widehat{\alpha}(k^\prime)$. 
     We simulate $500$ samples $(X_t)_{t=1,\dots,n}$ of an AR$(\varphi)$ model with absolute value student$(\alpha)$ noise for $n = 12\,000$, $\varphi = 0.5$, $\alpha = 1$, such that $\theta_{X} = 0.5$.  
		 }
     \label{fig:0.78000} 
\end{figure}

\begin{figure}
     \begin{subfigure}[b]{0.45\textwidth}
         \centering
    \includegraphics[width=\textwidth]{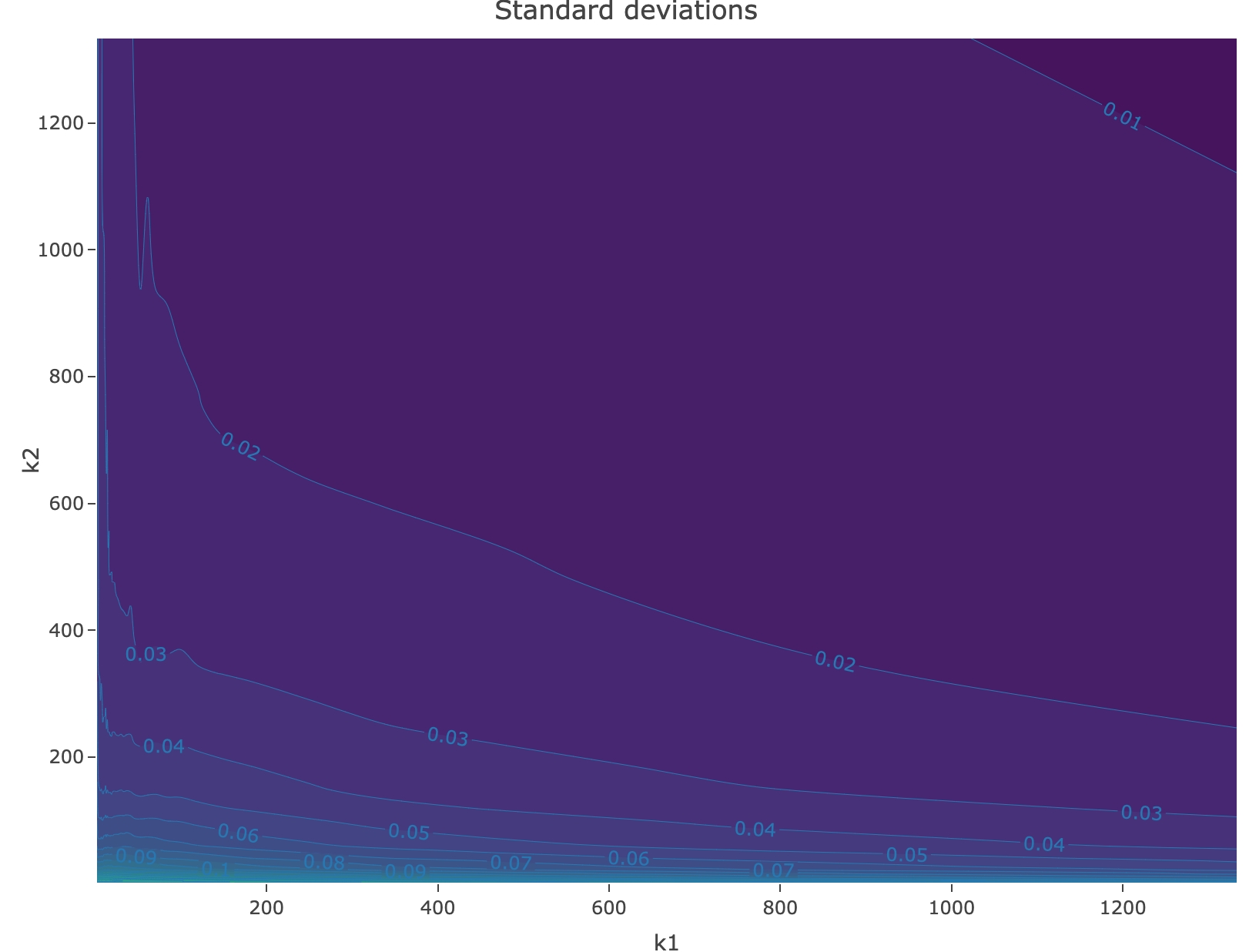}
    \end{subfigure}
     \hfill
     \begin{subfigure}[b]{0.44\textwidth}
         \centering
  \includegraphics[width=\textwidth]{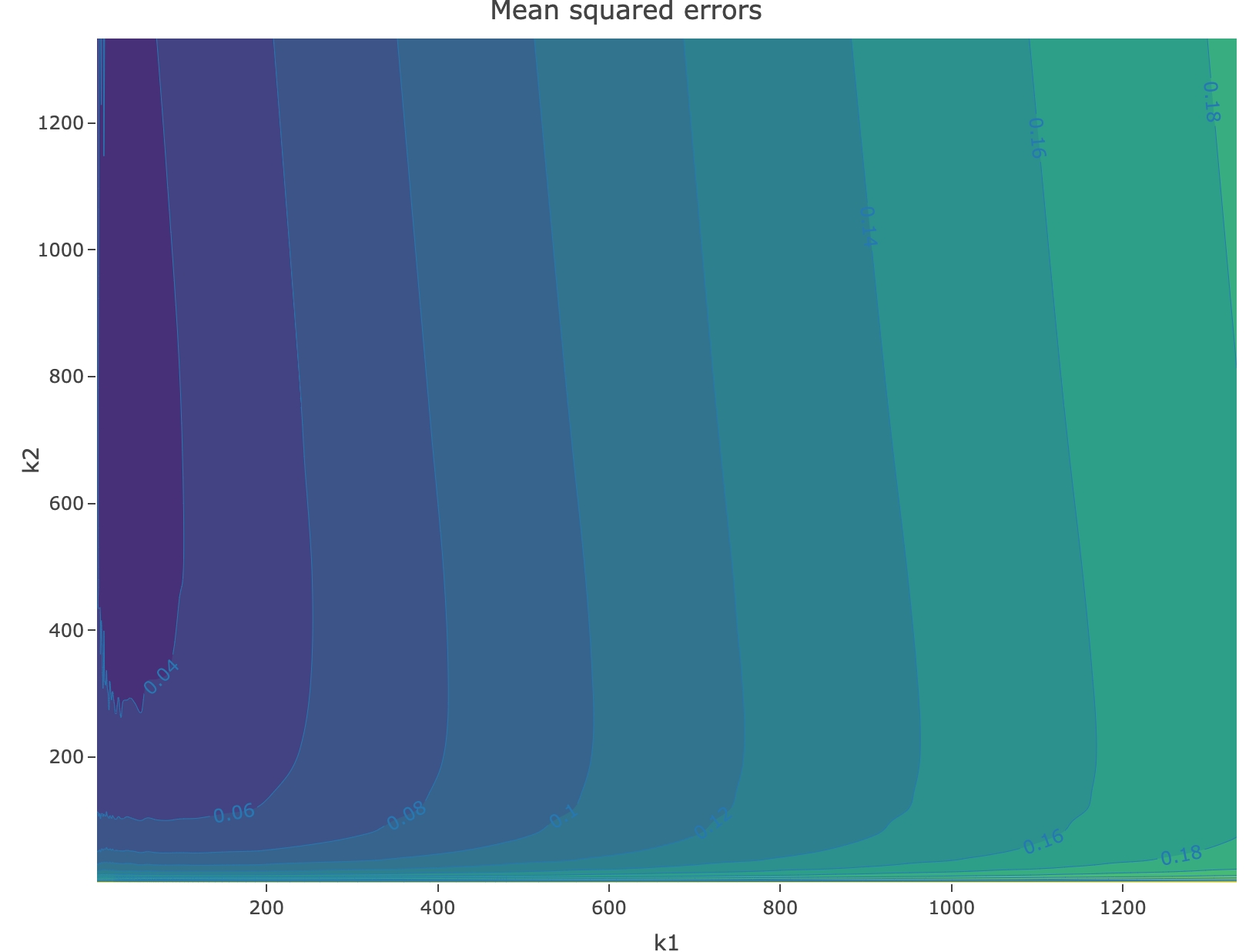}
      \end{subfigure}
      \caption{Heatmap with contour curves as in Figure~\ref{fig:0.78000}. Here we simulate $500$ samples $(X_t)_{t=1,\dots,n}$ of an AR$(\varphi)$ model with absolute value student$(\alpha)$ noise for $n = 12\,000$, $\varphi = 0.7$, $\alpha = 1$, such that $\theta_{X} = 0.3$.  }
     \label{fig:0.58000}
\end{figure}     

\begin{figure}
     \begin{subfigure}[b]{0.45\textwidth}
         \centering
    \includegraphics[width=\textwidth]{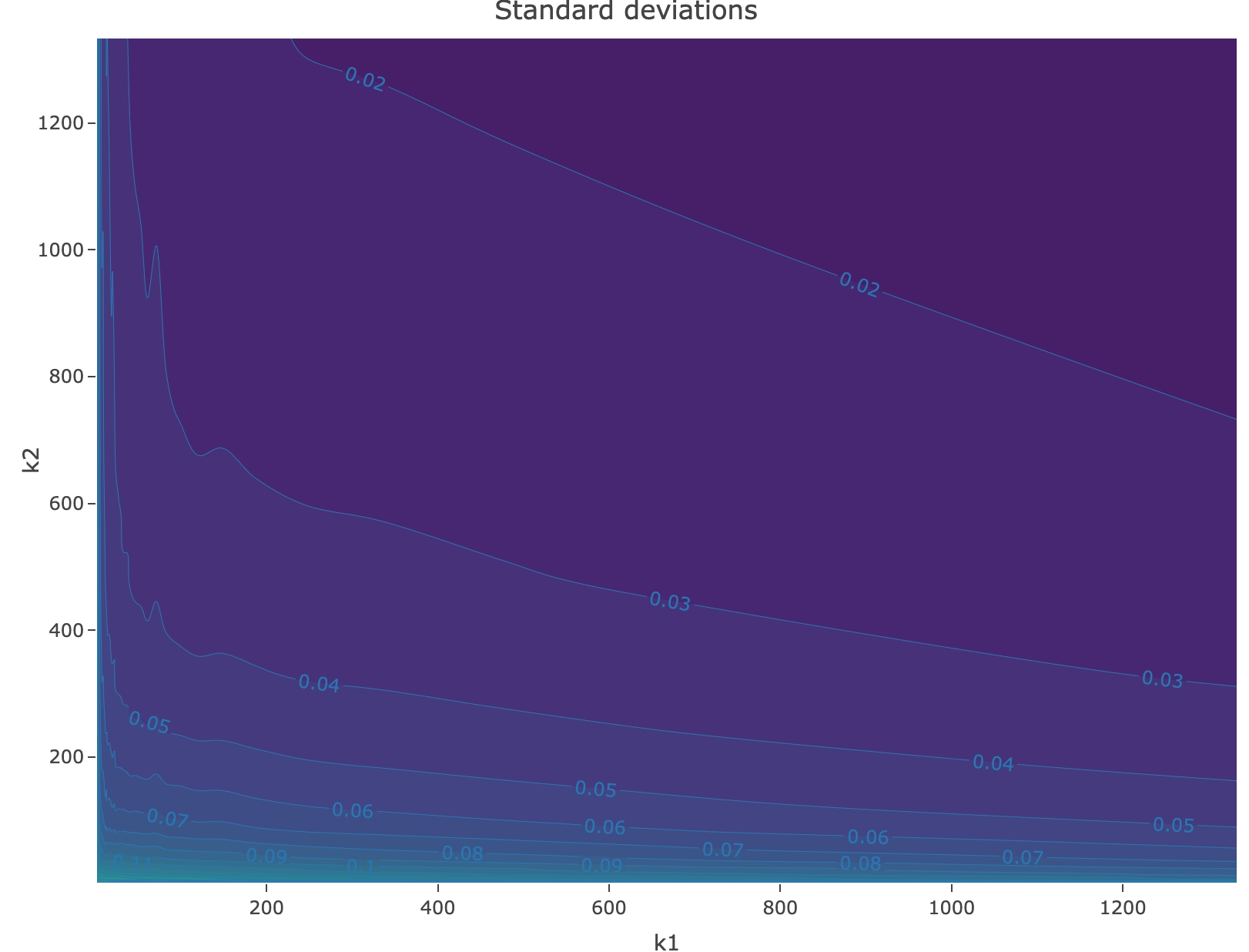}
    \end{subfigure}
     \hfill
     \begin{subfigure}[b]{0.44\textwidth}
         \centering
     \includegraphics[width=\textwidth]{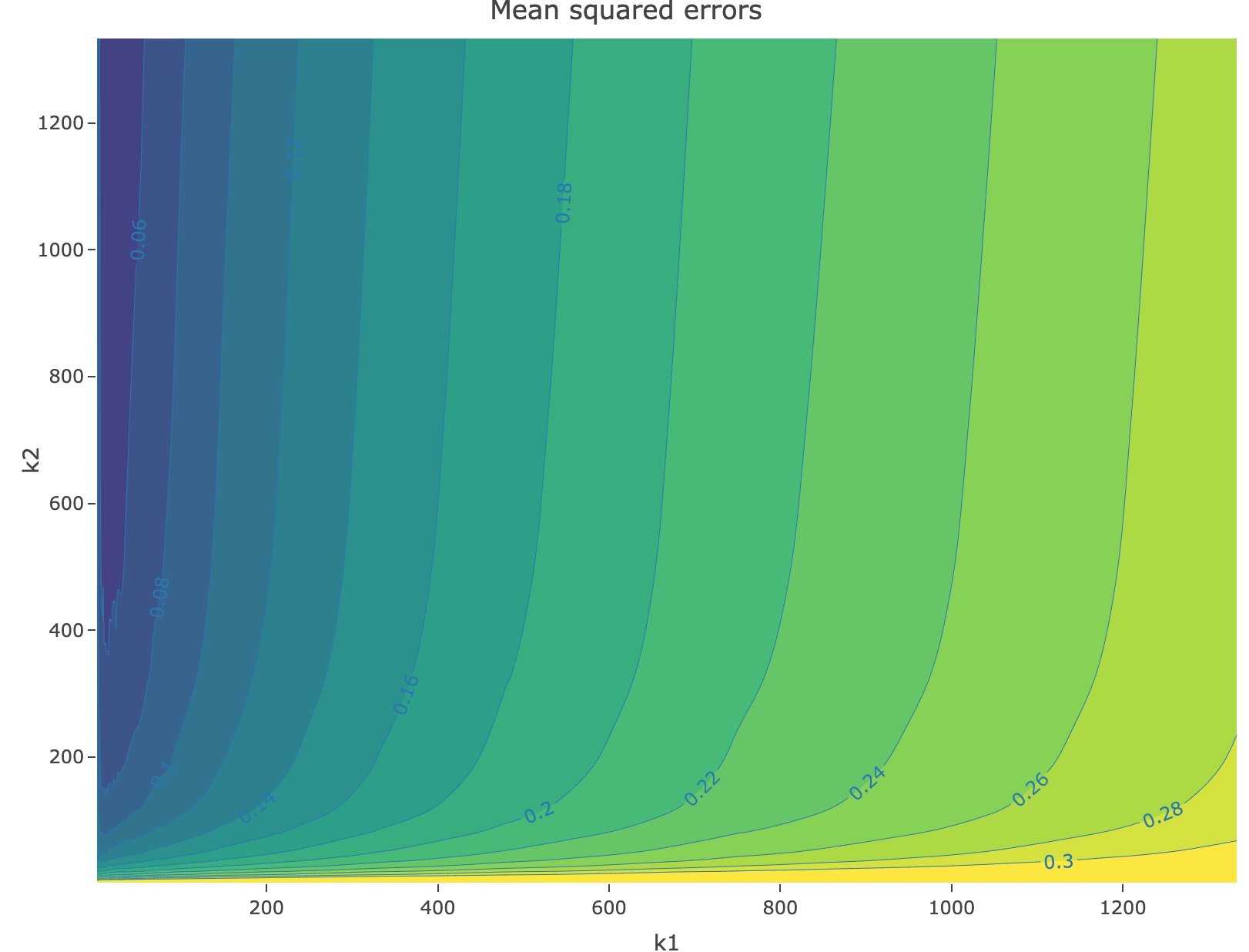}
     \end{subfigure}
      \caption{Heatmap with contour curves as in Figure~\ref{fig:0.78000}. Here we simulate $500$ samples $(X_t)_{t=1,\dots,n}$ of Example~\eqref{ex:sre} for $n = 12\,000$ such that $\theta_{X} \approx 0.2792$.  }
     \label{fig:sre:k1k2}
\end{figure}

\subsection{Cluster size probabilities} We reviewed in Section~\ref{example:cluster:sizes} how cluster sizes play a key role to model the serial behavior of exceedances. In this section, we implement the cluster size probabilites estimation procedure from Equation~\eqref{eq:cluster:sizes} in an example of a solution to the SRE under Kesten's conditions.

\begin{example}\label{ex:sre}
Consider the non-negative univariate random variables $A$, $B$, defined by $\log A = N - 0.5$, where $N$ denotes a standard Gaussian random variable, and $B$ is uniformly distributed in $[0,1]$. Let $(X_t)$ be the solution to the SRE in \eqref{eq:sre:def}. Then, $(X_t)$ satisfies ${\bf RV}_\alpha$ with $\alpha = 1$. If $(A_j)$ is a sequence of iid random variables with generic element $A$, then 
\beao 
Q_t &\eqd&  
\Pi_t/\|(\Pi_j)\|_{\alpha}\,,\qquad t\in \Z\,,
\eeao 
with 
\beao 
\Pi_t &\eqd& \begin{cases}
A_t \cdots A_1 & \text{ if }\quad t \ge 1, \\
A_{t} \cdots A_{-1}  & \text{ if } \quad t \le -1,\\
1& \text{ if } \quad t=0\,.
\end{cases}
\eeao 
This follows by Example 6.1 in \cite{janssen:segers:2014}, and Proposition 3.1 in \cite{buritica:mikosch:wintenberger:2021}. Then, for $p > \alpha/2$, the $p-$cluster based estimators \eqref{eq:estimator:Ei} \eqref{eq:estimator:c1}, and \eqref{eq:cluster:sizes} are asymptotically normally distributed. 
\end{example}

Recall the cluster sizes  $\pi_1,$ $\pi_2,$ $\dots $,  defined in \eqref{eq:xmas25c}. We infer the cluster sizes of Example~\ref{ex:sre} using $ \a-$cluster estimates. 
To illustrate Theorem~\ref{thm:sre:clt}, we run a Monte--Carlo simulation experiment based on $1\,000$ samples $(X_t)_{t=1,\dots, n}$ of length $n = 12\,000$ from Example~\ref{ex:sre}. For each sampled trajectory,  we obtain estimates $\widehat{\pi}_1,\widehat{\pi}_2,\dots$, letting $k = 10$ and $b = 38$ in \eqref{eq:cluster:sizes}. 
For the implementation we use Hill estimates of the tail-index $\widehat \a(k^\prime)$ with $k^\prime = 1\,000$.  We also estimate the extremal index $\theta_X$ of the series from Equation~\eqref{eq:estimator:Ei}. Theorem~\ref{thm:sre:clt} yields, for $j \ge 0$,
\beam\label{eq:asymptotic:variance:pi} \Var(\sqrt{k_n}\, (\widehat{\pi}_j - \pi^\bfQ_j)) &\to& \Var(\pi_j^\bfQ( \bfQ)), \quad n \to \infty,
\eeam 
where $\pi_j^\bfQ$ are the cluster functional yielding the cluster sizes $\pi_j$ with the notation in \eqref{eq:clt:cluster}. Notice that the asymptotic variance of our cluster sizes estimate is again a cluster statistic that we can infer.
We compute an estimate of the asymptotic variance in \eqref{eq:asymptotic:variance:pi}  using cluster-based estimates, and compare this estimate with the empirical variance obtained from the Monte-Carlo simulation study. 
Figure~\ref{fig:hist} plots the profile of the limit Gaussian distribution where the asymptotic variance is computed in these two ways. As expected from Equation~\eqref{eq:asymptotic:variance:pi}, the curves overlap, even if $k$ is small. 
In our simulation, a clear bias appears when we choose $k$ larger.
\par 
In the case of SRE equations, the cluster sizes were studied in detail in \cite{haan:resnick:rootzen:vries:1989}. The authors proposed a method to approximate the true values when the tail-index $\alpha$, and the random variable $A$ are known. We approximate true values using Equation 3.5 in  \cite{haan:resnick:rootzen:vries:1989}, and a Monte-Carlo study with $10\,000$ samples of length $500\,000$. The obtained values are pointed out in red in Figure~\ref{fig:hist}. We see that this choice of $k$ yields estimates centered around the true value. 


\begin{figure}
     \centering
    \begin{subfigure}[b]{0.49\textwidth}
         \centering
    \includegraphics[width=\textwidth]{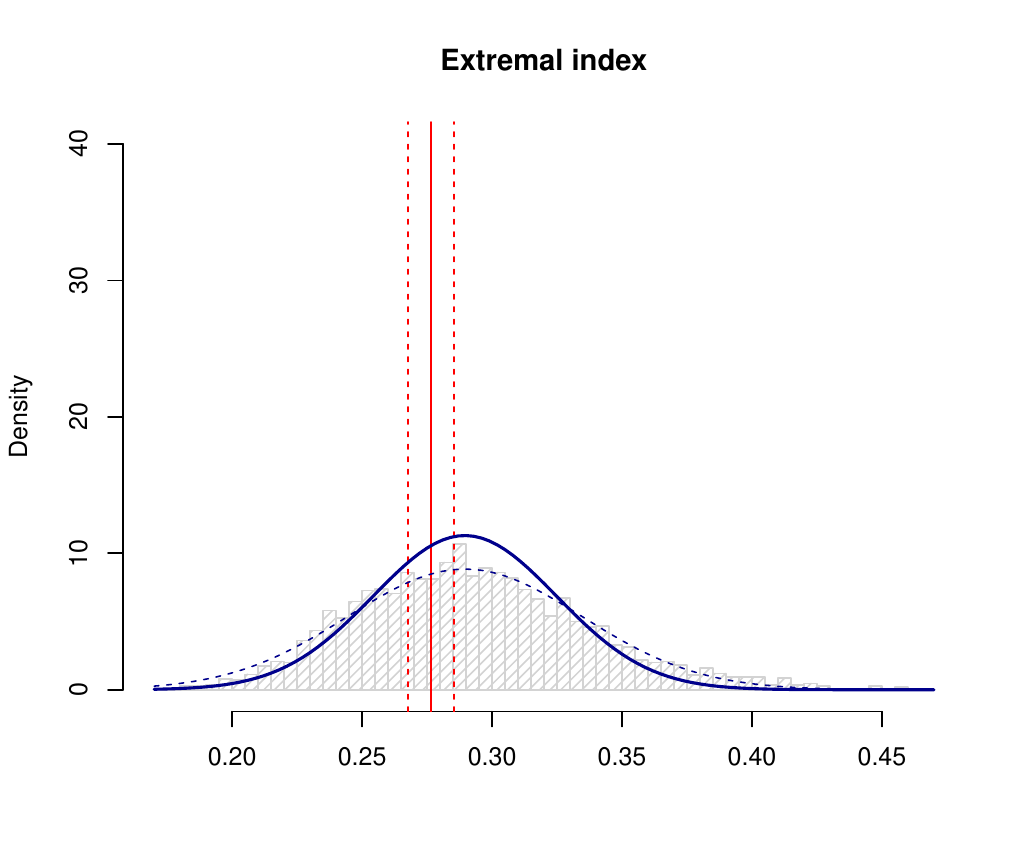}
     \end{subfigure}
     \hfill
     \begin{subfigure}[b]{0.49\textwidth}
     \includegraphics[width=\textwidth]{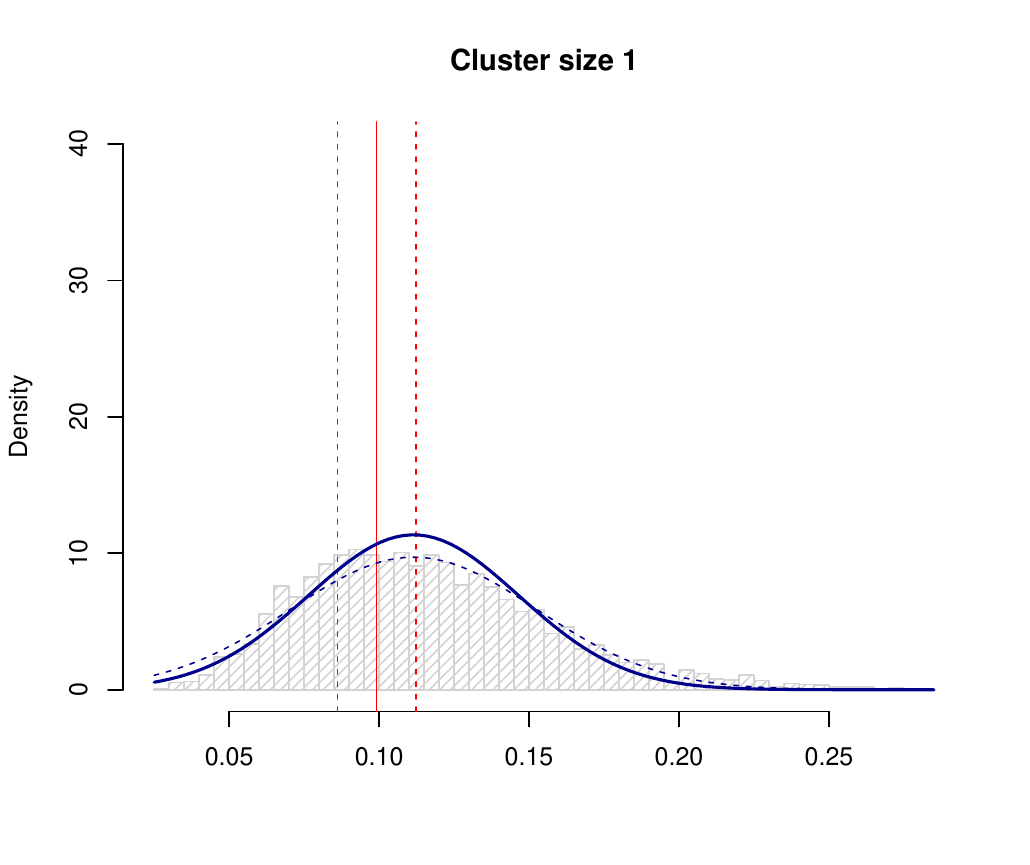}
     \end{subfigure}
     \begin{subfigure}[b]{0.49\textwidth}
        \centering
    \includegraphics[width=\textwidth]{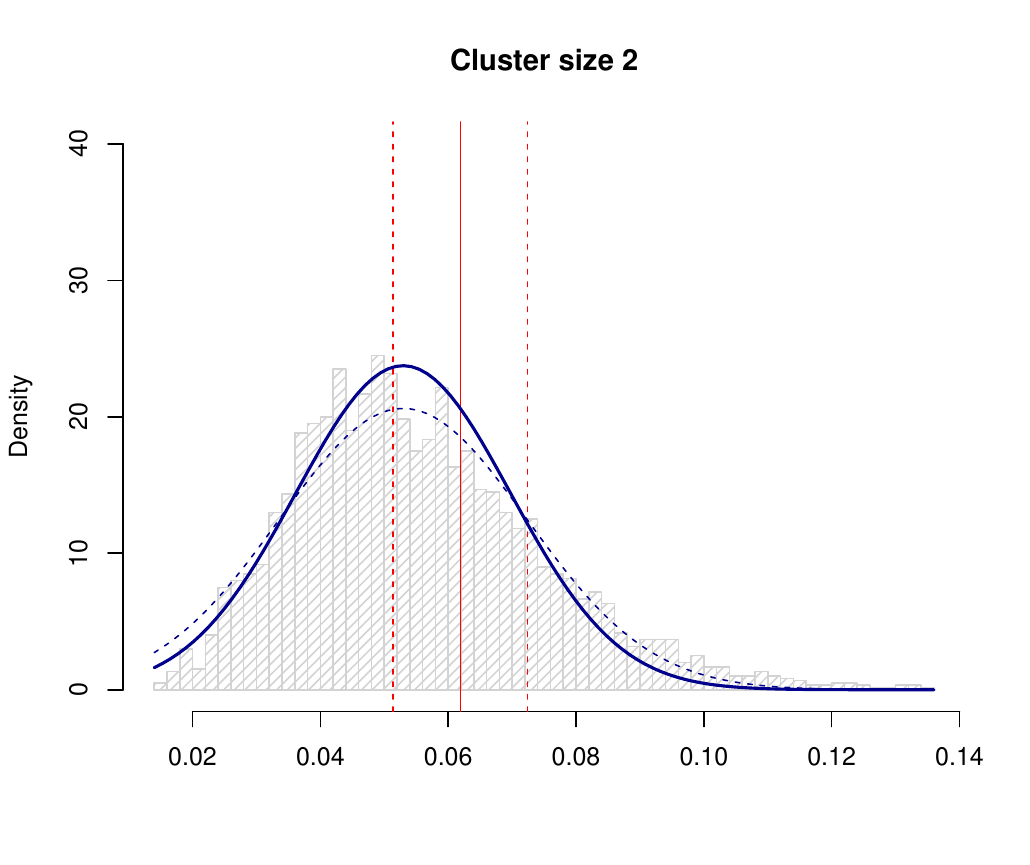}
     \end{subfigure}
     \hfill
     \begin{subfigure}[b]{0.49\textwidth}
     \includegraphics[width=\textwidth]{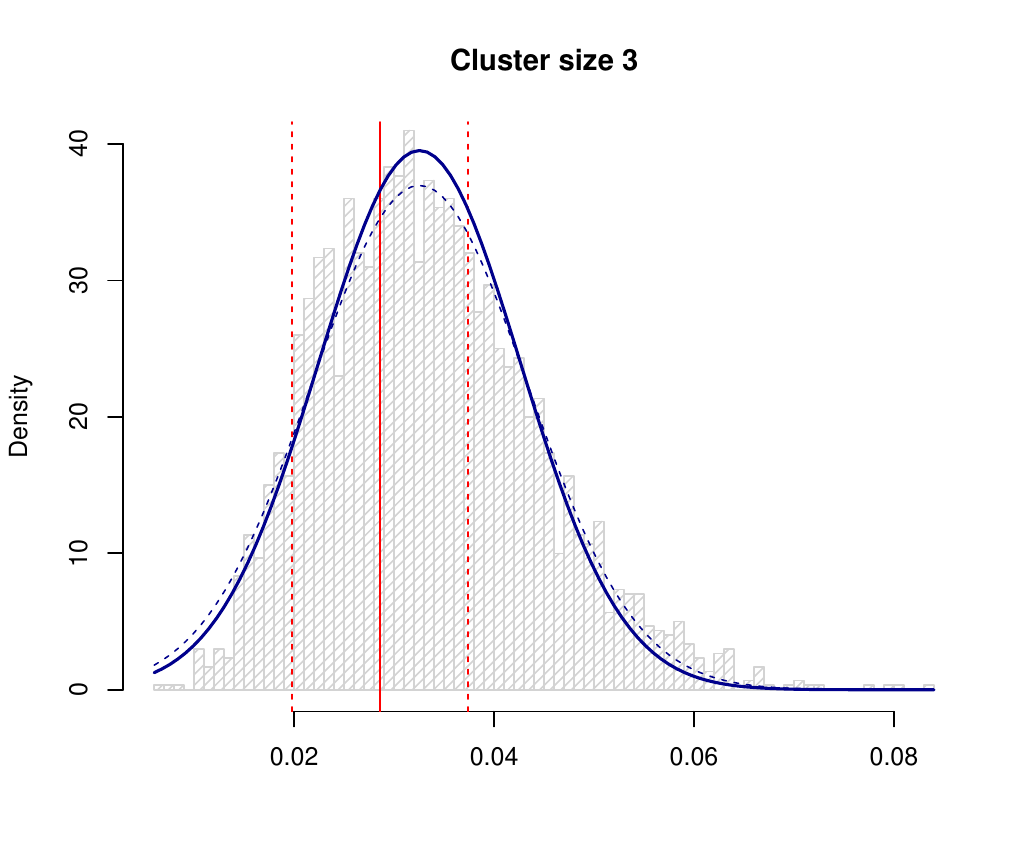}
     \end{subfigure}
     \caption{ Histogram of estimates $\widehat{\theta}_X$ of the extremal index  using \eqref{eq:estimator:Ei}, and the cluster size probability $ \widehat{\pi}_1$, $ \widehat{\pi}_2$, $ \widehat{\pi}_3$, using \eqref{eq:xmas25c}. We simulate $1\,000$ samples $(X_t)_{t=1,\dots,n}$ of Example~\ref{ex:sre} with $n = 12\,000$. The Gaussian density curves are centered in the median of the estimators. Their variances are  estimated by Monte-Carlo (dotted curve) or using the average of the cluster-based estimate of the asymptotic variance defined in \eqref{eq:asymptotic:variance:pi} (solid curve). The red lines point to the Monte-Carlo  approximation of the real values with standard deviation. These were computed using Equation 3.5 in  \cite{haan:resnick:rootzen:vries:1989}, and a simulation study with $10\,000$ samples of length $500\,000$.
	 }
     \label{fig:hist} 
\end{figure}

\subsection{Conclusion}
 Our main theoretical result in Theorem~\ref{thm:main} states asymptotic normality of $\a$ cluster-based disjoint block estimators $\widehat{f}_\alpha^\bfQ(\widehat \alpha)$, based on $k$ extremal $\ell^{\widehat \a}-$blocks, where $\widehat \alpha$ is an estimate of the tail index of the series. 
 The advantage of $\widehat \a$ cluster-based methods is that the choice of $k$ is robust to time dependencies and that it fully describes clusters of extreme values; see \cite{buritica:mikosch:wintenberger:2021}.
Equation~\eqref{eq:central:limit:1} characterizes their asymptotic variance in terms of a cluster statistic that we can also infer. 
We further show in Section~\ref{sec:examples:cluster} that many important indices in extremes can be written in terms of an  $\a-$cluster statistic, e.g., the extremal index and cluster sizes. 
Section~\ref{sec:examples} verifies that our assumptions hold for numerous models like causal linear models and SRE solutions under Kesten's conditions. For linear models, we obtain estimators with null asymptotic variance for classical indices as first conjectured in \cite{hsing:1993}. In the examples we considered, our estimators have a small variance that can also be estimated. 
To illustrate the performance of our $\a-$cluster inference methodology, we run finite-sample simulations in Section~\ref{sec:numerical:experiments}. 
Our simulations support that replacing $\alpha$ by $\widehat{\alpha}(k^\prime)$ as in Section~\ref{subsec:rep:alpha} does not have a big impact on the asymptotic variance.
 This is because in practice $k$ needs to be chosen small to obtain unbiased estimates, whereas $k^\prime$ can be chosen larger. 
Then, even if we choose $k$ small, the uncertainty of our procedure is well quantified by plugging an estimate of the asymptotic variance in the Gaussian limit. 
{
Finally, a complete study of the tuning parameters $k, k^\prime$ requires a careful analysis of the bias conditions for blocks: ${\bf B}_\a, {\bf B}$, as we pointed out in Remark~\ref{remark:bias}, which we see as a road for future research. 
}



\appendix



\section{Consistency of block estimators}\label{sec:consistency}
Let $p > 0$. In the following we assume the conditions of Proposition~\ref{prop:existence:cluster:process} hold. In this setting, the time series $(\bfX_t)$ admits a $p-$cluster $\bfQ^{(p)} \in \tilde{\ell}^p$ and \eqref{eq:constant:cp:1:tcl}, \eqref{limit:cluster:process} hold for $(x_n)$.
For inference purposes, we fix 
the sequence of block lengths $(b_n)$ and write $m_n = \lfloor n/b_n \rfloor$, such that $b_n \to \infty$, $m_n \to  \infty$.
 We assume 
\beam 
k \quad:=\quad  k_n(p) &=& \;\big\lfloor m_n\P(\|\mathcal{B}_1\|_p >  x_{b_n})\big\rfloor\label{eq:k:3} \\
&\sim& n\, c(p) \P(|\bfX_1| > x_{b_n}), \qquad n \to  \infty,\label{eq:k:2}
\eeam
where {the last relation can be concluded from Equation~\eqref{eq:constant:cp:1:tcl} and $c(p) \in (0,\infty)$ is as defined therein}.
We can verify $m_n/k_n \to \infty$ using the relation $n \P(|\bfX_1|>x_{b_n})\to 0$, as $n\to\infty$.
\par 
To infer the tail-index of the series $\alpha$ we implement the Hill estimator in \eqref{eq:hill:estimator}, hence we consider sequences $(x^\prime_n)$ and $(k_n^\prime)$ such that 
$(x_n^\prime)$  verifies \eqref{eq:constant:cp:1:tcl} for $p = \infty$ and 
\beam
k^\prime \; := \; k_n^\prime &=& \lfloor  n \P(|\bfX_1| > x^\prime_{b_n})\rfloor , \label{eq:k:app}
\eeam
holds.
Similarly as before, we can verify $m_n/k_n^\prime \to \infty$, as $n \to \infty$. 
The sequences $(b_n)$, $(m_n)$, $(x_n)$, $(k_n)$, and  $(x_n^\prime),$ $(k_n^\prime)$ that we consider henceforth are the ones fixed here.

\subsection{Proof of Lemma~\ref{lem:con}}


We start by showing the conditions in Lemma~\ref{lem:con} 
 entail \eqref{mixing:condition2} below. 
The proof of the next lemma is provided  in Appendix \ref{proof:lemma:mixing}.

\begin{lemma}\label{lem:mixing}
Let $p\in(0,\infty]$, 
let $f = f(p) \in \mathcal{G}_{+}(\tilde{\ell}^p)$ be a bounded Lipschitz continuous function. 
Assume  there exists a sequence $(\ell_n)$, satisfying $\ell_n \to  \infty$, as $n \to  \infty$,  such that the mixing coefficients $(\beta_t)$ satisfy
    \beam\label{eq:cond:mixing} 
        \lim_{n \to  \infty } m_n \beta_{\ell_n}/k_n & = & 
        \lim_{n \to  \infty} \ell_n/b_n \quad = \quad   0\,,
    \eeam 
    then the sequences $(x_{n})$ and $(b_n)$ satisfy, for all $u >0$, 
\begin{align}\label{mixing:condition2}
 \big| 
    \E \big[e^{ 
     - k_n^{-1} \sum_{t=1}^{m_n} f((ux_{b_n})^{-1} {\mathcal B}_t)\1(\|\mathcal B_t\|_p > u x_{b_n} ) }\big]
    - \E\big[ &e^{ 
    - k_n^{-1}f((ux_{b_n})^{-1} {\mathcal B}_1)\1(\|\mathcal B_1\|_p > u x_{b_n}) }
    \big]^{m_n}
    \big]
    \big|\nonumber\\ &\to 0, \qquad n \to \infty.
\end{align}
\end{lemma}

Next, we state the consistency of $\ell^p-$block estimators for arbitrary $p$, and we cover the case of  $\ell^{\widehat \a}-$block estimators.  
The proof of the next Lemma is deferred to Appendix~\ref{proof:lem:consis:complete}.
\begin{lemma}\label{pro:cor:consis}
Assume the mixing condition in \eqref{mixing:condition2} holds. 
Consider a function $f = f_\a(p) \in \mathcal{G}_+(\tilde{\ell}^p)$, not necessarily bounded,
 and assume $\E[f(Y\bfQ^{(p)})] < \infty$. 
 Moreover, assume
\beam \label{eq:cond:consistency:unbounded}
\limsup_{n \to \infty}\frac{
\E\big[\big(f(\bfX_{[1,n]}/x_{n})\big)^{1+\delta}\1(\|\bfX_{[1,n]}\|_{p} > x_{n}) \big]}
{\P(\|\bfX_{[1,n]}\|_p > x_{n})} < \infty.
\eeam 
Then, the $\ell^p-$block estimator in \eqref{eq:estimator:cluster} satisfies 
\beao 
\widehat{f^\bfQ}(p)  &\xrightarrow[]{\P}& f^{\bfQ}(p), \quad n \to  \infty.
\eeao 
Let $f = f_\a(p) :\tilde \ell^{p} \to \mathbb{R}$
and assume ${
\bf S}$ holds, and $f_\a(p)$, and $
\partial f_q/\partial q|_{q = 
\alpha}(p)$ satisfy {\bf C}, 
for fixed $\epsilon,\delta > 0$. Assume also $\widehat \alpha_n \xrightarrow[]{\P} \alpha$, 
as $n \to \infty$.  Then, 
\beao 
\widehat{f^\bfQ_{\widehat \a}}(p)  &\xrightarrow[]{\P}& f_{
\alpha
}^{\bfQ}(p), \quad n \to  \infty.
\eeao 
Furthermore, if $f_
\alpha(\alpha), 1(\a)$, and $
\partial f_q/\partial q|_{q = 
\alpha}(\alpha)$ satisfy ${
\bf C}$ for $p = 
\alpha$, then $\widehat{f_{\widehat \a}^\bfQ}(\widehat \a) \xrightarrow[]{\P} f_\a^{\bfQ}(\alpha)$, as $n \to \infty$.
\par
\end{lemma}

\begin{proof}[Proof of Lemma~\ref{lem:con}]
The proof consists in applying  Lemma~\ref{pro:cor:consis}. 
Notice Lemma~\ref{lem:con} already assumes ${\bf C}$ for fixed $\epsilon,\delta > 0$ and it assumes $c(q) < \infty$, for $q \in [\a-\epsilon, \a+\epsilon]$.
In addition, Lemma~\ref{lem:mixing} entails condition~\eqref{mixing:condition2} which we require to apply Lemma~\ref{pro:cor:consis}.
This verifies all the assumptions of  Lemma~\ref{pro:cor:consis} and thus this concludes the proof.
\end{proof}

\section{Asymptotic normality of block estimators}\label{sec:asymptotic:normality}
\subsection{Notation}
We focus on the case $p = \a$, $\a>0$. 
The general result on $p-$cluster inference, for $p \in (0, \infty]$, follows the same lines of this proof fixing $q=p$ thus we omit the details.
Let $f_\a \in \mathcal{G}_+(\tilde{\ell}^\a)$ be a cluster functional and let $f_\a ^\bfQ$ be its  $\a$-cluster statistic
\beam \label{eq:clust:statistic}
 f_\a^{\bfQ} &=& \E[f_\a(Y \bfQ)].
\eeam 
Let $\mathcal{F}_f \subseteq \mathcal{G}_+(\tilde{\ell}^\a)$ be the set including the functions: $\bfx \mapsto f_\a(\bfx)$, $\bfx \mapsto 1(\bfx)$,  and $\mathcal{F}_h \subseteq \mathcal{G}_+(\tilde{\ell}^\infty)$ the set containing
  $\bfx \mapsto h(\bfx)$ and 
$\bfx\mapsto e(\bfx)$ as in \eqref{eq:func:hill}, \eqref{eq:func:exceedances}, and define 
\beam \label{eq:F}
\mathcal F &=& \mathcal F_f \cup \mathcal F_h.
\eeam 
\par 
For inference purposes, recall we denote disjoint blocks as 
\beao 
  \mathcal{B}_t \quad :=  \quad \bfX_{(t-1){b_n} +[1,b_n]}, \qquad t=1, \dots, m_n.\,
\eeao 
In the following sections it will be useful to consider the deterministic threshold estimators defined by 
\beam\label{eq:oracle:cluster}
 \widetilde g^{\bfQ}(u,q) &:=&  \frac{1}{k_n}\sum_{t=1}^{m_n} g({\mathcal B}_{t}/(u\,x_{b_n}))\1(\|{\mathcal B}_{t}\|_{q} > u \,x_{b_n})\,,\qquad g\in \mathcal F_f\,,\\
 \widetilde g^\bfQ(u) &:=&  \frac{1}{k_n^\prime} \sum_{t = 1}^{m_n} g(\mathcal{B}_t/(u \,x^\prime_{b_n})) \1( \|\mathcal{B}_t\|_\infty > u\, x^\prime_{b_n}),  \qquad g\in \mathcal F_h\,,\label{eq:hill:oracle} 
\eeam
 where $u,q > 0$.
 The sequences $(x_n)$, $(b_n)$, $(m_n)$, $(k_n)$ that we consider, defining the block estimator are the ones fixed in \eqref{eq:k:2}, and $(x_n^\prime), (k_n^\prime)$ are as in \eqref{eq:k:app}.
Let $\epsilon, \delta > 0$ be such that the conditions of Theorem~\ref{thm:main} are satisfied and define
\beao
\mathcal{T}_f &:=& \{ \widetilde{g}^\bfQ(u,q) \}_{ \{  
g \in \mathcal{F}_f,
u \in [1-\epsilon, 1+\epsilon],  
q \in [\a-\epsilon,\a +\epsilon]  \}}, \\
\mathcal{T}_h &:=& \{ \widetilde{g}^\bfQ(u) \}_{ \{ 
g \in \mathcal{F}_h,
u \in [1-\epsilon, 1+\epsilon] \}}. \nonumber 
\eeao 
We are interested in the family of stochastic processes
\beam \label{eq:t:1} \label{eq:t} 
\mathcal{T} &=& \mathcal T_f \cup \mathcal T_h\, ,
\eeam 
indexed by four functions $f_\a, 1, h, e,$ and the values of $(u ,q) \in [1-\epsilon,1+\epsilon]\times [\a - \epsilon, \a  + \epsilon]$, and abusing notation, we also note $\mathcal{T}$ the index set relative to this family.
Consider also the blocks 
\beam\label{eq:notation:blocks}
 \mathcal{B}^*_t \quad := \quad  \bfX^*_{(t-1)b_n+[1,b_n]}, \qquad t=1, \dots, m_n\,,
\eeam 
such that the triangular array $(\mathcal{B}_t^*)_{1\le t \le m_n}$ is a $1$-dependent sequence.
In particular, this implies 
$(\mathcal{B}_t^*)_{1\le t \le m_n, t  \text{ odd}}$ is  a sequence of iid
blocks distributed as $\mathcal{B}_{1}$, and similarly for even time indices.
Define the  block estimators relative to $(\mathcal{B}_t^*)$ as
\beam\label{eq:oracle:iid}
\widetilde{g}_{,*}^\bfQ(u,q) 
 &=& \frac{1}{k_n }  \sum_{t=1}^{ m_n } g(\mathcal{B}^*_{t}/(u\, x_{b_n}) )\1( \|\mathcal{B}^*_{t}\|_q > u\, x_{b_n} )\,,  \quad (u,q,g) \in \mathcal{T},\\
\widetilde{g}_{,*}^\bfQ(u) 
 &=& \frac{1}{k^\prime_n }  \sum_{t=1}^{ m_n } g(\mathcal{B}^*_{t}/(u\, x^\prime_{b_n}) )\1( \|\mathcal{B}^*_{t}\|_\infty > u\, x^\prime_{b_n} ), \quad (u,q,g) \in \mathcal{T}. \nonumber 
\eeam



\subsection{Preliminaries}\label{sec:proof:sketch}
The main step to prove Theorem~\ref{thm:main} is the uniform central limit theorem of  the stochastic process: 
\beam \label{eq:sequence:st:processes:0}
\big\{\sqrt{k_n} \left( \widetilde{g}^\bfQ(u,q) - c(q)u^{-\a}g^{\bfQ}(q) \right) \big\}_{(u,q,g) \in \mathcal T  } 
 \eeam   
towards a centered Gaussian process $\mathbb{G} : \mathcal{T} \to \mathbb{R}$ 
defined by
\beam\label{eq:def:limit:gauss}
\mathbb G &=& \big(\mathbb G(g(\cdot/u), q) \big),
\eeam
with covariance structure

\beam\label{eq:covariance}\label{eq:cov:gaussian}
 \lefteqn{ \Cov\big( \mathbb{G}(g(\cdot/v),q) ,\mathbb{G}(f(\cdot/u), {q^\prime } ) \big) }\nonumber \\
&=& 
v^{-\alpha} \frac{c(q)}{c(p)} \E[ g ( Y\bfQ^{(q)})  f ( Y\bfQ^{(q)}  v/u  )\1(\|Y \bfQ^{(q)}\|_{q^\prime} > v/u) ], \\
\lefteqn{ \Cov\big( \mathbb{G}(g(\cdot/v),q) ,\mathbb{G}(h(\cdot/u), \alpha ) \big) }\nonumber \\
    &=& 
 v^{-\alpha} \kappa \frac{c(q)}{c(p)} \E[ g ( Y\bfQ^{(q)})  h ( Y\bfQ^{(q)}  (c(p)/\kappa)^{1/\alpha} v/u  )],\nonumber  \\ 
\lefteqn{ \Cov\big( \mathbb{G}(h(\cdot/v),\alpha) ,\mathbb{G}(e(\cdot/u), \alpha ) \big) }\nonumber \\
  &=& v^{-\alpha} \kappa \, \E[ h ( Y\bfQ)  e ( Y\bfQ v/u  ) ],\nonumber 
\eeam
for $u,v \in [1-\epsilon, 1+\epsilon]$, $q, q^\prime \in [\a - \epsilon, \a + \epsilon]$, and $g,f \in \mathcal{F}_f$, such that $\bfQ^{(q)} \in \tilde{\ell}^q$, is the $q-$cluster process of the series. 
We can replace $h$ by $e$ or $e$ by $h$ in the previous equalities  to obtain the full covariance structure of this Gaussian process.
\par 
In the following it will be useful to consider the totally bounded semi-metric $\rho : \mathcal{T} \to \mathcal{T} $ defined as
\beao 
\lefteqn{\rho^2\big( (u,q,g), (v, q',g') \big)}
 \\
&=& \Var\big( \mathbb{G}(g(\cdot/u),q) -
 \mathbb{G}(g^\prime(\cdot/v),q') \big),
\eeao 
for $q,q' \in [\a - \epsilon, \a + \epsilon]$ and $u,v \in [1-\epsilon, 1+\epsilon]$, and $g,g' \in \mathcal F$.
Fixing $g\in \mathcal F$ the envelope of the class $\mathcal{T}_{g}$ is 
\beao 
{\bfG}_{g}(\bfx) &:=& \sup_{(u,q) \in [1-\epsilon,1+\epsilon] \times[\a-\epsilon,\a+ \epsilon]}
\left| 
g(\bfx/u)\1(\|\bfx\|_q > u)
\right|\\
&= &  g(\bfx/(1-\epsilon)) \1(\|\bfx\|_{\alpha-\epsilon} > (1-\epsilon)),
\eeao
and note that
\beao\label{eq:finite:envelop} 
\E[(\mathbb{G}(\bfG_g,q))^2] < \infty\,,
\eeao
for every $q\in [\a-\epsilon, \a+\epsilon]$ under  Condition {\bf L}.

\par 
To show the uniform central limit theorem of 
the stochastic process \eqref{eq:sequence:st:processes:0}
with paths on $(\mathcal{T},\rho)$,
we plan to apply Theorem 18.14 in \cite{vandervaart:2000}.
We first show asymptotic equicontinuity of the family $\mathcal{T}$,
and finally we conclude by computing the finite-dimensional limits. 
\par
More precisely, we follow the plan below:
    \begin{enumerate}
        \item[(1)] We start by showing that, asymptotically, 
        we can replace the row-wise triangular array $(\mathcal B_t)_{1\le t \le m_n}$ by an array $(\mathcal B^*_t)_{1\le t \le m_n}$ as in \eqref{eq:notation:blocks}. 
          This will allow us to apply Theorem C.4.5 in \cite{kulik:soulier:2020} yielding  asymptotic equicontinuity on $(\mathcal{T},\rho)$,
        \item[(2)] we then compute the covariance structure of the block estimators \eqref{eq:oracle:cluster} and \eqref{eq:hill:oracle} yielding the finite-dimensional limits of the process $\wt g^{\bfQ}$ properly centered and normalized,
        \item[(3)] finally we establish a control on the complexity of the family $\mathcal{F}$ defined in \eqref{eq:t}.
    \end{enumerate}
Bearing this in mind, 
Section~\ref{sec:coupling} contains preliminaries on coupling theory and yields (1),  Section~\ref{sec:cov:stru} studies covariances structures to tackle (2), and
Section~\ref{sec:intro:entropy} discusses the uniform entropy theory to assess (3).
The proof of Theorem~\ref{thm:main} is deferred to Appendix~\ref{proof:thm:main}. 
The proofs of preliminary results are deferred to Appendix~\ref{sec:appendix:C}.

\subsection{Coupling theory}\label{sec:coupling}

The following Proposition tackles part number (1) from the proof sketch in Section~\ref{sec:proof:sketch}, and its proof is deferred to Section~\ref{subs:sec:proof:prop}.

\begin{proposition}\label{prop:coupling}
Consider a cluster functional $g \in \mathcal{G}_+(\tilde{\ell}^p)$ satisfying ${\bf L}$, and assume {\bf M}, ${\bf MX}_\beta$ are satisfied and 
 $m_n \beta_{b_n} \to 0$, as $ n \to \infty$. 
Then, there exists an array $(\mathcal{B}_t^*)_{1\le t \le m_n}$ as in \eqref{eq:notation:blocks} such that, for all $\delta > 0$,
\beao 
 \lim_{n \to \infty} \P\Big( \sup_{ \substack{ u \in [1-\epsilon, 1+ \epsilon],\\
q \in [\alpha- \epsilon, \alpha + \epsilon] }} \sqrt{k_n}\, | \widetilde g^{\bfQ}(u,q) - \widetilde g^\bfQ_{,*}(u,q)| > \delta \Big)&=&0 \,, \quad g \in \mathcal{F}_f, \\
\lim_{n \to \infty} \P\Big( \sup_{ u \in [1-\epsilon, 1+ \epsilon]} \sqrt{k^\prime_n}\, | \widetilde g^{\bfQ}(u) - \widetilde g^\bfQ_{,*}(u)| > \delta \Big)&=&0   \, , \quad g \in \mathcal{F}_h.
\eeao
\end{proposition}

Indeed, since the family $\mathcal F$ is indexed by only a finite number of functions, 
it is enough to check Proposition~\ref{prop:coupling} separately, for each $g \in \mathcal{F}$ (see \eqref{eq:F}).

\subsection{Covariance structure}\label{sec:cov:stru}

The following Lemma holds and we defer its proof to Section~\ref{sec:lem:covariance:1}.
\begin{lemma}\label{lem:covariance:1}
Consider functions $g,h :\tilde{\ell}^p \to \mathbb{R}$ satisfying $\bf L$. Assume that there exists a sequences $(\ell_n),$ satisfying $\ell_n \to  \infty$, and 
    \[\lim_{n \to \infty} m_n \beta_{\ell_n}/k_n \quad  = \quad  \lim_{n \to  \infty} \ell_n/b_n \quad = \quad  0.\]
Then, the relation below holds
\beam\label{eq:cov:zero:1}\label{eq:finite:covariance}
    \lim_{n \to  \infty } \frac{m_n}{k_n} \Cov(\, g(x_{b_n}^{-1}\mathcal{B}_1),f(x_{b_n}^{-1}\mathcal{B}_{2} ) \, ) &=& 0.
\eeam
\end{lemma}

The covariance structure of the deterministic threshold estimators is determined by the Lemma below whose proof is postponed to Section~\ref{proof:lem:iid:blocks}. 
This Lemma studies the step (2) of the proof plan in Section~\ref{sec:proof:sketch}.
\begin{lemma}\label{proof:cov:iid:blocks}
Consider $g, f :\tilde{\ell}^p \to \mathbb{R}$ to be cluster functional satisfying {\bf L},  and let $h , e : \tilde{\ell}^\infty \to \mathbb{R}$ be the cluster functional for the Hill estimator and the exceedances estimator in \eqref{eq:func:hill} and \eqref{eq:func:exceedances}. 
Then, if $p = \alpha$, for $q,q^\prime \in [\alpha - \epsilon, \alpha + \epsilon]$ and $u,v \in [1-\epsilon, 1+\epsilon]$,
\beam
\lefteqn{\lim_{n\to\infty} k_n \, \cov\big( \widetilde g^\bfQ_{,*}(v,q),   \widetilde f^\bfQ_{,*}(u,q^\prime)\big)}\nonumber \\
&=& \Cov\big( \mathbb{G}(g(\cdot/v),q) ,\mathbb{G}(f(\cdot/u), {q^\prime } ) \big), \label{eq:cov:iid:2}  
\eeam 
Moreover,  if $k_n/k_n^\prime \to \kappa > 0$, then
\beam\label{eq:cov:iid} 
\lefteqn{
\lim_{n\to\infty}   k_n \, \cov\big( \widetilde g^\bfQ_{,*}(v,q),   \widetilde h^\bfQ_{,*}(u)\big)} \nonumber\\ 
    &=& \Cov\big( \mathbb{G}(g(\cdot/v),q) ,\mathbb{G}(h(\cdot/u), \alpha ) \big), \\
\lefteqn{ \lim_{n\to\infty}  k_n \, \cov\big( \widetilde h^\bfQ_{,*}(v),   \widetilde e^\bfQ_{,*}(u)\big)}\nonumber\\
&=&   \Cov\big( \mathbb{G}(h(\cdot/v),\alpha) ,\mathbb{G}(e(\cdot/u), \alpha ) \big) \label{eq:cov:iid:3} 
\eeam
with the notation in \eqref{eq:def:limit:gauss}.
The result remains true for any $p\neq \a $ by fixing $q=q'=p$, and replacing $h$ by $e$  or $e$ by $h$. 
\end{lemma}

\subsection{Uniform entropy theory}\label{sec:intro:entropy}
In this section we discuss how we assess the step (3) of the proof plan (see Section~\ref{sec:proof:sketch} for details).
Theorem~\ref{thm:main} states asymptotic normality of $\ell^{\widehat \a}-$block estimators in \eqref{eq:estimator:cluster}.
Our proof considers a family of block estimators indexed by $q$, for $q$ in a neighborhood of $\a$,
and relies on Lemma~\ref{lem:pnorm:VCclass} below showing this family has low complexity in terms of entropy numbers.
We review the classical results of the theory of Vapnik-Cervonenkis below to measure the complexity of classes of functions. 
We refer to  \cite[Section 2.6]{vandervaart:wellner:1996} for a detailed treatment.
\par 
Let $\mathcal X$ be a measurable space and let $\mathcal V$ be a collection of sets from this space. 
The VC--dimension of  $\mathcal{V}$ is the smallest number $s$ such that for every set containing $s$ elements,
we can find a subset that is not picked out by the class $\mathcal{V}$. 
We say that $\mathcal{V}$ is a VC--class if its VC--dimension is finite.
A VC-class of functions $\mathcal{F}$ is such that the collections of all the subgraphs $
\{(\bfx,u)
: 
\phi(\bfx) >u\}$  of real-valued functions $\phi \in \mathcal{F}$ is a VC-class.  
The entropy number of a VC-class has a  polynomial expression on the VC-dimension  (see Theorem 2.6.7. in \cite{vandervaart:wellner:1996}).
Moreover, given a VC-class of functions $\mathcal F$, the VC-Hull of $\mathcal F$ is the collection of functions  $\mathcal{G}$ such that for every $g \in \mathcal{G}$ there exists a symmetric convex combination $f_m = \sum_{i=1}^m \alpha_i f_i$, with $\sum_{i=1}^m|\alpha_i|\le 1$, $f_i \in \mathcal{F}$, such that $g$ is the pointwise limit of the sequence $(f_m)_{m\in\mathbb{N}}$. 
Moreover, 
by Corollary 2.6.12 in \cite{vandervaart:wellner:1996} the entropy number of a VC-hull also has a polynomial expression on the VC-dimension of the underlying VC-class.
\par 
It is often easier to check that $\mathcal{F}$ is a VC-major class of functions $\mathcal{F}$, i.e., that $\{\bfx
: 
\phi(\bfx) > u\}$, for every $u\in \R$, is a VC-class.
We can construct new VC-major classes using classical operations: addition, products.
If $\mathcal{F}$ is VC-major, then the class of function $h \circ \phi$, with $h$ ranging over the monotone functions $h:\mathbb{R} \to \mathbb{R}$, with $\phi \in \mathcal{F}$, is VC-major. 
This is Lemma 2.6.19 in \cite{vandervaart:wellner:1996}.
One example are functions $\bfx \mapsto f(\bfx/u)\1(\phi( \bfx) > u)$, where $u \mapsto f(\bfx/u)$ is a non-increasing function.
Moreover, a bounded VC-major class satisfies the uniform entropy condition by Lemma 2.6.13 in \cite{vandervaart:wellner:1996}. 
\par 

In particular, we show that $\ell^q-$norms constitute a VC-major class. This is the purpose of the next Lemma whose proof is postponed to Section~\ref{sec:proof:lem:pnorm:VCclass}.

\begin{lemma}\label{lem:pnorm:VCclass}
    Consider the class containing all sets of the form $\{ \bfx \in \tilde{\ell}^{q_0} : \| \bfx \|_q > u\},$ where $q \in (q_0,q_0^\prime)$, and $u \in \R$. 
    Then, this is a VC-class of dimension 3.
    {
    This implies that the classes of functions \beao\mathcal{N} &:=& \{ \| \cdot \|_q: \tilde{\ell}^{q_0}  \mapsto [0,\infty), q \in (q_0,q_0^\prime)\},\eeao are VC-major for every $q_0^\prime>q_0>0$.}
\end{lemma}

\section{Proof of Theorem~\ref{thm:main}}\label{sec:lem:covariance}\label{proof:thm:main}\label{sec:proof:thm:clt}



\par 
Weak convergence of
the deterministic threshold estimators defined in \eqref{eq:t} towards the Gaussian limit in \eqref{eq:def:limit:gauss} follows by an application of Theorem 18.14 in \cite{vandervaart:2000}.
 In the following, we focus on showing the asymptotic equicontinuity and the convergence of finite-dimensional distributions. We assume $k_n/k_n^\prime \to \kappa$, for $\kappa \ge 0$, or equivalently $x_{b_n}^\prime/x_{b_n} \to (\kappa/ c(p))^{1/\alpha}$, as $n \to \infty$. Indeed, if we denote $|\bfX|'$s distribution by $F_{|\bfX|}$ and its left inverse by $F_{|\bfX|}^{\leftarrow}$, then the relations \eqref{eq:k:2} and \eqref{eq:k:app} imply
\beam \label{eq:rationx}
 \frac{x^\prime_{b_n}}{x_{b_n}}  &\sim & \frac{F^{\leftarrow}_{|\bfX|}\big(1-\frac{k^\prime_n}{n} \big) }{F^{\leftarrow}_{|\bfX|} \big(1-\frac{k_n}{n c(p) } \big)}  \; \sim \; \left( \frac{k_n^\prime c(p) }{k_n}\right)^{-1/\alpha} \;\to\; (\kappa/c(p) )^{1/\alpha},
\eeam 
as $n \to \infty$, such that the last relation follows by regular variation of $|\bfX|$ and \cite[Proposition 2.6 (v)]{resnick:2007}.
\par


\par

\subsection{Asymptotic equicontinuity}\label{subsec:uniform:gaussian}

In view of Proposition~\ref{prop:coupling},
 it suffices to establish asymptotic equicontinuity of block estimators 
 on the sequence $(\mathcal{B}^*_t)_{1\le t \le m_n}$ defined in \eqref{eq:notation:blocks}.
 Moreover, consider the restriction of the blocks estimator to odd indices:
 \beam \label{eq:odd}  
   \widetilde{g}_{\text{odd},*}^\bfQ(u,q) 
 &=& \frac{1}{k_n }  \sum_{\substack{t=1 \\ t \text{ odd}} }^{ m_n } g(\mathcal{B}^*_{t}/(u\, x_{b_n}) )\1( \|\mathcal{B}^*_{t}\|_q > u\, x_{b_n} )\,,  \quad (u,q,g) \in \mathcal{T}.
 \eeam 
Note it suffices to verify the asymptotic equicontinuity of \eqref{eq:odd} (the result for even indices will follow from stationarity of the series).
Moreover, the block estimator in \eqref{eq:odd} is written as the sum of independent random variables, then to show asymptotic equicontinuity of \eqref{eq:odd} we can rely on Theorem C.4.5 in \cite{kulik:soulier:2020}. 
 \par 
   To simplify notation, we fix any $g\in \mathcal F_f$. The same arguments will apply to $g\in \mathcal F_h$ replacing $x_{b_n}$ by $(c(p)/\kappa)^{1/\a}x_{b'_n}$ and using \eqref{eq:rationx} when $\kappa>0$. We denote 
\beao 
g(\mathcal{B}^*_{t}/(u \, x_{b_n}))\1(\|\mathcal{B}^*_{t}\|_q > (u \, x_{b_n})) &=& g(\mathcal{B}^*_{t}/{x_{b_n}})(u,q),
\eeao 
and
 we consider the family $\mathcal{T}_{g,\ast}$ of stochastic processes defined by the independent-block estimators in \eqref{eq:oracle:iid}, namely,
\beao 
 \mathcal{T}_{g,*} &=& \{ \widetilde{g}_{ {\text{odd}},*}^\bfQ(u,q)\}_{ \{u \in [1-\epsilon,1+\epsilon], q \in [\a - \epsilon, \a + \epsilon] \} },
\eeao 
where $ \widetilde{g}_{,*}^\bfQ(u,q) $ is as in \eqref{eq:odd}. 
Recall this process is indexed by $[ 1-\epsilon,1+\epsilon]\times[ 1-\epsilon,1+\epsilon] $.
Define the random metric $d_n(\cdot, \cdot)$ on this family by
\beam \label{eq:random:metric}
\lefteqn{
\big(
d_n((u,q), (v,q^\prime))\big)^2}\nonumber \\
&=& \frac{1}{ {k_n} } \sum_{{ \substack{t=1 \\ t \text{ odd} }} }^{m_n}\left( g(\mathcal{B}^*_{t}/{x_{b_n}})(u,q) - g(\mathcal{B}^*_{t}/{x_{b_n}})(v,q^\prime) \right)^2.
\eeam 
In the remaining of the proof, 
we verify the sequence of processes $\mathcal{T}_{g,*}$ satisfies the Lindeberg condition (i), 
continuity condition (ii), and entropy  condition (iii) 
from Theorem C.4.5 in \cite{kulik:soulier:2020} hold. 
\par 
In the following let $u_0 = 1- \epsilon < 1 < 1+ \epsilon = s_0$ and $q_0 = \a - \epsilon < \a < \a + \epsilon = q_0^\prime$.
\subsubsection{Lindeberg condition (i)}
Since $u \mapsto g(\bfx/u)$ is a non-increasing function,
then we it suffices to verify, for every $\eta > 0$, $q\in[q_0,q^\prime_0], u \in [u_0,s_0]$,
    \beao 
      \lefteqn{I} &=& \frac{m_n}{k_n}\E\left[ 
      \left( g(\mathcal{B}_{1}/ x_{b_n})(u,q)\right)^2 
 \, \1\big( g(\mathcal{B}_{1}/ x_{b_n})(u,q) > \sqrt{\eta k_n} \big) 
   \right] \\
      &\to& 0, \quad n \to \infty.
    \eeao
Indeed, we have
\beao 
I&\le& \frac{m_n}{k_n}
\E\left[ 
      \left(g(\mathcal{B}_{1}/ x_{b_n})(u,q)\right)^{2+\delta }\right]^{\frac{2}{2+\delta}}
      \P\left(  g(\mathcal{B}_{1}/  x_{b_n})(u,q) > \sqrt{\eta k_n}\right)^{\frac{\delta}{2+\delta}} \\
&\le&  (\eta\,k_n)^{-{\delta}/{2}} \frac{m_n}{k_n}  \E\left[ 
      \left(g(\mathcal{B}_{1}/ x_{b_n})(u,q)\right)^{2+\delta}\right].\\
\eeao
where $\delta$ here follows the notation in \eqref{eq:cond:lind:1}. Then, appealing to assumption $\bf L$,
we deduce $I \to 0$, as $n \to \infty$. 
\par 
\subsubsection{Continuity condition (ii)}
Let $s > u$, $q , q^\prime$, 
and denote 
\beam\label{eq:cov:notation}
c((u,q),(s,q^\prime)) : = \Cov\big( \mathbb{G}(g(\cdot/u),q) ,\mathbb{G}(g(\cdot/v),  {q^\prime }) \big),
\eeam 
and $c(u,q) = c((u,q),(u,q))$, and $c(q)$ is as in \eqref{eq:constant:cp:1:tcl}.
Then, applying
Lemma~\ref{proof:cov:iid:blocks},
\beao 
\lefteqn{
{2 \, } \E\big[\big(
d_n((u,q), (s,q^\prime))\big)^2\big]
}\\
&=&
\frac{m_n}{k_n}\E\left[ \left(g(\mathcal B_{1}/x_{b_n})(u,q) - g(\mathcal B_{1}/x_{b_n})(s,q^\prime) \right)^2\right]
\\
 &\xrightarrow[]{}&  c(u,q) + c(s,q^\prime) - 2c( (u,q), (s,q^\prime))\\
 &=& 
 u^{-\a}c(q)(g^2)^{\bfQ}(q) +  s^{-\a}c(q^\prime)(g^2)^{\bfQ}(q^\prime)\\
 && - 2 u^{-\a} c(q)\E[g(Y\bfQ^{(q)})g(Y\bfQ^{(q)} u/v) \1(\|Y\bfQ^{(q)}u/v\|_{q^\prime} > 1) ]. 
\eeao 
We now use the fact that $ v \mapsto g(\cdot/v)$ is a non-increasing function, for $u > v$,
\beam\label{eq:bound:covariance:uq} 
\lefteqn{ c(u,q) + c(s,q^\prime) - 2c( (u,q), (s,q^\prime))} \nonumber \\
&\le&   u^{-\a}c(q)(g^2)^{\bfQ}(q) +  s^{-\a}c(q^\prime)(g^2)^{\bfQ}(q^\prime) \nonumber  \\
 && - 2 u^{-\a} c(q)\E[g(Y\bfQ^{(q)})^2 \1(\|Y\bfQ^{(q)}\|_{q^\prime} > 1) ] \nonumber  \\
 &=& s^{-\a}c(q^\prime)(g^2)^{\bfQ}(q^\prime) - u^{-\a}(g^2)^{\bfQ}(q)  \nonumber \\
 && + 2 u^{-\a} c(q)\E[g(Y\bfQ^{(q)})^2 \1(\|Y\bfQ^{(q)}\|_{q^\prime} \le 1) ].  
\eeam 
We now focus on the last term. Notice that for $q\ge q^\prime$, the last term equals zero. 
We consider the case $q < q^\prime $,
\beam\label{eq:bound:metric:continuity} 
\lefteqn{ 
   \E[g(Y\bfQ^{(q)})^2 \1(\|Y\bfQ^{(q)}\|_{q^\prime} \le 1) ] } \nonumber \\
   &\le& \E[g(Y\bfQ^{(q)})^{2+\delta}]^{2/(2+\delta)} \P(\|Y\bfQ^{(p)}\|_{q^\prime} \le 1)^{\delta/(2+\delta)}.
\eeam 
Furthermore, for $q < q^\prime $,
\beao 
\lefteqn{ \P(Y \|\bfQ^{(q)}\|_{q^\prime} \le 1) }\\
    &\le &  \P( Y^q ( \|\bfQ^{(q)}\|_q^{q} - (q-q^\prime)\mbox{$\sum_{t\in \mathbb{Z}}$} |\bfQ^{(q)}_t|^{q^\prime } \log1/ | \bfQ^{(q)}_t|)  \le 1) \\
    &=&  \P( Y^q( 1 - (q-q^\prime)\mbox{$\sum_{t\in \mathbb{Z}}$} |\bfQ^{(q)}_t|^q \log1/ | \bfQ^{(q)}_t| ) \le 1) \\ 
    &\le& 1 - \E[ (1 - (q - q^\prime)\mbox{$\sum_{t\in \mathbb{Z}}$} |\bfQ^{(q)}_t|^{q^\prime } \log1/ | \bfQ^{(q)}_t|))^{\alpha/q} ] \\
    &\le& (q-q^\prime)\alpha q^{-1} \, \E[ \mbox{$\sum_{t\in \mathbb{Z}}$} |\bfQ^{(q)}_t|^{q^\prime } \log1/ | \bfQ^{(q)}_t| ].
\eeao 
Therefore, as $q < q^\prime$, we obtain
\beao 
\P(Y \|\bfQ^{(q)}\|_{q^\prime} \le 1) &\le& (q-q^\prime)\alpha q^{-1} \, \E\Big[\sum_{t\in \mathbb{Z}} |\bfQ^{(q)}_t|^{q} \log1/ | \bfQ^{(q)}_t| \Big].
\eeao 
Then, notice that by the change-of-norms in Equation~\eqref{eq:change:of:norms}
\beao
\E\Big[ \sum_{t\in \mathbb{Z}} |\bfQ^{(q)}_t|^{q} \log 1/ | \bfQ^{(q)}_t|  \Big] 
&=&  c(q)^{-1}\E\Big[ 
\frac{\|\bfQ\|_q^{\a}}{\|\bfQ\|_q^q} \sum_{t\in \mathbb{Z}} |\bfQ_t|^{q} \log \frac{\|\bfQ\|_q}{ | \bfQ_t|}  \Big] \\
&\le& (\epsilon c(q))^{-1} \E[ \| \bfQ\|_{\a - \epsilon}^{\a + \epsilon} ],
\eeao
and the last relation follows by the monotonicity of $\ell^p-$norms and the fact that
for every chosen $0<\eta <1$ we have 
\beam \label{eq:bound:logarithm}
\log (1/x)\; =\; 1/\eta\log(1/x^\eta) \; \le\;  1/\eta x^{-\eta}\,,\qquad 0<x\le 1\,,
\eeam 
Hence, appealing to conditions {\bf M} and {\bf L} the term in \eqref{eq:bound:metric:continuity} is bounded by constant $C < \infty$.
We can now conclude the following bound for \eqref{eq:bound:covariance:uq} 
\beao 
\lefteqn{ c(u,q) + c(s,q^\prime) - 2c( (u,q), (s,q^\prime))}\\
&\le& s^{-\a}c(q^\prime)(g^2)^{\bfQ}(q^\prime) - u^{-\a}c(q)(g^2)^{\bfQ}(q)  + 2 u^{-\a}c(q)(q-q^\prime)\alpha q^{-1}C.
\eeao 
Finally, recall from Proposition~\ref{prop:existence:cluster:process} that
\beao 
c(q) &=& \E[ \|\bfQ\|_q^\alpha ]   \; = \; (\E[ 1/ \|\bfQ^{(q)}\|_{\alpha}^{\alpha}])^{-1},
\eeao 
and then it is easy to see by monotone convergence that $q \mapsto c(q)$ is a continuous function at $\a$. Finally, this implies
\beao 
\lim_{\eta  \downarrow 0} \limsup_{n \to  \infty} \sup_{ 
\substack{ u,s \in [u_0,s_0], \\ 
q,q^\prime \in [q_0, q^\prime_0], \\
\rho((u,q),(s,q^\prime)) < \eta } }  \E\big[\big(
d_n((u,q), (s,q^\prime))\big)^2\big] &=& 0.
\eeao 
From this we conclude that (ii) holds.

\subsubsection{Entropy condition (iii)}
Consider $g\in \mathcal{F}_f$ as a cluster functionals: $\bfx \mapsto g(\bfx)$ for $\bfx \in\tilde{\ell}^{q_0}$, $q_0>0$ (see \eqref{eq:1:functional}, \eqref{eq:func:hill}, \eqref{eq:func:exceedances}).
Denote $\mathcal{T}_{g}$ the class of functions 
\beao 
(\bfx,u,q) &\mapsto&  g(\bfx/u)\1(\|\bfx\|_q > u),
\eeao 
indexed by $[u_0,s_0] \times [q_0,q^\prime_0]$ for fixed $g\in \mathcal F$. 
It is sufficient to show that the class $\mathcal{T}_{g}$ satisfies the entropy condition in (iii) with respect to the random metric introduced in 
\eqref{eq:random:metric}. Indeed $\mathcal{T}_{g,\star}$ will also satisfy the entropy condition considering $\mathcal B_1/x_{b_n}$ as a sequence in $\tilde{\ell}^{q_0}$.
Moreover, notice we can apply Lemma C.4.8 in \cite{kulik:soulier:2020}. 
Indeed, we can verify condition C.4.8 using Lemma~\ref{proof:cov:iid:blocks}, as 
\beao 
\lefteqn{\frac{m_n}{k_n}\E\left[ \left(g(\mathcal B_{1}/x_{b_n})(u,q) - g(\mathcal B_{1}/x_{b_n})(s,q^\prime) \right)^2\right]}
\\
 &\xrightarrow[]{}&  c(u,q) + c(s,q^\prime) - 2c( (u,q), (s,q^\prime)) < \infty, \qquad n \to \infty. 
 \eeao 
This means that
it is enough to check that $\mathcal{T}_{g}$ is a VC-hull class, as introduced in Section~\ref{sec:intro:entropy}, and then apply Corollary 2.6.12 in \cite{vandervaart:wellner:1996} giving a satisfactory bound on the entropy. 
In the following we treat separately the case $g$ equal to $f_\a $ or $1$ and the case $g$ equal to $h$ or $e$. 

\subsubsection*{Case $g = f_\a$ and $g = 1$}
Consider the class of functions $ \|\bfx\|_q: \bfx \mapsto \|\bfx\|_q $, for elements $\bfx \in\tilde{\ell}^{q_0}$, and $q \in [q_0, q^\prime_0]$. 
By Lemma~\ref{lem:pnorm:VCclass}, this class of functions is a VC-major class, as the sets $\{ \bfx\in\tilde{\ell}^{q_0} : \|\bfx\|_q > u\}$, for $q \in [q_0, q^\prime_0]$, and $u > 0$, forms a VC-class of dimension $3$. 
Finally, applying Lemma 2.6.19 in \cite{vandervaart:wellner:1996} for the monotone functions $\psi_u : \mathbb{R} \to \mathbb{R}$ defined by:
\beao\label{eq:class} 
u &\mapsto&  g(\bfx/u) \1(\|\bfx\|_{q}> u),
\eeao 
indexed by $u$, we see that $\mathcal{T}_{g}$ is a VC-major. Finally, Lemma 2.6.13 in \cite{vandervaart:wellner:1996} states that bounded VC-major classes are VC-hull classes and this yields the desired result.

\par 
\subsubsection*{Case $g = h, e$}
This case has been studied previously, for example, we can borrow the results in \cite{kulik:soulier:2020}. 
Here by an applications of Corollary C.4.20 in \cite{kulik:soulier:2020} we conclude the entropy condition is satisfied since these are linearly ordered functionals.

To sum up, we have verified the sequence of processes $\mathcal{T}_{g,*}$ satisfies the Lindeberg condition $(i)$, the continuity condition $(ii)$, and the entropy condition $(iii)$ from Theorem C.4.5. in \cite{kulik:soulier:2020}. 
Therefore, we conclude the asymptotic continuity of the processes indexed by $\mathcal{T}$.


\subsection{Uniform central limit theorem}\label{prop:finite:dimensional}

Consider a function $f_\a:\tilde{\ell}^\a \to \mathbb{R}$ satisfying $\bf L$. {It is enough to check the uniform central limit theorem on the processes indexed by $\mathcal F_f$ when $\kappa = 0$. Therefore we focus on the cases where $\kappa>0$ and the family $\mathcal F$ as in \eqref{eq:F}.}
Note that by the assumptions on Theorem~\ref{thm:main} we have that Proposition~\ref{prop:coupling}, and Lemma~\ref{proof:cov:iid:blocks} hold. 
Asymptotic equicontinuity of the family
\beam \label{eq:sequence:st:processes:2}
\big\{\sqrt{k_n} \left( \widetilde{g}^\bfQ(u,q) - c(q)u^{-\a}g^{\bfQ}(q) \right) \big\}_{(u,q,g) \in \mathcal T}
 \eeam  
holds by the calculations from Section~\ref{subsec:uniform:gaussian}, then it remains to verify the
asymptotic normality of the finite-dimensional parts of the family in \eqref{eq:sequence:st:processes:2}. 
Applying the Wold device it is enough to check that every linear combination of deterministic threshold 
estimators in \eqref{eq:oracle:cluster} is asymptotically normal distributed.
{Similarly as before, it
suffices to verify this condition with respect to the block estimator~\eqref{eq:oracle:iid} 
by an application of Proposition~\ref{prop:coupling}. 
}
Moreover, note any such a linear combination $g$ is again a cluster functional satisfying ${\bf L}$ { because $\kappa>0$ and 
\[
\frac{\mathbb{P}(\| {\bf X}_{[1,n]}  \|_p  > x_n ) }
{\mathbb{P}(\| {\bf X}_{[1,n]}  \|_p  > x^\prime_n ) }
\sim \Big(  \frac{x_n}{x_n'} \Big)^{-\alpha} \to \frac{\kappa}{c(p)}\,, \qquad \nto\,.
\]
}We thus apply Proposition~\ref{prop:coupling} to $g$. 
\par 
Hence, for such a linear combination $g$ denote by $\widetilde{g}^\bfQ_{,*}$ its corresponding block estimator which will be a linear combination of block estimators as in~\eqref{eq:oracle:iid}, and denote $g^\bfQ(p)$ its cluster-statistic with the notation in \eqref{eq:clust:statistic}.  Then, it is enough to check that the real-valued variable:
\beam \label{eq:sequence:st:processes:3}
\sqrt{k_n}  \left( \widetilde{g}^\bfQ_{,*} - g^{\bfQ}(p)\right),
 \eeam  
 admits a Gaussian limit, as $n \to \infty$.
Moreover, note we can replace $g^\bfQ(p)$ by the expectation of $\widetilde{g}_{,*}^\bfQ$ in \eqref{eq:sequence:st:processes:3} thanks to the bias assumptions ${\bf B}_\alpha(k_n)$ and ${\bf B}(k^\prime_n)$. 
{
In addition, note that if \eqref{eq:sequence:st:processes:3} is the sum of independent random variables, we can apply the Lindeberg central limit theorem for triangular arrays \cite[Section 18]{billingsley:2013}.
We apply an extension of Lindeberg's theorem for weakly mixing triangular arrays provided in Theorem 4.4 \cite{rio:2017} noting the triangular array $(\mathcal{B}^*_t)$ is $1$--dependent. 
Condition $(a)$  in there is granted by the variance calculations in Lemma~\ref{proof:cov:iid:blocks}.
To verify $(b)$ it suffices to check Equation~(4.14) therein, which in our setting is granted by condition {\bf L}.
}
Finally, Lemma~\ref{proof:cov:iid:blocks}  allow us to compute the asymptotic variance of \eqref{eq:sequence:st:processes:3} recentering by its expectation, and this is enough for the weak convergence of \eqref{eq:sequence:st:processes:3} to a normal distribution as $n \to \infty$.
Then, we conclude the uniform asymptotic normality of the family in \eqref{eq:sequence:st:processes:2} towards the Gaussian process $\mathbb G$ defined in \eqref{eq:def:limit:gauss}.

\black
\subsection{Variance calculations}\label{sec:variance:calculation}

In the case where \eqref{eq:sequence:st:processes:2} holds,
applications of Vervaat's lemma allow us to compute the variance of the Gaussian limit of the Hill and $\ell^{\widehat \a}-$block estimators. 
\begin{lemma}[Asymptotics of the Hill estimator $1/\widehat{\alpha}$]\label{lem:Hill:limit}
    Assume the conditions of Theorem~\ref{thm:main} are satisfied and $k_n/k_n^\prime \to \kappa$, for $\kappa > 0$. 
    Then $\|\mathcal{B}_t\|_{\widehat \alpha, (k+1)}/x_{b_n} \xrightarrow[]{\P} 1$, and 
    \beam \label{eq:normality:hill}
    \sqrt{k_n}\, ( \widehat{\a} - \a ) & \xrightarrow[]{d} &  \a \mathbb G(\a h(\cdot/1) - e(\cdot/1), \alpha ), \qquad n \to \infty.
    \eeam 
\end{lemma}


\begin{lemma}[Asymptotics of the random threshold estimator]\label{lem:random:threshold}
     Assume the conditions of Theorem~\ref{thm:main} are satisfied and $k_n/k_n^\prime \to \kappa$, for $\kappa \ge   0$. 
\beao 
 \sqrt{k_n}\, \big(\,\widehat{f_{\a}^\bfQ}(\widehat \a) - f^\bfQ_\a\,\big)
 &\xrightarrow[]{d}&  f_\a^\bfQ \mathbb G( \, f_\a(\cdot/1)/f_\a^\bfQ - 1(\cdot/1)\,,\alpha), \qquad n \to \infty.
\eeao
\end{lemma}
The proofs of these lemma are deferred to Section \ref{sec:appendix:E}.



We now focus on establishing the asymptotic variance in  \eqref{eq:central:limit}.
Recall that when $f$ depends on $\a$, i.e., $f = f_\a$, we impose a smoothness assumption ${\bf S}$ on the function $q \mapsto f_q$. More precisely, we assume 
\beao 
f_q(\bfx) &=& f_\a(\bfx) + (q-\alpha)f^\prime_\a(\bfx) + \frac{1}{2}(q-\a)^2 R_\a(\bfx),
\eeao 
where $f^\prime_\a$ is the derivative of $f_a$ with respect to $a$, i.e., $f^\prime_\a = \tfrac{\partial f_q}{ \partial q} \big|_{q = \a }$, and 
\beao
R_\a(\bfx) &\le&  \sup_{q \in (\a - \epsilon, \a + \epsilon)} \Big|\frac{\partial^2 f_q}{ \partial q^2} (\bfx) \Big|.  
\eeao 
{Then, we see by Lemma~\ref{lem:con} that $f^\prime_\alpha$, $\sup_{q \in (\a-\epsilon, \a + \epsilon)} | \partial^2 f_q/\partial q^2| $ admit consistent $\ell^{\widehat \alpha}-$cluster estimates. Moreover we have
\beam \nonumber
 \sqrt{k_n}\, \big(\,\widehat f^\bfQ_{\widehat \a}(\widehat \a) - f^\bfQ_\a \,\big) 
&=& \sqrt{k_n} \big(\,\widehat f_{\a}^\bfQ (\widehat \a) - f_\a^\bfQ \big) + \frac{\sqrt{k_n}}{\sqrt{k_n^\prime}} \sqrt{k_n^\prime}  \big( \alpha - \widehat \alpha \big)  \widehat  f_\a^{{\prime}^\bfQ}  (\widehat \a)   \\\label{eq:var:decomp}
&& +  \frac{\sqrt{k_n}}{\sqrt{k_n^\prime}} \sqrt{k_n^\prime} \big( \alpha - \widehat \alpha \big)^2 \widehat R_\a^\bfQ(\widehat \a).  
\eeam 
as $n \to \infty$, where $k_n^\prime$ is the tuning parameter for the Hill estimator whereas $k_n$ is used to tune the extremal cluster estimator. 
\par 
Recall  $k_n/k_n^\prime \to \kappa \ge 0$. Then, we consider separately the cases $\kappa = 0$ and $\kappa > 0$.  
{
We start with the case $\kappa > 0$. Applying Lemma~\ref{lem:Hill:limit} we obtain 
\beao 
\sqrt{k_n}(\alpha - \widehat \alpha) &\xrightarrow[]{d}& \a \mathbb{G}\big(\a h(\cdot/1) - e(\cdot/1) , \alpha \big), 
\eeao 
as $n \to \infty$. Moreover, an application of Lemma~\ref{lem:random:threshold}} yields 
\beao 
 \sqrt{k_n}\, \big(\,\widehat{f_{\a}^\bfQ}(\widehat \a) - f^\bfQ_\a\,\big)
 &\xrightarrow[]{d}&  f_\a^\bfQ \mathbb G( \, f_\a(\cdot/1)/f_\a^\bfQ - 1(\cdot/1)\,,\alpha).
\eeao
Then, we can conclude from Equation~\eqref{eq:var:decomp} that
\beao 
\lefteqn{ \sqrt{k_n}\, \big(\,\widehat f^\bfQ_{\widehat \a}(\widehat \a) - f^\bfQ_\a \,\big)}\\
& \xrightarrow[]{d}& \mathbb{G} \big( (f_\a(\cdot/1) - f_\a^\bfQ 1(\cdot/1)) + (f_\a^{\prime \bfQ})  \a( \a  h(\cdot/1) - e(\cdot/1)) , \alpha \big)\\
&=& \mathcal{N}(0,\sigma^2(\kappa)).
\eeao
The limit variance $\sigma^2(\kappa)$ can be computed from the covariance structure in \eqref{eq:cov:gaussian} and is given explicitly in \eqref{eq:sigma:2}. 
Notice that it depends on the parameter $\kappa > 0$.
More precisely,
\beao
\sigma^2(\kappa) &=&  \var(f_\alpha(Y\bfQ)) +   \kappa  \alpha^2(f^{\prime\bfQ} _\alpha)^2 \sigma_\alpha^2 +  
2 \kappa  \alpha\,f^{\prime\bfQ} _\alpha\, \sigma_{f,\a}(\kappa),
\eeao 
where
\beao 
\sigma_\a^2 &:=& \E[(\a h(Y\bfQ) - e(Y\bfQ))^2],\\
\sigma_{f,\a}(\kappa) &:=& \Var[(f_\alpha(Y\bfQ) - f_\alpha^\bfQ) ( \alpha h(Y\bfQ \kappa^{-1/\alpha}) - e(Y\bfQ \kappa^{-1/\alpha} ) )].
\eeao  
\par 
Furthermore, notice that
if $f_\a(\bfx) = f_\a(\bfx/\|\bfx\|_\a)$, then $\sigma_{f,\a}(\kappa) = 0$, for all $\kappa > 0$. Otherwise, notice that
by Jensen's inequality we have
\beao 
(\sigma_{f,\a}(\kappa))^2 
 &\le&    \, \E[(f_\alpha(Y\bfQ)-f^{\bfQ}_\a)^2] \E[( \alpha h(Y\bfQ\kappa^{-1/\alpha}) - e(Y\bfQ \kappa^{-1/\alpha}) )^2], 
\eeao 
such that $\E[f_\a(Y\bfQ)^2] < \infty$. Then, we focus on the right-hand term in the previous equation.
Then, relying on the properties of $h$ and $e$ as defined in \eqref{eq:func:hill}, \eqref{eq:func:exceedances}, we obtain that for $\kappa < 1$,
\beao 
\lefteqn{
(\sigma_{f,\a}(\kappa))^2/\Var(f_\a(Y \bfQ))}\\
&\le& 
 \E\big[(\a h(Y\bfQ\kappa^{-1/\a})-e(Y\bfQ\kappa^{-1/\a}))^2 \big]\\
 &=& \E\big[(\a h(Y\bfQ\kappa^{-1/\a})-e(Y\bfQ\kappa^{-1/\a}))^2 \1(\| Y\bfQ \kappa^{-1/\alpha}\|_\a > 1) \big]\\
&=& \kappa \int_{0}^\infty \E[(\a h(y\bfQ)-e(y\bfQ))^2 \1( y > \max\{\kappa^{-1/\alpha}, 1\} )]d(-y^{-\a}) \\
& \le & \kappa \sigma_\a^2.
\eeao 
Hence, we conclude that, for $\kappa < 1$,
\beao 
(\sigma_{f,\a}(\kappa))^2 &\le& \kappa \sigma_\a^2 \,\Var(f_\alpha(Y\bfQ)).
\eeao 
In particular this implies $\sigma^2(\kappa) \to \Var(f_\a(Y\bfQ)),$ as $\kappa \to 0$.
\par 
If $\kappa = 0$, then 
under the assumption of Theorem~\ref{thm:main},
similar steps as for the proof provided in Section~\ref{proof:thm:main}, but now restricting the family $\mathcal F$ to $\mathcal F_h$ with the notation in \eqref{eq:t},
allow us to conclude the following
\beao 
\sqrt{k_n^\prime} \big( \alpha - \widehat \alpha \big) &\xrightarrow{d}& \a\,\mathcal{N}\big(0, \E[(\alpha h(Y\bfQ) - e(Y\bfQ))^2] \big),
\eeao 
as $n \to \infty$. Hence,
we see
$\frac{\sqrt{k_n}}{\sqrt{k_n^\prime}} \sqrt{k_n^\prime}  \big( \alpha - \widehat \alpha \big) \xrightarrow[]{\P} 0$, as $n \to \infty$.  Then, Equation~\eqref{eq:var:decomp} together with the previous limit implies
\beao 
\sqrt{k_n}\, \big(\,\widehat f^\bfQ_{\widehat \a}(\widehat \a) - f^\bfQ_\a \,\big) 
&=& \sqrt{k_n} \big(\,\widehat{f_{\a}^\bfQ} (\widehat \a) - f_\a^\bfQ) + o(1)\\
&\xrightarrow[]{ d }& \mathbb{G} \big( (f_\a(\cdot/1) - f_\a^\bfQ 1(\cdot/1)), \alpha \big) \\  
&=& \mathcal{N}\big(0, \Var(f_\a(Y \bfQ))\big),
\eeao 
as $n \to \infty$, where the last limit follows again from  Lemma~\ref{lem:random:threshold}.}
\par 
Overall, this calculations demonstrate the expression of the asymptotic variance in  \eqref{eq:central:limit}, and this concludes the proof of Theorem~\ref{thm:main} for $p = \alpha$.
\par 
{
In all generality, for arbitrary $p$, 
similar calculations yield the desired result with 
\beam\label{eq:sigma:2}
\sigma_\a^2 &:=& \Var\Big( \mathbb{G}\big( (\a h - e)(\cdot/1),\, \alpha\big) \Big),\\
\sigma_{f,\a}(\kappa) &:=& \Cov\Big( \mathbb{G}\big( (f_\alpha - f_\alpha^\bfQ)(\cdot/1),\, p \big),  \mathbb{G}\big( (\alpha h - e)( \cdot/1 ) , \,\alpha \big) \Big).\nonumber
\eeam
}


\black

\section{Proofs of Appendix~\ref{sec:consistency} }\label{sec:profs:app:A}

\subsection{Proof of Lemma~\ref{lem:mixing}}\label{proof:lemma:mixing}
We recall the notation on disjoint blocks defined in \eqref{eq:notation:blocks}. 
We also denote $\ell := \ell_n \to \infty$, and disjoint blocks as
\begin{align*}
  \mathcal{B}_{t,\ell} \quad := \quad  \bfX_{(t-1)b_n +[1,b_n-\ell_n]}, \qquad t=1, \dots, m_n.
\end{align*}
Notice that for all $\delta > 0, \epsilon > 0$, and for every bounded Lipschitz-continuous function $f \in \mathcal{G}_+(\tilde{\ell}^p)$ 
\beao 
    \big| \E\big[ \exp \big\{ - \tfrac{1}{k_n}\mbox{$ \sum_{t=1}^{m_n}$} \lefteqn{f(x_{b_n}^{-1} \mathcal{B}_{t})
    \big\}\big] - \E\big[ \exp \big\{ - \tfrac{1}{k} \mbox{$\sum_{t=1}^{m_n}$} f(\underline{x_{b_n}^{-1} \mathcal{B}_{t}}_\epsilon)
    \big\}
    \big]
    \big|} \\
         &\le& \E\big[\big| \tfrac{1}{k_n} \mbox{$\sum_{t=1}^{m_n}$} f(x_{b_n}^{-1} \mathcal{B}_t)
         -  \tfrac{1}{k_n} \mbox{$\sum_{t=1}^{m_n}$} f(\underline{x_{b_n}^{-1} \mathcal{B}_t}_\epsilon)
         \big|  \big] \\
        &\le& \E\big[ \tfrac{1}{k_n} \mbox{$\sum_{t=1}^{m_n}$} \big| f(x_{b_n}^{-1} \mathcal{B}_t)
        -  f(\underline{x_{b_n}^{-1} \mathcal{B}_t}_\epsilon)
        \big|  \big]\\
        &=& o\big( {m_n\P(\|\overline{\mathcal{B}_1/x_b}^\epsilon\|_p > \delta )}/k_n\big), \quad n \to \infty. 
\eeao 
This term vanishes by condition ${\bf CS}_p.$ Moreover, define
\begin{align*}
    &I = \big| \E\big[ \exp \big\{ - \tfrac{1}{k_n} \mbox{$\sum_{t=1}^{m_n}$} f_\epsilon(x_{b_n}^{-1} \mathcal{B}_{t})
    \big\}\big] - \E\big[ \exp \big\{ - \tfrac{1}{k_n}\mbox{$ \sum_{t=1}^{m_n}$} f_\epsilon(x_{b_n}^{-1} \mathcal{B}_{t,\ell})
    \big\}
    \big]
    \big|, 
\end{align*}
where $f_\epsilon(\bfx_t) := f( \underline{\bfx_t}_\epsilon)$. Then, there exists a constant $c >0$ such that
\beao
    I &\le& c \, \frac{1}{k_n} \,  \P \big( \max_{1 \le j \le m_n} \max_{1 \le i \le \ell_n}|\bfX_{(j-1)b_n -i + 1}| > \epsilon \, x_{b_n} \big) \\
            &\le& c\, \frac{\, m_n}{k_n} \P( \|\mathcal{B}_{1,\ell}\|_\infty > \epsilon \, x_{b_n}) \\
            & \le& c\, \frac{m_n \ell_n}{k_n}  \P(|\bfX_0| > \epsilon \, x_{b_n}) \quad 
            = \quad  O(\ell_n/b_n),
\eeao 
as $n \to \infty$.
Thus, we conclude that $\lim_{n \to \infty} \ell_n/b_n = 0$ is a sufficient condition yielding $I \to 0$, as $n \to \infty$. Recall the definition of the mixing coefficients $(\beta_t)$ in Section~\ref{definition:mixing}. Moreover, applying the mean value theorem we have that for all $z,y>0$, then $|e^{-z} - e^{-y}| \le |z-y|$.
For this reason,
\beao 
     \lefteqn{\big| \E\big[ \exp \big\{ - \tfrac{1}{k_n} \mbox{$\sum_{t=1}^{m_n}$} \lefteqn{ f_\vep(x_{b_n}^{-1} \mathcal{B}_{t,\ell} )
     \big\}\big]  - \E\big[ \exp \big\{ - \tfrac{1}{k_n} \mbox{$\sum_{t=1}^{m_n}$} f_\vep(x_{b_n}^{-1} \mathcal{B}^*_t)
     \big\}\big]\big|}} \\
    &\le& \frac{1}{k_n}\sum_{t=1}^{m_n}
    \big| \E\big[ f_\vep(x_{b_n}^{-1} \mathcal{B}_{t,\ell})
    - f_\vep(x_{b_n}^{-1} \mathcal{B}^*_t)
    \big] \big|\\
       &\le& \frac{2 m_n}{k_n} \|f\|_\infty  \, d_{TV} \big(\, \mathcal{L} (\mathcal{B}_{t,l}) \otimes \underbrace{ \mathcal{L}(\bfX_1) \otimes \cdots \otimes \mathcal{L}(\bfX_1)}_{\ell \text{ times } } \, , \, \mathcal{L} (\mathcal{B}_{t})\, \big) \\
        &\le&  \frac{2 m_n}{k_n} \|f\|_\infty  \beta_{\ell_n}
           \to 0, \quad n \to \infty.
\eeao 
We use first the definition of the total variation distance, and second a reformulation of the distance in terms of the mixing coefficients $(\beta_t)$. This last relation allows us to conclude that \eqref{mixing:condition2} holds under the conditions of Lemma~\ref{lem:mixing}.

{
\subsection{Proof of Lemma~\ref{pro:cor:consis} }\label{proof:lem:consis:complete}
Consider a non-negative continuous function $f=f_q: \tilde{\ell}^p \to \mathbb{R}$ satisfying the assumptions of Lemma~\ref{pro:cor:consis}.
We focus on showing the case $p=\alpha$, as the general case can be deduced following the lines of this proof keeping $q = p$ fixed.
We follow a similar argument as in the proof of Lemma 8.5 in \cite{buritica:mikosch:wintenberger:2021}
to show
\beao 
|\widetilde f_q^{\bfQ}(u,q) - f_\alpha^\bfQ(\alpha)| &\xrightarrow[]{\P}& 0, 
\eeao 
as $(n,u,q) \to (\infty, 1, \alpha)$. 
Fix $\eta > 0$, and denote $g = f_\alpha \land \eta $, then
\beam \label{eq:cons:dec}
\lefteqn{
|\widetilde f_q^{\bfQ}(u,q) - f_\alpha^\bfQ(\alpha)|} \nonumber \\
&\le &
|\widetilde f_q^{\bfQ}(u,q) - 
\widetilde f_\alpha^{\bfQ}(u,q)| + |\widetilde f_\alpha^{\bfQ}(u,q) - \widetilde g^{\bfQ}(u,q)| + |\widetilde g^{\bfQ}(u,q) - g^\bfQ(\alpha)| \nonumber \\
&& + |g^\bfQ(\alpha) - f_\alpha^\bfQ(\alpha)| \nonumber \\
&&\qquad = I + II + III + IV.
\eeam
We treat the four terms in \eqref{eq:cons:dec} separately. 
\par 
Regarding term $III$ in \eqref{eq:cons:dec}, note that
\beao 
\lefteqn{
|\widetilde g^{\bfQ}(u,q) - g^\bfQ(\alpha)|}\\
&\le &
|\widetilde g^{\bfQ}(u,q) - 
\widetilde g^{\bfQ}(1,q)| 
+ 
|\widetilde g^{\bfQ}(1,q) - 
\widetilde g^{\bfQ}(1,\alpha)|
+ 
|\widetilde g^{\bfQ}(1,\alpha) - g^\bfQ(\alpha)|  \\
&=& |\widetilde g^{\bfQ}(u,q) - 
\widetilde g^{\bfQ}(1,q)| 
+  
|\widetilde g^{\bfQ}(1,q) - 
\widetilde g^{\bfQ}(1,\alpha)| +
o_\P(1), \quad n \to \infty,
\eeao 
such that the last relation holds by an application of Lemma 8.5. in \cite{buritica:mikosch:wintenberger:2021}.
Moreover, note that $g$ is a bounded function and with probability zero $Y\bfQ$ belongs to the points of discontinuity of $g$. 
Moreover, this implies that $g$ can be approximated by a monotone sequence of Lipschitz-continuous functions. 
Therefore, we can assume without loss of generality that $g$ is Lipschitz-continuous. 
In this case note 
\beao 
\lefteqn{
|\widetilde g^{\bfQ}(u,q) - g^\bfQ(\alpha)|}\\
&\le &
(\widetilde g^{\bfQ}(u,q)  - 
\widetilde g^{\bfQ}(1,q) )
+ 
(\widetilde g^{\bfQ}(1,q) - 
\widetilde g^{\bfQ}(1,\alpha))
+ 
|\widetilde g^{\bfQ}(u,q) - 
\widetilde g^{\bfQ}(1,q)| \\
&\le & \eta (u_0^{-\alpha} - 1) \tilde 1^\bfQ(u_0,q_0) + 
(\widetilde g^{\bfQ}(1,q_0) - 
\widetilde g^{\bfQ}(1,\alpha))
+ 
|\widetilde g^{\bfQ}(u,q) - 
\widetilde g^{\bfQ}(1,q)|,
\eeao 
where we have also used here the monotonicity of the function $q \mapsto \1(\|\bfx_t\|_q > 1)$.
Then there exists $u_0, q_0$, such that $0 <  u_0^{-\alpha} - 1 < \eta^{-1} \epsilon/3$ and $ 0 < c(\alpha) - c(q_0) < \epsilon/3$, $u_0^{-\alpha} c(q_0) < 2$, and $1 > u > u_0$,  $ \alpha > q> q_0$. This last assertion follows by continuity and monotonicity of the functions $u \mapsto u^{-\alpha}$, and also of $q \mapsto c(q)$. This last is granted by the fact that $c(q) = \E[\|\bfQ\|_q^\alpha] < \infty$ and monotonicity of the $\ell^p-$norms yields the continuity of this last function.
Similarly, we can find $n_0 \in \mathbb{N}$ such that, for $1>u > u_0, 1> q > q_0$, and $n > n_0$, 
\beao 
\lefteqn{
|\widetilde g^{\bfQ}(u,q) - g^\bfQ(\alpha)|} \\
&\le & \eta (u_0^{-\alpha} - 1) \tilde 1^\bfQ(u_0,q_0) + (g^\bfQ(1,q_0) - g^\bfQ(1,\alpha)) + \epsilon/3 \\
&\le & \eta (u_0^{-\alpha} - 1) ( u_0^{-\alpha} c(q_0) + o_\P(1))  + (c(q_0) - 1) + \epsilon/2\\
&\le & \epsilon. 
\eeao 
where the second inequality holds by an application of Lemma 8.5 in \cite{buritica:mikosch:wintenberger:2021}.
Since $\epsilon$ was chosen aribitrarily, then we conclude
$|\widetilde g^{\bfQ}(u,q) - g^\bfQ(\alpha)| \xrightarrow[]{\P} 0 $, as $(n,u,q) \to (\infty, 1, \alpha)$.
\par 
We now turn our attention to the approximations in $II$ and $IV$ in \eqref{eq:cons:dec}, which are the two of them of a similar nature.
In this case, note  
\beao 
\lefteqn{ 
\E\Big[ \sup_{
\substack{u \in 
[1 - \epsilon,
1 + \epsilon 
] , \\
q \in 
[\alpha - \epsilon,
\alpha + \epsilon 
]
}
} |\widetilde f^{\bfQ}_\alpha(u,q) - \widetilde  g^{\bfQ}(u,q)| 
\Big]}\\
&\le & 
\frac{m_n}{k_n}
\E\Big[ \sup_{u,q} f_\alpha(\mathcal{B}_t/(u x_{b_n})) \1(\|\mathcal{B}_t\|_q > (u x_{b_n})) \\
    && \times \1\big(f_\alpha(\mathcal{B}_t/(u x_{b_n})) \1(\|\mathcal{B}_t\|_q > (u x_{b_n})) > \eta \big)  \Big]\\
&\le &  
\frac{m_n}{k_n} \E\left[ \sup_{u} \left( f_\alpha(\mathcal{B}_t/(u x_{b_n}))\right)^{1+\delta} \1(\|\mathcal{B}_t\|_{q_0} > (u_0 x_{b_n})) \right]^{\frac{1}{1+\delta}}  \\
    && \times \P\big( \sup_{u} f_\alpha(\mathcal{B}_t/(u x_{b_n})) \1(\|\mathcal{B}_t\|_{q_0} > (u_0 x_{b_n})) > \eta\big)^{\frac{\delta}{1+\delta} }   \\
&\le& \frac{m_n}{k_n} \eta^{-\delta} \E\left[\sup_{u}  \left( f_\alpha(\mathcal{B}_t/(u x_{b_n}))\right)^{1+\delta} \1(\|\mathcal{B}_t\|_{q_0} > (u_0 x_{b_n})) \right].
\eeao 
Hence, an application of condition {\bf C} together with the previous bound implies $II \xrightarrow[]{\P} 0$ letting $(n,u,q) \to (\infty,1,\alpha)$, and then letting $\eta \to \infty$ in the previous inequality by the Lindeberg-type condition in \eqref{eq:cond:consistency:unbounded}.
A similar argument implies $IV \to 0$ letting $(n,u,q) \to (\infty,1,\alpha)$, and then letting $\eta \to \infty$.
\par 
To conclude, we now show that $I$ in
\eqref{eq:cons:dec} also converges to zero.
In this case, notice that by assumption {\bf S} we have 
\beao 
\lefteqn{
I} &=& |\widetilde f_q^{\bfQ}(u,q) - \widetilde f_\alpha^{\bfQ}(u,q)| 
  \; \le \;  |q-\alpha|  \sup_{\substack
  { q  \in [\alpha-\epsilon, \alpha + \epsilon]
  \\u \in [1-\epsilon, 1+ \epsilon]
  }} 
 \widetilde{\tfrac{ \partial g_q}{\partial q}} |_{q = q}^\bfQ (u,q ). 
\eeao 
Therefore, 
\beao 
\E[I] &=& \E[|\widetilde f_q^{\bfQ}(u,q) - \widetilde f_\alpha^{\bfQ}(u,q)|] \\
&\le & \frac{m_n}{k_n} |q-\alpha| \P( \|\mathcal{B}_t\|_{q_0}> u_0)^{\frac{\delta}{1+\delta}}  \\
    && \times \, \E\Big[  \sup_{\substack
  { q  \in [\alpha-\epsilon, \alpha + \epsilon]\\
  u \in [1-\epsilon, 1+ \epsilon]
  }}  \left( 
{\tfrac{ \partial g_q}{\partial q}} |_{q = q} \right)^{1+\delta} (\mathcal{B}_t/(ux_{b_n})) \1( \|\mathcal{B}_t\|_q > ux_{b_n} )  \Big]^{\frac{1}{1+\delta} }. 
\eeao 
Then, applying assumption {\bf C} we conclude $I \xrightarrow[]{\P} 0$ as $(n,u,q) \to (\infty, 1, \alpha)$. 
Then, this shows, $|\widetilde f_q^{\bfQ}(u,q) - f_\alpha^\bfQ(\alpha)| \xrightarrow[]{\P}0$,
as $(n,u,q) \to (\infty, 1, \alpha)$. 
Recall we assumed that $ \widehat \alpha \xrightarrow[]{\P} \alpha$,  
and recall $q \mapsto c(q) = \E[\|\bfQ\|_q^\alpha]$ is a continuous function,
then the continuous mapping theorem implies 
\[
|\widetilde f_{\widehat \a}^{\bfQ}(u, \widehat \alpha) - f_\alpha^\bfQ(\alpha)|  \; \xrightarrow[]{\P}\;  0, \]
as $(n,u) \to (\infty,1)$.
The previous limit also holds for the functional $1(\alpha)$.
We now apply the argument in \cite{resnick:2007}, p.81 which allow us to conclude 
$\|\mathcal{B}_t\|_{\widehat \alpha, (k+1)} / x_{b_n} \xrightarrow[]{\P} 1$, as $n \to \infty$.
A final application of the continuous mapping theorem for the scaling function implies 
\[
|\widetilde f_{\widehat \a}^{\bfQ}(\|\mathcal{B}_t\|_{\widehat \alpha, (k+1)} / x_{b_n}, \widehat \alpha) - f_\alpha^\bfQ(\alpha)| \;=\; | \widehat f_{\widehat \alpha}^\bfQ -f_\alpha^\bfQ| \; \xrightarrow[]{\P}\;  0, \]
as $n \to \infty$, and this gives the desired result.
\qed


}

\section{Proofs of auxiliary results}\label{sec:appendix:C}


\subsection{Proof of Proposition~\ref{prop:coupling}}\label{subs:sec:proof:prop}\label{proof:lem:covariance}

{



We show the result for $g \in \mathcal{F}_f$ as the proof in the case $g \in \mathcal{F}_h$ follows the same line of arguments. 
Let $u_0 = 1- \epsilon < 1 < 1+ \epsilon = s_0$ and $q_0 = \a - \epsilon < \a < \a + \epsilon = q_0^\prime$.
We argue using a coupling argument and design recursively coupled blocks $(\mathcal{B}_t^*)_{1\le t\le m_n}$ as follows: for every $t=1,\ldots ,m_n$, we apply the maximal coupling Theorem 5.1 in \cite{rio:2017} to create the block $\mathcal{B}_t^*$
independent of the past blocks  $(\mathcal{B}_j,\mathcal{B}_j^*)_{j<t-1}$, distributed as $\mathcal B_1$, 
and such that 
\beam \label{eq:maximal:coupling}
\P\big(
\mathcal{B}_t \not = \mathcal {B}^*_t
\big) = \beta_{b_n}, \quad t =1,2,\dots,m_n\,,
\eeam 
This construction yields a 1--dependent sequence $(\mathcal{B}_t^*)_{1\le t\le m_n}$.
In particular, this implies $(\mathcal{B}_t^*)_{1\le t\le m_n, t \text{ even } }$ and $(\mathcal{B}_t^*)_{1\le t\le m_n, t \text{ odd } }$ are two sequences each one consisting of independent blocks with the same distribution as $\mathcal B_1$.
Now, applying the Markov inequality of order $2$ yields
\beao
\lefteqn{\P\Big( \sup_{u,q}   \sqrt{k_n}\, | \widetilde{g}^\bfQ(u,q) -  \widetilde{g}_{,*}^\bfQ(u,q)| > \delta \Big)}\\
&\le&\dfrac{1}{k_n\delta^2}\E\Big[ \sup_{u,q}
\Big( \sum_{t=1}^{m_n}(g(\mathcal{B}_{t}/(u\, x_{b_n}) )\1( \|\mathcal{B}_{t}\|_q > (u\, x_{b_n}) )\\
&&-g(\mathcal{B}^*_{t}/(u\, x_{b_n}) )\1( \|\mathcal{B}^*_{t}\|_q > (u\, x_{b_n}) ))\Big)^2\Big]\,\\
&\le&\dfrac{2}{k_n\delta^2}\E\Big[ \sup_{u,q}
\Big( \sum_{t=1, t \text{ odd } }^{m_n}(g(\mathcal{B}_{t}/(u\, x_{b_n}) )\1( \|\mathcal{B}_{t}\|_q > (u\, x_{b_n}) )\\
&&-g(\mathcal{B}^*_{t}/(u\, x_{b_n}) )\1( \|\mathcal{B}^*_{t}\|_q > (u\, x_{b_n}) ))\Big)^2\Big]\,. \\
&& + \dfrac{2}{k_n\delta^2}\E\Big[ \sup_{u,q}
\Big( \sum_{t=2, t \text{ even } }^{m_n}(g(\mathcal{B}_{t}/(u\, x_{b_n}) )\1( \|\mathcal{B}_{t}\|_q > (u\, x_{b_n}) )\\
&&-g(\mathcal{B}^*_{t}/(u\, x_{b_n}) )\1( \|\mathcal{B}^*_{t}\|_q > (u\, x_{b_n}) ))\Big)^2\Big]
\eeao
Both terms can be handled in a similar manner so we focus on the first one. 
We develop the square and obtain a diagonal term
\beao
\dfrac{m_n}{k_n}\E\Big[\sup_{u,q}(g(\mathcal{B}_{1}/(u\, x_{b_n}) )\1( \|\mathcal{B}_{1}\|_q > (u\, x_{b_n}) )-g(\mathcal{B}^*_{1}/(u\, x_{b_n}) )\1( \|\mathcal{B}^*_{1}\|_q > (u\, x_{b_n}) ))^2\Big] 
\eeao
Note that, denoting $G$ the envelope function,
\beam \label{eq:G:envelope}
  G(\mathcal B_1/x_{b_n}) & := & \sup_{u\in [u_0,s_0],q \in [q_0,q_0^\prime] }g(\mathcal B_1/(u x_{b_n}))\1(\|\mathcal{B}_1\|_q > (u x_{b_n}))\\
&=& g(\mathcal B_1/(u_0 x_{b_n}))\1(\|\mathcal{B}_1\|_{q_0} > (u_0 x_{b_n})). \nonumber 
\eeam  
This means we now need to control 
\beao
\lefteqn{\dfrac{m_n}{k_n}\E\Big[\sup_{u,q}(g(\mathcal{B}_{1}/(u\, x_{b_n}) )\1( \|\mathcal{B}_{1}\|_q > (u\, x_{b_n}) )}\\
&& -g(\mathcal{B}^*_{1}/(u\, x_{b_n}) )\1( \|\mathcal{B}^*_{1}\|_q > (u\, x_{b_n}) ))^2\1(\mathcal{B}_{1}\neq\mathcal{B}_{1}^*) \Big]\\
&\le& \dfrac{2m_n}{k_n}\E[G(\mathcal{B}_{1}/x_{b_n})^2\1(\mathcal{B}_{1}\neq\mathcal{B}_{1}^*)]
\eeao
\par 
We develop the square and obtain a diagonal term
\beao
\lefteqn{
\dfrac{m_n}{k_n}\E\Big[\sup_{u,q}(g(\mathcal{B}_{1}/(u\, x_{b_n}) )\1( \|\mathcal{B}_{1}\|_q > (u\, x_{b_n}) )} \\
&& -g(\mathcal{B}^*_{1}/(u\, x_{b_n}) )\1( \|\mathcal{B}^*_{1}\|_q > (u\, x_{b_n}) ))^2\1(\mathcal{B}_{1}\neq\mathcal{B}_{1}^*)\Big]\\
&\le&
2 \E[G(\mathcal{B}_{1}/x_{b_n})^2\1(\mathcal{B}_{1}\neq\mathcal{B}_{1}^*)]\\
&\le& \dfrac{m_n}{k_n}\E[G(\mathcal{B}_{1}/x_{b_n})^{2+\delta}]^{2/(2+\delta)}\P(\mathcal{B}_{1}\neq\mathcal{B}_{1}^*)^{\delta/(2+\delta)}\\
&\le& \dfrac{2 m_n}{k_n}\P(\|\mathcal{B}_1\|>x_{b_n})^{2/(2+\delta)}\beta_{b_n}^{{\delta/(2+\delta)}}\sim 2 (m_n\beta_{b_n}/k_n)^{\delta/(2+\delta)},
\eeao
as $n \to \infty$, if $G$ satisfies the assumption {\bf L}.

We now focus on the crossed terms in the development of the square.
Note we assume without loss of generality that $g$ is non-negative by considering separately the positive and negative parts. Then note that for $j \not = t$, we obtain
\beao 
&&\E\Big[\sup_{u,q} \Big( g(\mathcal B_t/(ux_{b_n})) \1(\|\mathcal B_t\|_q > (u x_{b_n}))
- g(\mathcal B^*_t/(ux_{b_n})) \1(\|\mathcal B^*_t\|_q > (u x_{b_n}))
\Big)\\
&&\times \Big( g(\mathcal B_j/(ux_{b_n})) \1(\|\mathcal B_j\|_q > (u x_{b_n}))
- g(\mathcal B^*_j/(ux_{b_n})) \1(\|\mathcal B^*_j\|_q > (u x_{b_n}))
\Big)
\Big]
\eeao 
Moreover, note 
\beao 
\lefteqn{\dfrac{m_n^2}{k_n} \E\Big[\sup_{u,q} \Big( g(\mathcal B_1/(ux_{b_n})) \1(\|\mathcal B_1\|_q > (u x_{b_n}))}\\
&& \qquad 
- g(\mathcal B^*_1/(ux_{b_n})) \1(\|\mathcal B^*_1\|_q > (u x_{b_n}))
\Big)\Big]^2\\
&\le& 
\dfrac{m_n^2}{k_n} \E\Big[\sup_{u,q} \Big( g(\mathcal B_1/(ux_{b_n})) \1(\|\mathcal B_1\|_q > (u x_{b_n}))\\
&& \qquad 
- g(\mathcal B^*_1/(ux_{b_n})) \1(\|\mathcal B^*_1\|_q > (u x_{b_n}))
\Big)^2\Big] \\
&& \times \P( \mathcal{B}_1 \neq \mathcal{B}^*_1) \\
&\le &  \frac{2 m_n^2}{k_n} \E[G(\mathcal{B}_{1}/x_{b_n})^2]
\P( \mathcal{B}_1 \neq \mathcal{B}^*_1)\\
&\sim& m_n \beta_{b_n},
\eeao 
as $n\to \infty$. 
Therefore, the diagonal term can be computed as 
\beao 
\lefteqn{ \frac{2 m_n}{k_n} \sum_{\substack{j=3 \\ j \text{ odd} } }^{m_n}\E\Big[\sup_{u,q} \Big| g(\mathcal B_1/(ux_{b_n})) \1(\|\mathcal B_1\|_q > (u x_{b_n}))}\\
&& \qquad 
- g(\mathcal B^*_1/(ux_{b_n})) \1(\|\mathcal B^*_1\|_q > (u x_{b_n}))
\Big|  \\
&&\times \Big| g(\mathcal B_j/(ux_{b_n})) \1(\|\mathcal B_j\|_q > (u x_{b_n}))\\
&& \qquad 
- g(\mathcal B^*_j/(ux_{b_n})) \1(\|\mathcal B^*_j\|_q > (u x_{b_n}))
\Big|
\Big] \\
&= & \frac{2 m_n}{k_n} \sum_{\substack{j=3 \\ t \text{ odd} } }^{m_n}  \Cov\Big(\sup_{u,q} \Big| g(\mathcal B_1/(ux_{b_n})) \1(\|\mathcal B_1\|_q > (u x_{b_n}))\\
&& \qquad 
- g(\mathcal B^*_1/(ux_{b_n})) \1(\|\mathcal B^*_1\|_q > (u x_{b_n}))
\Big| \\
&&\times \sup_{u,q} \Big| g(\mathcal B_j/(ux_{b_n})) \1(\|\mathcal B_j\|_q > (u x_{b_n}))\\
&& \qquad 
- g(\mathcal B^*_j/(ux_{b_n})) \1(\|\mathcal B^*_j\|_q > (u x_{b_n}))
\Big|
\Big) + o(1), 
\eeao 
as $n \to \infty.$
where in the last equality we added and subtracted the centering term yielding a covariance. 

Finally, by a direct corollary of Theorem 1.1. in \cite{rio:2017}, directly stated in equation (1.12b) therein, yields that the remaining term of the development on the crossed-terms is of the form
\beao 
\lefteqn{ \frac{2 m_n}{k_n} \sum_{\substack{j=3 \\ j \text{ odd} } }^{m_n}
\beta_{(j-2)b_n}^{\tfrac{\delta}{2+\delta}}} \\
&& \times \E\Big[ \sup_{u,q} \Big| g(\mathcal B_j/(ux_{b_n})) \1(\|\mathcal B_j\|_q > (u x_{b_n}))\\
&& \qquad - g(\mathcal B^*_j/(ux_{b_n})) \1(\|\mathcal B^*_j\|_q > (u x_{b_n}))
\Big|^{2+\delta} \Big]^{\tfrac{2}{2+\delta}} \\
&\le &  c\, \frac{ m_n}{k_n} \sum_{\substack{j=3 \\ j \text{ odd} } }^{m_n}
\beta_{(j-2)b_n}^{\tfrac{\delta}{2+\delta}}  \E[ G(\mathcal{B}_1/x_{b_n})^{2+\delta} ]^{\tfrac{2}{2+\delta}} \\
&\sim& c\, \sum_{\substack{j=3 \\ j \text{ odd} } }^{m_n}
\Big( \frac{m_n\beta_{(j-2)b_n}}{k_n} \Big)^{\tfrac{\delta}{2+\delta}},
\eeao 
as $n \to \infty$. 
The above inequality can be extended to bounded random variables letting $\delta \to \infty$. 
Then, to sum up, provided ${\bf MX}_\beta$ holds, we obtain
\beao 
\P\Big( \sup_{u,q}   \sqrt{k_n}\, | \widetilde{g}^\bfQ(u,q) -  \widetilde{g}_{,*}^\bfQ(u,q)| > \delta \Big) \to 0, \quad n \to \infty,
\eeao 
and this yields the desired result and concludes the proof of Proposition~\ref{prop:coupling}.

\qed 
}

{

\subsection{Proof of Lemma~\ref{proof:cov:iid:blocks}}\label{proof:lem:iid:blocks}

We start by showing that if $g:\tilde{\ell}^p \to \mathbb{R}$ satisfies ${\bf L}$ then 
\beam \label{eq:rihgt:hand:side}
\frac{m_n}{k_n}\E[g\big(\mathcal{B}_1/(v x_{b_n})\big)\1\big(\|\mathcal{B}_1\|_q > v x_{b_n}\big)] &\to& v^{-\alpha} \frac{c(q)}{c(p)}\E[g(Y \bfQ^{(q)})],
\eeam 
as $n \to \infty$.
Note that if $g$ is a bounded function, then the previous relation follows straightforwardly from Proposition~\ref{prop:existence:cluster:process}. 
If $g$ is not bounded, then note that for $\delta > 0$ as in \eqref{eq:cond:lind:1}, and for all $\eta > 0$,
\beam \label{eq:truncating:function}
\lefteqn{
\frac{m_n}{k_n} \E[g\big(\mathcal{B}_1/(v x_{b_n})\big)\1\big(\|\mathcal{B}_1\|_q > v x_{b_n}\big)] } \nonumber \\
    &=& \frac{m_n}{k_n} \E[g\big(\mathcal{B}_1/(v x_{b_n})\big)\1\big(\|\mathcal{B}_1\|_q > v x_{b_n}\big) \1(g\big(\mathcal{B}_1/(v x_{b_n})\big)\1\big(\|\mathcal{B}_1\|_q > v x_{b_n}\big)  > \eta) ] \nonumber \\
    && + \eta \, \frac{m_n}{k_n} \P(g\big(\mathcal{B}_1/(v x_{b_n})\big)\1\big(\|\mathcal{B}_1\|_q > v x_{b_n}\big)  > \eta ) \nonumber \\
    && +
    \frac{m_n}{k_n} \E[g\land \eta  \big(\mathcal{B}_1/(v x_{b_n})\big) \1\big(\|\mathcal{B}_1\|_q > v x_{b_n}\big) ] \nonumber \\
    &&\qquad \qquad = I + II + III.
\eeam
Applying Proposition~\ref{prop:existence:cluster:process} we have $III \to c(q) v^{-\alpha}\, \E[g\land \eta(Y \bfQ^{(q)})]/c(p)$, as $n \to \infty$, and letting $\eta \to \infty$ we obtain the right-hand side term in \eqref{eq:rihgt:hand:side}. Hence, it remains to show that $I+ II \to 0$, letting $n \to \infty$, and then $\eta \to \infty$.
\par 
Regarding term $I$ in \eqref{eq:truncating:function}, 
\beao 
\lefteqn{I}\\
    &=& \E\big[g\big(\mathcal{B}_1/(v x_{b_n})\big)\1\big(\|\mathcal{B}_1\|_q > v x_{b_n}\big) \1(g\big(\mathcal{B}_1/(v x_{b_n})\big)\1\big(\|\mathcal{B}_1\|_q > v x_{b_n}\big)  > \eta) \big] \\
&\le& \frac{m_n}{k_n}\E\big[ \big(g\big(\mathcal{B}_1/(v x_{b_n})\big)\big)^{1+\delta} \1 ]^{\frac{1}{1+\delta}} 
\P\big( g\big(\mathcal{B}_1/(v x_{b_n})\big)\1\big(\|\mathcal{B}_1\|_q > v x_{b_n}\big) > \eta \big)^{\frac{\delta}{1+\delta}}\\
&\le& \frac{m_n}{k_n} \eta^{-\delta} \E\big[ \big(g\big(\mathcal{B}_1/(v x_{b_n})\big))^{1+\delta}\1\big(\|\mathcal{B}_1\|_q > v x_{b_n}\big)  ]\,,
\eeao 
and we see that by assumption {\bf L} that $I \to 0$ letting first $n \to \infty$ and lastly $\eta \to \infty$. 
Finally, for term $II$ in \eqref{eq:truncating:function}, we apply a Markov inequality which yields 
\beao 
II &=& \eta \frac{m_n}{k_n} \P\big( g\big(\mathcal{B}_1/(v x_{b_n})\big)\1\big(\|\mathcal{B}_1\|_q > v x_{b_n}\big) > \eta \big)\\
&\le& \eta^{-\delta} \frac{m_n}{k_n} \E\big[ \big(g\big(\mathcal{B}_1/(v x_{b_n})\big))^{1+\delta}\1\big(\|\mathcal{B}_1\|_q > v x_{b_n}\big)  ],
\eeao 
and we conclude also for term $II$ that $II \to 0$ letting $n \to \infty$ and then $\eta \to \infty$. 
Overall this shows \eqref{eq:rihgt:hand:side} holds.
\par 
We now consider the functional $h$ in \eqref{eq:func:hill} defining the Hill estimator. 
Assume that $k_n/k_n^\prime \to \kappa$ and $\kappa > 0$.
In this case notice that since $(k_n)$ and $(k^\prime_n)$  satisfy \eqref{eq:k:3} and \eqref{eq:k:4}, respectively, then 
\beao 
\frac{m_n}{k_n^\prime} \sim \frac{k_n}{ k_n^\prime \P(\|\mathcal B_1\|_p > x_{b_n})} &\sim&
\frac{k_n}{k_n^\prime} \frac{ \P(|\bfX_1| > x^\prime_{b_n}) }{ c(p) \P(|\bfX_1| > x_{b_n}) } \frac{1}{b_n \P(|\bfX_1| > x^{\prime}_{b_n})}\\
&\sim&\frac{1}{b_n c(p) \P(|\bfX_1| > x^{\prime}_{b_n})}, \quad n \to \infty,
\eeao 
and this holds since $k_n/k_n^\prime \to \kappa$, and $\P(|\bfX_1| > x_{b_n}^\prime)/\P(|\bfX_1| > x_{b_n}) \to \kappa^{-1} $, as $n \to \infty$, due to \eqref{eq:rationx}, and regular variation of the variable $|\bfX_1|$.
Then,
\beam\label{eq:right:handside:hill}
\frac{m_n}{k^\prime_n} \E[h\big(\mathcal{B}_1/(v x^\prime_{b_n})\big)]  &=& \frac{m_n}{k^\prime_n} \E[h\big(\mathcal{B}_1/(v x^\prime_{b_n})\big)\1\big(\|\mathcal{B}_1\|_\infty > v x^\prime_{b_n}\big)] \nonumber \\
&\to& v^{-\alpha}c(\infty) \,  \E[h(Y \bfQ^{(\infty)})], \quad n \to \infty, \nonumber \\
&=& v^{-\alpha}  \E[h(Y \bfQ )]. 
\eeam 
such that the first limit holds applying Proposition~\ref{prop:existence:cluster:process}, and the last equality follows from the change-of-norms equation in \eqref{eq:change:of:norms}.
Moreover by similar arguments we can also show that $m_n\E[e(\mathcal B_1/(v x^\prime_{b_n})]/k^\prime_n \to v^{-\a} \E[e(Y\bfQ)]$, for the exceedances functional in \eqref{eq:func:exceedances}, and also
$m_n\P(\|\mathcal B_1\|_q > (v x'_{b_n}) )/k^\prime_n \to v^{-\a} c(q)$, as $n \to \infty$. 
\par 
We now focus on computing the covariance of blocks.
Denote $\kappa_n = k_n/k^\prime_n$, and denote $c_n = x_{b_n}/x^\prime_{b_n}$. Note that by our assumptions $\kappa_n \to \kappa > 0$ and by Equation~\eqref{eq:rationx} we also have $c_n \to (\kappa/{c(p)})^{-1/\alpha}$, as $n \to \infty.$
We start by considering the case where $g:\tilde{\ell}^p \to \mathbb{R}$ is a functional satisfying {\bf L}, and $h$ is the functional defining the Hill estimator in \eqref{eq:func:hill}.
{
Moreover,
\beam \label{eq:littleo}
\lefteqn{ 
k_n \, \cov\big( \widetilde g^\bfQ_{,*}(v,q),   \widetilde h^\bfQ_{,*}(v)\big) 
} \nonumber \\
    &=& \frac{k_n m_n }{k_n k_n^\prime}\E\Big[ g\big(\mathcal{B}_1/(v x_{b_n}) \big) h\big( \mathcal{B}_1/(u x^\prime_{b_n}) \big) \1\big(\|\mathcal{B}_1\|_q > v x_{b_n}\big)  \Big] \nonumber \\
    && - \, \frac{k_n m_n }{k_n k_n^\prime}\E\Big[ g\big(\mathcal{B}_1/(v x_{b_n}) \big) \1\big(\|\mathcal{B}_1\|_q > v x_{b_n}\big) \big)  \Big]\E\Big[h\big( \mathcal{B}_1/(u x^\prime_{b_n}) \big)\Big] + o(1), \nonumber \\
    &&= \frac{m_n \kappa_n }{k_n}\E\Big[ g\big(\mathcal{B}_1/(v x_{b_n}) \big) h\big( \mathcal{B}_1 c_n /(u x_{b_n} ) \big) \1\big(\|\mathcal{B}_1\|_q > v x_{b_n}\big)  \Big] + o(1),
\eeam 
as $n \to \infty$, where the first equality follows from Lemma~\ref{lem:covariance:1} and the} last equality follows from \eqref{eq:rihgt:hand:side} and \eqref{eq:right:handside:hill} and the fact that $k_n/m_n \sim \P(\|\mathcal{B}\|_p > x_{b_n})  \to 0$, as $n \to \infty$.
Moreover, by similar calculations as in \eqref{eq:truncating:function} we can see that it suffices to show, for  all $\eta > 0$,
\beam \label{eq:interm:step}
\lefteqn{ \frac{m_n \kappa_n }{k_n}  \E\Big[ g \land \eta \big(\mathcal{B}_1/(v x_{b_n}) \big) h \land \eta \big( \mathcal{B}_1/(u x^\prime_{b_n}) \big) \1\big(\|\mathcal{B}_1\|_q > v x_{b_n}\big)  \Big]} \nonumber \\
    &\to&  v^{-\alpha} \kappa \frac{c(q)}{c(p)} \E[ g \land \eta( Y\bfQ^{(q)})  h \land \eta( Y\bfQ^{(q)} {v (\kappa/c(p))^{-1/\alpha}} /u  )],
\eeam 
as $n \to \infty$, and then we conclude by monotone convergence by letting $\eta \to \infty$ in the right-hand side of the previous equation. 
\par 
We now focus on showing \eqref{eq:interm:step} holds. For this note that the function $c \mapsto h \land \eta ((\bfx_t)/c)$ is a non-increasing function. 
Then, 
relying on \eqref{eq:rationx}
we obtain that, for all $\epsilon > 0$, and for $n$ sufficiently large,
\beao 
\lefteqn{ \frac{m_n}{k_n}  \E\Big[ g \land \eta \big(\mathcal{B}_1/(v x_{b_n}) \big) h \land \eta \big( \mathcal{B}_1/(u x^\prime_{b_n}) \big) \1\big(\|\mathcal{B}_1\|_q > v x_{b_n}\big)  \Big] } \\
& =  & 
 \frac{m_n }{k_n}  \E\Big[ g \land \eta \big(\mathcal{B}_1/(v x_{b_n}) \big) h \land \eta \big( \mathcal{B}_1 c_n /(u x_{b_n}) \big) \1\big(\|\mathcal{B}_1\|_q > v x_{b_n}\big)  \Big] \\
&\le & 
 \frac{m_n }{k_n}  \E\Big[ g \land \eta \big(\mathcal{B}_1/(v x_{b_n}) \big) h \land \eta \big( \mathcal{B}_1( { (\kappa/c(p))^{-1/\alpha}} +\epsilon)/(u  x_{b_n}) \big) \1\big(\|\mathcal{B}_1\|_q > v x_{b_n}\big)  \Big] \\
 &\to& v^{-\alpha} \frac{c(q)}{c(p)} \E[ g \land \eta( Y\bfQ^{(q)})  h \land \eta( Y\bfQ^{(q)}  v({(\kappa/c(p))^{-1/\alpha}}  + \epsilon)/u  )], 
\eeao 
as $n \to \infty$, 
where the last limit is again an application of Proposition~\ref{prop:existence:cluster:process}. 
In a similar manner we can also obtain a lower bound with respect to $\epsilon$.
Finally, notice these bounds hold for arbitrary $\epsilon$, thus, letting $\epsilon \downarrow 0$ we can conclude that \eqref{eq:interm:step} holds. 
To sum up, the relation in \eqref{eq:littleo}  together with similar calculations as in \eqref{eq:truncating:function}  allow us to establish that \eqref{eq:cov:iid} holds. 
Furthermore, the calculations for \eqref{eq:cov:iid:2} and \eqref{eq:cov:iid:3} follow the same line of arguments thus we omit the details. 

\qed
}


\subsection{Proof of Lemma~\ref{lem:covariance:1}}\label{sec:lem:covariance:1}
We assume without loss of generality that the functions $g,f :\tilde{\ell}^p \to \mathbb{R}$ satisfying ${\bf L}$ take non-negative values.
Notice that for all $t > 0$
\beam 
\frac{m_n}{k_n}\Cov( g(x_{b_n}^{-1}\mathcal{B}_1)f(x_{b_n}^{-1}\mathcal{B}_t)) 
&=& \frac{m_n}{k_n} \E[g(x_{b_n}^{-1}\mathcal{B}_1)f(x_{b_n}^{-1}\mathcal{B}_t)] + o(1),
\eeam 
as $n \to \infty$, since $m_n\, \E[f^2(x_{b_n}^{-1}\mathcal{B}_1)]/k_n \to \E[f^2(Y\bfQ^{(p)})]$, as $n \to  \infty$, by the moment assumption in \eqref{eq:cond:lind:1}. Moreover,
\beao 
\lefteqn{I \quad = \quad \frac{m_n}{k_n}\E[g(x_{b_n}^{-1}\mathcal{B}_1)\1(g(x_{b_n}^{-1}\mathcal{B}_1) > \eta) f(x_{b_n}^{-1}\mathcal{B}_t)]}\\
&\le&\frac{m_n}{k_n}\E[g(x_{b_n}^{-1}\mathcal{B}_1)^2\1(g(x_{b_n}^{-1}\mathcal{B}_1) > \eta)]^{1/2}\E[ f(x_{b_n}^{-1}\mathcal{B}_1)^2]^{1/2} \\
&\le& \frac{m_n}{k_n} (\eta)^{-\delta/2}\E[g(x_{b_n}^{-1}\mathcal{B}_1)^{2+\delta}]^{\tfrac{1}{2+\delta}}\E[ g(x_{b_n}^{-1}\mathcal{B}_1)^{2+\delta}]^{\tfrac{\delta}{2(2+\delta)}} \E[ f(x_{b_n}^{-1}\mathcal{B}_1)^2]^{1/2}.
\eeao 
We deduce from Equation~\eqref{eq:cond:lind:1} that $I = O(\eta^{-\delta/2})$, as $n \to \infty$, thus this term is negligible letting $n\to \infty$, and then $\eta \to  \infty$. Therefore,
\beam\label{eq:limit:cond}
\lefteqn{\lim_{n \to  \infty}\frac{m_n}{k_n}\E[ g(x_{b_n}^{-1}\mathcal{B}_1)f(x_{b_n}^{-1}\mathcal{B}_t)]} \nonumber \\
&=& \lim_{\eta \to \infty} \lim_{n \to \infty} \frac{m_n}{k_n} \E[(g\land \eta)(x_{b_n}^{-1}\mathcal{B}_1) (f\land \eta)(x_{b_n}^{-1}\mathcal{B}_t)].
\eeam
We conclude that it suffices to establish \eqref{eq:finite:covariance} for continuous bounded functions. We consider Lipschitz-continuous bounded functions $f,f^\prime : \tilde{\ell}^p \to \mathbb{R}$ in $\mathcal{G}_{+}(\tilde{\ell^p})$. The extension to continuous bounded functions then holds following a Portmanteau argument.
Now notice
\beao
\lefteqn{\frac{m_n}{k_n}\E[f(x_{b_n}^{-1}\mathcal{B}_1) f^\prime(x_{b_n}^{-1} \mathcal{B}_{t })]} \\
&=& \frac{m_n}{k_n} \E[ f(\underline{x_{b_n}^{-1}\mathcal{B}_1}_\epsilon)f^\prime(x_{b_n}^{-1} \mathcal{B}_{t })] + \frac{m_n}{k_n}\E[(f(x_{b_n}^{-1}\mathcal{B}_1) - f(\underline{x_{b_n}^{-1}\mathcal{B}_1}_\epsilon)) f^\prime(x_{b_n}^{-1} \mathcal{B}_{t })] 
\eeao
For the second term, we rely on condition ${\bf CS}_p$ since $f,f^\prime$ are bounded Lipschitz-continuous functions. In this case, the second term is negligeable letting first $n \to \infty$ and then $\epsilon \downarrow 0$. Similarly, we deduce from condition ${\bf CS}_p$ that for $t = 2,3,\dots$, and for all $\epsilon >0$, 
\beam\label{eq:f:to:epsilon}
        \frac{m_n}{k_n}\E[f(x_{b_n}^{-1}\mathcal{B}_1) f^\prime(x_{b_n}^{-1} \mathcal{B}_{t })]
        &\sim& \frac{m_n}{k_n}\E[f(\underline{x_{b_n}^{-1}\mathcal{B}_1}_\epsilon) f^\prime(\underline{x_{b_n}^{-1} \mathcal{B}_{t }}_\epsilon )],
\eeam
if we let $n \to \infty$, and then $\epsilon \downarrow 0$. Thus it suffices to show that
\beao
    \lim_{n \to \infty}\frac{m_n}{k_n}\Cov(f_\epsilon(x_{b_n}^{-1}\mathcal{B}_{1}),f^\prime_\epsilon(x_{b_n}^{-1}\mathcal{B}_{t} )) 
    &=& 0\,,\qquad t=2,3,\dots\,.
\eeao 
 Consider the sequence $(\ell_n)$ satisfying the condition of Lemma~\ref{lem:covariance:1}, and recall the notation $\mathcal{B}_{t,\ell} = \bfX_{(t-1)b + [1,b-\ell]}$. We can use similar steps as in the proof of Lemma~\ref{lem:mixing} and replace $\mathcal{B}_1$ by $\mathcal{B}_{1,\ell}$ inside the covariance term. For this step we require
 $\ell_n/b_n \to 0$ as $n \to \infty$. Moreover, 
\beao
    \frac{m_n}{k_n}\Cov(f_\epsilon(x_{b_n}^{-1}\mathcal{B}_{1,\ell}),f^\prime_\epsilon(x_{b_n}^{-1}\mathcal{B}_{t} )) 
    &\le& 2 \|f\|_\infty \| f^\prime \|_\infty \frac{m_n}{k_n} \beta_{\ell +  (t -2)b}.
\eeao 
{
Finally, we have $m_n\beta_{\ell_n}/k_n \to 0$, as $n\to\infty$.
In the case where $g$ or both $g,h$, equal one of the functionals defining the Hill and exceedances estimator, i.e., $h, e,$ in \eqref{eq:func:hill} and \eqref{eq:func:exceedances}, the proof follows similar steps and thus we omit the details.
In this case we use \eqref{eq:interm:step} instead of \eqref{eq:limit:cond} in the first step of the proof. 
}
\qed 


\subsection{Proof of Lemma~\ref{lem:pnorm:VCclass}}\label{sec:proof:lem:pnorm:VCclass} 
An application of Riesz-Thorin Theorem implies that for every fixed $\bfx \in \tilde{\ell}^{q_0}$, the map 
\beam  \label{eq:log:q:nrom}
1/q & \mapsto & \log \|\bfx\|_q, \quad q \in (q_0,q_0^\prime).
\eeam 
is convex. 
Denote $\bfx \mapsto \psi_{q}(\bfx)$ the derivative of the function $1/q \mapsto \log \|\bfx\|_q$ which is defined by
\beao 
\psi_{q}(\bfx)
&=& \log(\|\bfx\|_{q}^{q}) - \sum_{t\in \mathbb{Z}} \frac{|\bfx_t|^{q}\log(|\bfx_t|^{q})}{\|\bfx\|_{q}^{q}} \\
&=& \sum_{t\in \mathbb{Z}} \frac{|\bfx_t|^{q}\log(\|\bfx\|_{q}^{q}/|\bfx_t|^{q})}{\|\bfx\|_{q}^{q}},
\eeao 
which corresponds to the derivative of the log $\ell^q-$norm function as in \eqref{eq:log:q:nrom}. 
It is easy to see that $\psi_q(\bfx) \ge 0$.
Moreover, $1/q \mapsto \psi_q(\bfx)$ is a non-decreasing function. 
To verify this, it is enough to compute the derivative of the function $1/q \mapsto \psi_q(\bfx)$. 
Moreover, the convexity of the mapping in~\eqref{eq:log:q:nrom} implies, for all $q_1, q_2 \in (q_0,q_0^\prime)$,
\beao 
\log \|\bfx\|_{q_1} & \ge &
\log \|\bfx\|_{q_2} + ( q_1^{-1} - {q_2}^{-1})\psi_{q_2}(\bfx).
\eeao 
Then, rewriting this previous relation yields
\beam \label{eq:convexity:lp:norm}
\|\bfx\|_{q_1} & \ge &
\|\bfx\|_{q_2}\exp\{ ({q_1}^{-1} - {q_2}^{-1})\psi_{q_2}(\bfx)\}.
\eeam 
\par  
It is easy to see that the VC-dimension of our class is larger than two.  
For example, we can take the point with $x^1_t = 1$, only if $t = 0$, and the point $x^2_t = 1/m$, only if $t = 0,1$, and $2^{1/q_0} < m < 2^{1/q_0\prime}$. 
Then, we check readily that $\|\bfx^1\|_{q} = 1$ for all $q \in \mathbb{R}$, and $ \|\bfx^2\|_{q_0} > 1 > \|\bfx^2\|_{q_1}$.
Therefore, we conclude that our class of sets separates these two points. 
\par 
We now show that our class of sets can't shatter three different points. 
Consider the points $\bfx^1$, $\bfx^2,$ and $\bfx^3$ with values in $\tilde{\ell}^{q_0}$.
Assume that there exists $q_1,q_2 \in [q_0,q_0^\prime]$ such that $\|\bfx^1\|_{q_1} > \|\bfx^2\|_{q_1}$ and $\|\bfx^2\|_{q_2} > \|\bfx^1\|_{q_2}$. 
This means our class picks out the sets $\{\bfx^1\}$ and $\{\bfx^2\}$. 
Assume also, without loss of generality, $q_1 < q_2$. 
Actually, this means that for all $q$ such that $q_1 < q_2 \le q$, then $\|\bfx^2\|_q > \|\bfx^1\|_q$. 
Indeed, for $q \ge q_2$ Equation~\eqref{eq:convexity:lp:norm} implies 
\beao 
\|\bfx^1\|_{q} & \ge &
\|\bfx^1\|_{q_1}\exp\{ ({q}^{-1} - {q_1}^{-1})\psi_{q_1}(\bfx^1)\} \\
&\ge& \|\bfx^2\|_{q_1}\exp\{ ({q}^{-1} - {q_1}^{-1})\psi_{q_1}(\bfx^1\} \\
&\ge& \|\bfx^2\|_{q}\exp\{ ({q}^{-1} - {q_1}^{-1})(\psi_{q_1}(\bfx^1) -\psi_{q}(\bfx^2))  \}\\
&\ge& \|\bfx^2\|_{q}\exp\{ ({q}^{-1} - {q_1}^{-1})(\psi_{q_1}(\bfx^1) -\psi_{q_2}(\bfx^2))  \},
\eeao 
where the last equality holds since $1/q \mapsto \psi_{q}$ is a non-decreasing function and $q_1<q_2 \le  q$.
In particular, letting $q = q_2$ then the relation $\|\bfx^2\|_{q_2} > \|\bfx^1\|_{q_2}$ implies 
\beao 
(\psi_{q_1}(\bfx^1) -\psi_{q_2}(\bfx^2)) & > & 0, 
\eeao 
Moreover, we also have
\beao 
\|\bfx^2\|_{q} & \ge &
\|\bfx^2\|_{q_2}\exp\{ ({q}^{-1} - {q_2}^{-1})\psi_{q_2}(\bfx^2)\} \\
&\ge& \|\bfx^1\|_{q_2}\exp\{ ({q}^{-1} - {q_2}^{-1})\psi_{q_2}(\bfx^2)\} \\
&\ge& \|\bfx^1\|_{q}\exp\{ ({q}^{-1} - {q_2}^{-1})(\psi_{q_2}(\bfx^2) -\psi_{q}(\bfx^1))  \}\\
&\ge& \|\bfx^1\|_{q}\exp\{ ({q}^{-1} - {q_2}^{-1})(\psi_{q_2}(\bfx^2) -\psi_{q_1}(\bfx^1))  \},
\eeao 
but this implies $\|\bfx^2\|_q > \|\bfx^1\|_q$, for all $q$ satisfying $q_1 < q_2 \le q$. Similarly $\|\bfx^2\|_q < \|\bfx^1\|_q$, for all $q$ satisfying $q\le q_1 < q_2$.
Assume we can shatter the points $\bfx^1$, $\bfx^2,$ and $\bfx^3$. Then there exist $q_i>0$ satisfying
\beao 
\|\bfx^i\|_{q_i} > \max\{\|\bfx^j\|_{q_i},\|\bfx^l\|_{q_i}, l,j \not = i \},
\eeao 
for every permutation $ \{i,l,j\} = \{1,2,3\}$.
Then, eventually renaming the points, we can suppose without loss of generality $q_1 < q_2 < q_3$. 
In this case, we claim we can't pick out the set $\{\bfx^1,\bfx^3\}$. Indeed, for $q\ge q_2 >q_1$ then $\|\bfx^1\|_q<\|\bfx^2\|_q$ and for $q\le q_2 <q_3$ then $\|\bfx^2\|_q<\|\bfx^3\|_q$.

{\black
\section{Appendix on variance calculations}\label{sec:appendix:E}


\subsection{Proof of Lemma~\ref{lem:Hill:limit} }\label{sec:lem:hill}
Uniform convergence of the independent-block estimators towards the Gaussian process entails
\beao 
\sqrt{k_n}
\big( (\widetilde{h}^{\bfQ}(u ),
\widetilde e^{\bfQ} (1 )) 
- (u^{-\a} h^{\bfQ},1)  \big) 
&\xrightarrow[]{ d }& \,  
\big(
\mathbb{G}( h(\cdot/u),\alpha) ,\mathbb{G}(e(\cdot/1 ),\alpha)
\big),
\eeao
as $n \to \infty$. 
Now notice the identity
\beao 
\sqrt{k_n} \big(\,|\bfX/x^\prime_{b_n}|_{(\lfloor k^\prime u \rfloor)} - u^{-1/\alpha}\big)= \sqrt{k_n}( \widetilde e^{\bfQ}(u)^{\leftarrow} - (u^{-\alpha})^{\leftarrow} ) , \quad u \in [1-\epsilon, 1+\epsilon],  
 \eeao
where $|\bfX/x^\prime_{b_n}|_{(1)} \ge|\bfX/x^\prime_{b_n}|_{(2)}  \ge  \cdots \ge |\bfX/x^\prime_{b_n}|_{(n)}$, are the order statistics of the sample $(|\bfX_t/x^\prime_{b_n}|)$. Then, by an application of Vervaat's lemma,
\beao 
\sqrt{k_n} \big(\,|\bfX/x^\prime_{b_n}|_{(\lfloor k^\prime \rfloor)}  - 1\big) &\xrightarrow[]{ d }& -\, \alpha^{-1} \mathbb{G} \big(e(\cdot/1),\alpha \big),
\eeao 
as $n \to  \infty$, in particular, $|\bfX/x'_{b_n} |_{(\lfloor k^\prime  \rfloor)}\xrightarrow[]{\P}1$ as $\nto$. 
Furthermore, denoting $1/\widehat \alpha$ the Hill estimator in Equation~\eqref{eq:hill:estimator},
\beao
\sqrt{k_n}\, \Big(\,\frac{1}{\widehat \a} - \frac{1}{\alpha} \,\Big)
&=&
\sqrt{k_n}\, \alpha^{-1} \,  \big(\,\alpha \widetilde h^\bfQ(\,  
	|\bfX|_{\lfloor k^\prime \rfloor }/x^\prime_{b_n} \, ) -1\,\big) \nonumber
\\
 &= & \sqrt{k_n}\, \alpha^{-1} \, \big(\,\alpha \widetilde h^\bfQ(|\bfX/x^\prime_{b_n}|_{\lfloor k^\prime \rfloor }) - (|\bfX/x^\prime_{b_n}|_{\lfloor k^\prime \rfloor })^{-\a}\, \big) \,  \nonumber \\
&&+  \, \sqrt{k_n}\, \alpha^{-1} \, \big(\,(|\bfX|_{\lfloor k^\prime \rfloor }/x^\prime_{b_n}|)^{-\a} - 1\big) \nonumber \\
& \xrightarrow[]{d} & \alpha^{-1} \, \mathbb{G} (\alpha h(\cdot/1) - e(\cdot/1), \alpha ), \qquad n \to \infty. 
\eeao 	
Finally, an application of the Delta-method yields Equation~\eqref{eq:normality:hill} and allow us to conclude.\qed

\subsection{Proof of Lemma~\ref{lem:random:threshold}}\label{sec:proof:lem:random:threshold}

The following Lemma will be useful to proof Lemma~\ref{lem:random:threshold}. We defer its proof to the end of this Section.

\begin{lemma}[Asymptotics of $c(\widehat \a)$]\label{lem:continuity:c(p):2}
Assume the conditions of Theorem~\ref{thm:main} are satisfied and $k_n/k_n^\prime \to \kappa$,  for $\kappa >  0$, then
\beao 
\sqrt{k_n}( c(\widehat \a ) - 1) & \xrightarrow[]{ d} &  - \a \psi_\a^{\bfQ}\, \mathbb{G}(h(\cdot/1) - \alpha^{-1} e(\cdot/1), \alpha ), \quad n\to \infty,
\eeao 
{
and, for $\kappa = 0$, $\sqrt{k_n}( c(\widehat \a ) - 1) \xrightarrow[]{\P} 0$, as $n \to \infty$. 
Here 
\beam\label{eq:osu}
\psi_\alpha^\bfQ &=&\E\Big[ \sum_{t \in \mathbb{Z}} |\bfQ_t|^{\a }
\log\big( \tfrac{\|\bfQ\|_\a }{|\bfQ_t|}\big)   \Big], 
\eeam 
and $\psi_\a^{\bfQ} < \infty$ due to condition ${\bf M}$.   }
\end{lemma}

Under the assumptions of Theorem~\ref{thm:main} we have $\widehat \alpha \xrightarrow[]{\P}\alpha$, as $n \to \infty$, from an application of Lemma~\ref{lem:Hill:limit}.
Moreover,  the uniform central limit theorem on the process $\wt g^\bfQ(u,q)$ entails 
\beao 
\sqrt{ k_n} (\widetilde {1}^\bfQ(u,\widehat \a) - u^{-\a}c(\widehat \a) ) &\xrightarrow[]{d}& \mathbb G\big( 1(\cdot/u),\alpha), \quad n \to \infty.
\eeao 
{Hence, if $k_n/k_n^\prime \to \kappa$ and $\kappa > 0$, an application of Lemma~\ref{lem:continuity:c(p):2} implies}
\beao 
\lefteqn{ \sqrt{ k_n} (\widetilde {1}^\bfQ(u,\widehat \a) - u^{-\a} )}  \\
    &=& \sqrt{ k_n} (\widetilde {1}^\bfQ(u,\widehat \a) - u^{-\a}c(\widehat \a) ) - u^{-\a} \sqrt{ k_n} (1 - c(\widehat \a ) ) \\
&\xrightarrow[]{d}& \mathbb G\big( 1(\cdot/u) -  u^{-\a} \a \psi_\a^\bfQ ( h(\cdot/1) - \a^{-1} e(\cdot/1) ) \,, \alpha\big)\,,\qquad n\to\infty\,.
\eeao 
{
Instead, if $k/k^\prime \to 0$, we conclude 
\beao 
\sqrt{ k_n} (\widetilde {1}^\bfQ(u,\widehat \a) - u^{-\a} ) &\xrightarrow[]{d} \mathbb{G}\big( 1(\cdot/u), \alpha )\,,\qquad n\to\infty\,.
\eeao 
}
In addition, recall the sequence $(x_{b_n})$ satisfies the assumptions of Proposition~\ref{prop:existence:cluster:process}, and $(k_n)$ satisfies $k = k(\alpha) \sim b_n \P(|\bfX_1| > x_{b_n})$, as $n \to \infty$.  
We derive the relation, as $n \to \infty$, 
\beao 
 \sqrt{k_n}( \widetilde{1^{\bfQ}}(u , \widehat \a )^{\leftarrow} 
 - (u^{-\alpha})^{\leftarrow})  = \sqrt{k_n} \big(\,\|{\mathcal B}_1/x_{b_n}\|_{\widehat \alpha,(\lfloor k u \rfloor)} - u^{-1/\alpha} \big)+o_\P(1), 
 \eeao
for $ u \in [1-\epsilon, 1+\epsilon]$. Then, by an application of Vervaat's lemma,
\beao 
\sqrt{k_n} \big(\,\|{\mathcal B}_1/x_{b_n}\|_{ \widehat \alpha ,(\lfloor k  \rfloor)} - 1\big)
&\xrightarrow[]{ d }& -\, \alpha^{-1} \mathbb{G} \big (1(\cdot/1)) - \a \psi_\a^{\bfQ}(\cdot/1) - \a^{-1} e(\cdot/1)), \alpha \big),
\eeao 
as $n \to  \infty$.
{
These calculations hold in the case $k_n/k_n^\prime \to \kappa$ and $\kappa > 0$. Instead, if $k_n/k_n^\prime \to 0$ a similar argument allow as to conclude
\beao 
\sqrt{k_n} \big(\,\|{\mathcal B}_1/x_{b_n}\|_{ \widehat \alpha ,(\lfloor k  \rfloor)} - 1\big) &\xrightarrow[]{d}& - \alpha^{-1} \mathbb{G}(1(\cdot/1), \alpha ), \quad  n \to \infty.
\eeao 
In particular, $\|{\mathcal B}_1\|_{\widehat \alpha ,(\lfloor k\rfloor)}/x_{b_n} \xrightarrow[]{\P}1$ for every $\kappa \ge  0$.
}
\par 
Moreover, since $f_\a$ is a continuous non-increasing function, and $c(\widehat{\a}) \xrightarrow[]{\P} 1$, similar calculations entail 
\beam \label{eq:f:alpha:hat}
\lefteqn{ \sqrt{k_n}(\,\widetilde{ f_\a^\bfQ}(u, \widehat{\a}) - u^{-\a}f_\a^\bfQ \,) }\nonumber \\
    &\xrightarrow[]{d}& \mathbb{G} \big(f_\a(\cdot/u)  - f_\a^{\bfQ}  u^{-\a} \a \psi_\a^{\bfQ} \,( h(\cdot/1) - \a^{-1} e(\cdot/1))\,, \alpha \big) .
\eeam  
Overall, we conclude
\beao 
\lefteqn{ \sqrt{k_n}\, \big(\,\widehat{f_\a^\bfQ}(\widehat \a) - f^\bfQ_\a\,\big)}\\
&=&
 \sqrt{k_n}\, 
 \big(\,
 \widetilde{f_\a^\bfQ}(\, 
{ \|{\mathcal B}_1\|_{\widehat{\alpha},(\lfloor  k \rfloor )}}/{x_{b_n} \, }, \widehat \a )
  - f^\bfQ_\a\, 1 \,\big) 
\\
&=& \sqrt{k_n}\, \, 
\big(\,
 \widetilde{f_{\a}^\bfQ}(\, 
{ \|{\mathcal B}_1\|_{\widehat{\alpha},(\lfloor  k \rfloor )}}/{x_{b_n} \, }, \widehat \a )
 - 
f_\a^\bfQ \, (\|{\mathcal B}_1\|_{{\widehat{\alpha}},(\lfloor k\rfloor)}/x_{b_n})^{-\a}\, \big) \,\\
&&+  \sqrt{k_n}\, f_\a^\bfQ \, 
\big(\,
 (\|{\mathcal B}_1\|_{{\widehat{\alpha}},(\lfloor k\rfloor)}/x_{b_n})^{-\a}
 - 1 \, \big) \, \\
 &\xrightarrow[]{d}&   \mathbb G( \, f_\a(\cdot/1) - f_\a^\bfQ\,1(\cdot/1)\,,\alpha)\,,\qquad n\to \infty\,,
\eeao
and this yields the desired result. 
\qed

\subsection{Proof of Lemma~\ref{lem:continuity:c(p):2}}\label{sec:lem:continuity:cp:2}
{
To study the function $q \mapsto c(q)$ in a neighborhood of $\alpha$ we write it as
\beam  \label{eq:cq:rewritting}
c(q) &=& \E[ \|\bfQ\|_q^\alpha ] 
  \; = \; (\E[ 1/ \|\bfQ^{(q)}\|_\alpha^\alpha])^{-1}.
\eeam 
Borrowing the Taylor expansions from Remark~\ref{remark:check:S:ei} we have
\beao
\lefteqn{ 
(1/c(q) - 1)}\nonumber\\
&=&  \E[\|\bfQ^{(q)}\|_q^q - \|\bfQ^{(q)}\|_\a^\a/ \|\bfQ^{(q)}\|_\a^\a] \nonumber \\
&=& (q-\a) \E\Big[  \sum_{t \in \mathbb{Z}} |\bfQ^{(q)}_t|^{\a }
\log\big( 1/{|\bfQ^{(q)}_t|}\big)/\| \bfQ^{(q)}\|_\alpha^\alpha\Big]  + \frac{1}{2} (q-\a)^2 \nonumber  \\
    && \times
    \E\Big[ 
 \frac{1}{\|\bfQ^{(q)}\|_{q^\prime}^{q^\prime}}\sum_{t \in \mathbb{Z}} \sum_{j \in \mathbb{Z}}  
\frac{|\bfQ^{(q)}_t|^{q^\prime}}
{\|\bfQ^{(q)}\|_{q^\prime}^{q^\prime }} 
\frac{|\bfQ^{(q)}_t|^{q^\prime}}
{\|\bfQ^{(q)}\|_{q^\prime}^{q^\prime }} 
\log\big(\frac{1}{|\bfQ^{(q)}_t|}\big)\log\big(\frac{|\bfQ^{(q)}_t|}{|\bfQ^{(q)}_j|}\big)
\Big], \nonumber 
\eeao 
for $q^\prime \in [q \lor \alpha, q \land  \alpha], q \in [\alpha - \epsilon, \alpha + \epsilon],$ and $\epsilon > 0$ such that ${\bf M}$ holds.
Moreover, an application of the change-of-norms from Equation~\eqref{eq:change:of:norms} yields
\beam\label{eq:continuity:q}
\lefteqn{(1/c(q) - 1)}\nonumber\\ 
&=& (q-\a) \E\Big[ \underbrace{ \|\bfQ\|_q^{\a} \sum_{t \in \mathbb{Z}} |\bfQ_t|^{\a }
\log\big( \tfrac{\|\bfQ\|_q }{|\bfQ_t|}\big)}_{:= \psi_{q,\a}(\bfQ) }  \Big]/c(q)  + \frac{1}{2} (q-\a)^2 \nonumber  \\
    && \times
    \E\Big[ \underbrace{ 
\|\bfQ\|_q^{\a} \frac{\|\bfQ\|_{q}^{q^\prime}}{\|\bfQ\|_{q^\prime}^{q^\prime}}\sum_{t \in \mathbb{Z}} \sum_{j \in \mathbb{Z}}  
\frac{|\bfQ_t|^{q^\prime}}
{\|\bfQ\|_{q^\prime}^{q^\prime }} 
\frac{|\bfQ_t|^{q^\prime}}
{\|\bfQ\|_{q^\prime}^{q^\prime }} 
\log\big(\frac{\|\bfQ\|_q}{|\bfQ_t|}\big)\log\big(\frac{|\bfQ_t|}{|\bfQ_j|}\big)
}_{:= \psi^\prime_{q,q^\prime} (\bfQ)  }\Big]/c(q) \nonumber \\
&=& (q-\a)\E[ \psi_{q,\a}(\bfQ)  ] /c(q) + \frac{1}{2} (q-\a)^2 \E[\psi^\prime_{q,q^\prime}(\bfQ) ]/c(q).\nonumber\\
\eeam 
We start by noting that by an application of Equation~\eqref{eq:bound:logarithm} we can derive the bound
\beao 
0 \;<\;  \E[ \sup_{q \in [\alpha - \epsilon, \alpha + \epsilon]} \psi_{q,\a}(\bfQ)  ] &\le& \epsilon^{-1} \,\E[\|\bfQ\|_{\a - \epsilon}^{2\a}] \;<\;  \infty,
\eeao
and this last term is finite due to Assumption {\bf M}.}{
\par 
In what follows we verify that  
$\E[\sup_{q,q^\prime}|\psi^\prime_{q,q^\prime}|]<\infty$ for $q$ and $q^\prime$ close to $\a$. 
First note that $|\psi^\prime_{q,q^\prime}|$ is non-increasing in $q$. Thus it is enough to check that 
    \beao
 \E\Big[ 
      \sup_{ \substack{  q \in [\a - \epsilon, \a + \epsilon]  } }
 \frac{\|\bfQ\|_{\a-\epsilon}^{q+\a}}{\|\bfQ\|_{q}^{q}}
 \Big| \sum_{t \in \mathbb{Z}} \sum_{j \in \mathbb{Z}}  
\frac{|\bfQ_t|^{q}}
{\|\bfQ\|_{q}^{q}} 
\frac{|\bfQ_j|^{q}}
{\|\bfQ\|_{q}^{q }} 
\log\Big(\frac{\|\bfQ\|_{\a - \epsilon}}{|\bfQ_t|}\Big)\log\Big(\frac{|\bfQ_t|}{|\bfQ_j|}\Big)\Big|
\Big] < \infty, 
\eeao
for $\vep>0$ sufficiently small. Moreover H\"older's inequality provides
\beao
\sum_{t=-\infty}^\infty |\bfQ_t|^{\a-\vep}&=&\sum_{t=-\infty}^\infty|\bfQ_t|^{(\a-\vep)q/(q+\a)}|\bfQ_t|^{(\a-\vep)\a/(q+\a)}\\
&\le & \Big(\sum_{t=-\infty}^\infty|\bfQ_t|^q\Big)^{(\a-\vep)/(q+\a)}\Big(\sum_{t=-\infty}^\infty|\bfQ_t|^{(\a-\vep)\a/(q+\vep)}\Big)^{(q+\vep)/(q+\a)}\,.
\eeao
Thus $\|\bfQ\|_{\a-\vep}^{q+\a}\le \|\bfQ\|_q^q\|\bfQ\|_{(\a-\vep)\a/(q+\vep)}^\a$.
For every chosen $0<\delta<1$ we have from Equation~\eqref{eq:bound:logarithm},
and 
assuming that every $|\bfQ_t| \le 1$ a.s., 
\beao
\sum_{t \in \mathbb{Z}}  
\frac{|\bfQ_t|^{q}}
{\|\bfQ\|_{q}^{q}} 
|
\log(|\bfQ_t|)|^2&\le& \sum_{t \in \mathbb{Z}}  
\frac{|\bfQ_t|^{q-\delta}}
{\|\bfQ\|_{q}^{q}} \\
&\le& \frac{\big(\sum_{t \in \mathbb{Z}}  
|\bfQ_t|^{q}\big)^{q/(q-\delta)}}
{\|\bfQ\|_{q}^{q}}\\
&\le& \Big(\sum_{t \in \mathbb{Z}}  
|\bfQ_t|^{q}\Big)^{\delta/(q-\delta)}\,.
\eeao
Squaring both sides of the last inequality we obtain an upper bound \beao
\Big(\sum_{t \in \mathbb{Z}}  
\frac{|\bfQ_t|^{q}}
{\|\bfQ\|_{q}^{q}} 
|
\log(|\bfQ_t|)|^2\Big)^2&\le& \Big(\sum_{t \in \mathbb{Z}}  
|\bfQ_t|^{q}\Big)^{2\delta/(q-\delta)}\\
&\le& \|\bfQ\|_q^{q2\delta/(q-\delta)}\\
&\le& \|\bfQ\|_{\a-\delta}^{(\a+\vep)2\delta/(\a-\vep-\delta)}\,.
\eeao 
This upper bound is also valid for $|\log(\|\bfQ\|_q)|$, $q-\vep<q+\vep$, up to constant. 

Combining these inequalities we conclude the sufficient condition 
\beao
\E\big[\|\bfQ\|_{(\a-\vep)\a/(q+\vep)}^\a\|\bfQ\|_{\a-\delta}^{(\a+\vep)2\delta/(\a-\vep-\delta)}\big]<\infty\,.
\eeao
Since $\vep$ and $\delta$ can be taken arbitrarily small this allows to conclude 
\beao 
0 \;<\;  \E[ \sup_{q,q^\prime \in [\alpha - \epsilon, \alpha + \epsilon]} \psi^\prime_{q,q^\prime}(\bfQ)  ] &\le& \E[\|\bfQ\|_{\a - \epsilon}^{\a+\epsilon}] \;<\;  \infty,
\eeao
where the last equality follows again from {\bf M} choosing $\epsilon$ sufficiently small. } 
\par

{

Now let $k,k^\prime$ be such that $k/k^\prime \to \kappa$, $\nto$, for $\kappa > 0$.
Recall the asymptotic development of $q \mapsto c(q)$ and the definition of $\psi_{q,\a}$ in \eqref{eq:continuity:q}.
We write
\beao
\psi_\alpha^\bfQ &=& \E[ \psi_{ \alpha,\alpha}(\bfQ) ] \;=\; \E\Big[  \sum_{t \in \mathbb{Z}} |\bfQ_t|^{\a }
\log\big( \tfrac{\|\bfQ\|_\a }{|\bfQ_t|}\big)   \Big], 
\eeao
which corresponds to the constant in \eqref{eq:osu}.
Moreover, notice $q \mapsto \E[\psi_{q,\a}(\bfQ)]$ is a continuous function at $\a$ under condition ${\bf M }$.
Therefore, applying~\eqref{eq:normality:hill} and \eqref{eq:continuity:q} we obtain 
\beao 
\sqrt{k_n}(1 - c(\widehat \a )) & \xrightarrow[]{ d} &  \a \psi_\a^{\bfQ}\, \mathbb{G}(h(\cdot/1) - \alpha^{-1} e(\cdot/1), \alpha ),
\eeao 
  We notice in particular $c(\widehat \a) \xrightarrow[]{\P} 1 $, as $n \to \infty$.
\par 
In the case where $k/k^\prime \to 0$ we omit the details here as this mimics the proof provided in Section~\ref{proof:thm:main}. First restricting the family $\mathcal F$ to $\mathcal F_h$ following the notation in \eqref{eq:t}, we obtain
\beao 
\sqrt{k^\prime}(1 - c(\widehat \a )) & \xrightarrow[]{ d} &  \a \psi_\a^{\bfQ}\, \mathcal{N}(0, \E[(h(Y\bfQ) - \alpha^{-1} e(Y\bfQ))^2] )\,, \quad n \to \infty.
\eeao 
This last equation implies
$\sqrt{k_n}(1 - c(\widehat \a )) \xrightarrow[]{\P} 0,$ as $n \to \infty$.
Finally, this concludes the proof of Lemma~\ref{lem:continuity:c(p):2}. \qed


\section{Proofs of the results of Section \ref{sec:examples}}\label{section:models}
\color{black} 
\subsection{Proof of Proposition~\ref{lem:CS:m0}}\label{proof:lem:cs:m0}
From the discussion in Section~\ref{sec:example:m0}, we can see that all assumptions in Theorem~\ref{thm:main}  are satisfied. 
Note that if $p=\alpha$, then $|\bfQ|$ has a deterministic expression in the shift-invariant space. 
 Moreover, $\bfQ$ has at most $m_0$ non-zero coordinates. 
Thereby, condition $\bf M$ is satisfied.
Moreover, the index estimators in \eqref{eq:estimator:Ei}, \eqref{eq:estimator:c1}, and \eqref{eq:cluster:sizes}, with $f_\alpha : \bfx \mapsto \|\bfx\|^\alpha_\infty/\|\bfx\|_\alpha^\alpha$, $f_\alpha: \bfx \mapsto \|\bfx\|_1^\alpha/\|\bfx\|_\alpha^\alpha$, and $f_\alpha: \bfx \mapsto (|\bfx|^\alpha_{(j)} - |\bfx|^\alpha_{(j+1)})/\|\bfx\|_\alpha^\alpha$, respectively, satisfy $\Var(f_\alpha(Y\bfQ)) = 0$.
In addition, appealing to Remark~\ref{remark:check:S:ei}, $\partial f_q/\partial q |_{q=\a}$ and $ \sup_{q \in [\alpha - \epsilon, \alpha + \epsilon]}\partial^2f_q/\partial q^2|_{q=q^\prime}$ are bounded continuous functions in $\mathcal{G}_+(\ell^\alpha)$, and this proves $\bf S$ holds.   
 Finally, using the change of norms formula in \eqref{eq:change:of:norms}, we can also show $\Var(f_\a(Y\bfQ^{(p)})) = 0$, for any $p \in (0,\infty]$, and this concludes the proof. 
\qed

\subsection{Proof of Proposition~\ref{lem:linear:process}}\label{sec:models:ca:cs}

We start by noticing that Equation \eqref{eq:truncation:k:linear:1} rewrites as: for all $\delta > 0$,
\beam \label{eq:truncation:k:linear}
\lim_{s \to  \infty} \limsup_{n \to  \infty} \frac{\P( \mbox{$ \sum_{t=1}^n | \sum_{|j| > s} \varphi_j \bfZ_{t-j}/x_n|^p > \delta $} )}{ n \P(|\bfX_1| > x) } = 0.
\eeam 
Assuming \eqref{eq:truncation:k:linear:1} holds, {\bf AC} and ${\bf CS}_p$ follow straightforwardly since for all $s > 0$, the series $(\bfX_t^{(s)})$ is a linear $m_0-$dependent sequence with $m_0 = 2 s +1$, such that $\bfX_t^{(s)} = \sum_{|j| \le s} \varphi_j \bfZ_{t-j}$. The former satisfies ${\bf AC}, {\bf CS}_p$, for $p > \alpha/2$, as in Example~\ref{sec:example:m0}.

\par 
We now turn to the verification of Equation \eqref{eq:truncation:k:linear}. Actually, by monotonicity of the $\ell^p-$norms, if \eqref{eq:truncation:k:linear} holds for $ \alpha/2 < p < \alpha$, then it also holds for $p \ge \alpha$. In the following we assume $ \alpha/2 < p < \alpha$.

For $p \le 1$, the subadditivity property yields
\beao 
| \mbox{$\sum_{|j| > s}$} \varphi_j \bfZ_{t-j}|^p   &\le& \mbox{$\sum_{|j| > s}$} |\varphi_j \bfZ_{t-j}|^p  \;=:\; I_{1,t}.
\eeao 
 That the partial sums of $(I_{1,t})$ satisfy \eqref{eq:truncation:k:linear} for $0<\a<1$ follows from standard arguments, see for instance Section 6.1 of \cite{hult:samorodnistky:2010}. We provide an alternative prove below, also valid for every $\a>0$.

For $p > 1$, a Taylor decomposition of functional $|\cdot|^p:\mathbb{R} \to \mathbb{R}$ entails, for all $a,b, \in \mathbb{R}$,
\beao
 |a+b|^p
    &=& |a|^p + p\, \sign(a)|a|^{p-1}b + \frac{p(p-1)}{2}\,|a|^{p-2}b^2  + \cdots + R_{[p]}(a,b),
\eeao 
where the remaining term satisfies
\beao 
R_{[p]} \; = \; R_{[p]}(a,b) & \le & \frac{p(p-1) \cdots (p-[p] )}{[p]!}|b - \xi a|^{p-[p] } b^{[p]},
\eeao 
for one $\xi \in [0,1]$. To simplify notation, in the remaining lines of the proof we denote $(|\bfZ_t|)$ by $(Z_t)$. Then, the Taylor expression above yields
\beao 
\lefteqn{|\mbox{$\sum_{|j| > s }$} \varphi_j \bfZ_{t-j}/x_n |^p}\\
&\le&
|\varphi_s Z_{t-s}/x_n|^p  + \; p\, |\varphi_s Z_{t-s}/x_n|^{p-1}   (\mbox{$\sum_{\substack{|j| \ge s\\
j\not = s}}$} |\varphi_j Z_{t-j}/x_n|) \\
&& + \cdots + R_{[p]}.
\eeao 

Moreover, to handle the remaining term $R_{[p]}$, we use subadditivity of the real function $x \mapsto x^{p - [p]}$. Hence, 
\beao 
\lefteqn{|\mbox{$\sum_{|j| \ge s}$} \varphi_j \bfZ_{t-j}/x_n|^p }\\
&\le& |\varphi_s Z_{t-s}/x_n|^p + c |\varphi_s Z_{t-s}/x_n|^{p-1}      \big( \mbox{$\sum_{\substack{|j| \ge s\\
    j\not = s}}$} |\varphi_j| Z_{t-j}/x_n \big) \\
&& + \cdots +  c  |\varphi_s Z_{t-s}/x_n|^{p-[p]} \big( \mbox{$\sum_{\substack{|j| \ge s \\ j\not = s}}$} |\varphi_j| Z_{t-j}/x_n \big)^{[p]} \\
&& + c\, \mbox{$\sum_{ \substack{|j| \ge s \\ j \not = s}}$} | |\varphi_j| Z_{t-j}/x_n|^{p-[p]}   \big( \mbox{$\sum_{ \substack{|j| \ge s \\ j \not = s}}$} |\varphi_j| Z_{t-j}/x_n  \big)^{[p]} \\
&\leq & c\,( I_{0,t} + I_{1,t} + \cdots + I_{[p],t} + I_{[p] + 1,t})\,,
\eeao 
where $c>0$ is a constant of no interest, only depending on $p$. Relying on the bound above, we require to control the previous $[p]+2$ terms.
We argue using a truncation argument. Our goal is to prove that for all $l = 0, \dots, [p] +1 $, for all $\epsilon, \delta >0$
\begin{align}\label{eq:sum:terms2}
    \lim_{s \to  \infty}\limsup_{n \to  \infty} \frac{\P( \sum_{t=1}^n \overline{I_{l,t}}^\epsilon > \delta ) }{ n \P(|\bfX_1| > x_n) } +  \frac{\P( \sum_{t=1}^n \underline{I_{l,t}}_\epsilon > \delta ) }{ n \P(|\bfX_1| > x_n)}\; =\; 0,
\end{align}
where the truncated terms are defined as follows: for $l = 0, \dots, [p] $
\beao 
\overline{I_{l,t}}^\epsilon 
    &:=& |\overline{ \varphi_s Z_{t-s}/x_n}^\epsilon |^{p - l}  \big(  \mbox{$\sum_{\substack{|j| \ge s \\ j\not = s}}$}   \overline{ |\varphi_j| Z_{t-j}/x_n}^\epsilon \big)^{l},\\
\overline{I_{[p]+1,t}}^\epsilon
    &:=& \mbox{$\sum_{\substack{|j| \ge s \\ j\not = s}}$} | \overline{ \varphi_j Z_{t-j}/x_n}^\epsilon |^{p-[p]}  
\big( \mbox{$\sum_{\substack{|j| \ge s \\ j\not = s}}$} \overline{ |\varphi_j| Z_{t-j}/x_n}^\epsilon \big)^{[p]},\\
\underline{I_{l,t}}_\epsilon 
    &:=& |\underline{ \varphi_s Z_{t-s}/x_n}_\epsilon |^{p - l} 
\big(   \mbox{$\sum_{\substack{|j| \ge s \\ j\not = s}}$} \underline{ |\varphi_j| Z_{t-j}/x_n}_\epsilon \big)^{l},\\
\underline{I_{[p]+1,t}}_\epsilon 
    &:=& \mbox{$\sum_{ \substack{|j| \ge s \\ j \not = s}}$} |\underline{\varphi_j Z_{t-j}/x_n}_\epsilon |^{p-[p]} 
\big( \mbox{$\sum_{\substack{|j| \ge s \\ j\not = s}}$} \underline{ |\varphi_j| Z_{t-j}/x_n}_\epsilon \big)^{[p]}.
\eeao
To study each term, we write for $q \in \mathbb{N}$, $\mathcal{J} \subseteq \mathbb{N}$, $(\psi_j) \in \mathbb{R}^{|\mathcal{J}|}$, 
\beam\label{eq:power:sum}
(\mbox{$\sum_{j \in \mathcal{J}}$} \psi_j)^q &=& \mbox{$\sum_{\substack{i_1,\dots,i_q \\ i_j \in \mathcal{J}}  }$} \psi_{i_1}\cdots \psi_{i_q}. 
\eeam 
\par 
We start by analyzing the terms corresponding to the truncation from below. An application of  Markov's inequality together with Equation~\eqref{eq:power:sum} yield
\beao 
\lefteqn{ 
\P(\mbox{$\sum_{t=1}^n \underline{I_{[p]+1,t}}_\epsilon $} > \delta)}\\
&\le& \delta^{-1}n \E[\underline{I_{[p]+1,t}}_\epsilon]\\
&\le& \delta^{-1} n \mbox{$\sum_{ \substack{ i_1,\cdots, i_{[p] +1} \\ |i_j| \ge s } }$}  |\varphi_{i_1} \cdots \varphi_{i_{[p]} }| |\varphi_{i_{[p]+1}}|^{p-[p]}\, \\
&& \times \E[ | \underline{Z_{-i_1}/x_n}_{\epsilon \varphi_{i_1}^{-1}}| \cdots |\underline{ Z_{-i_{[p] } }/x_n}_{\epsilon \varphi_{i_{[p]}}^{-1}} | |\underline{ Z_{-i_{[p]+1}}/x_n}_{ {\epsilon \varphi_{i_{[p] +1}}^{-1}}}|^{p-[p]}]. 
\eeao 
\par 
Moreover, recall the noise sequence $(Z_t)$ are iid random variables satisfying {\bf RV}$_\alpha$. Therefore, for the expectation we can factor the independent noise terms as the product of at most $[p]+1$ terms. For each term, the noise random variable $Z_{-i_j}$ will be raised to the power at most $p$. As $p < \alpha$, we can use Karamata's theorem on each of these terms.  
\par 
Finally, an application of Karamata's theorem and Potter's bound yield there exists $\kappa > 0$, such that for all $\epsilon, \delta > 0$
\beao 
\frac{\P( \mbox{$\sum_{t=1}^n \underline{I_{[p]+1,t}}_\epsilon $} > \delta) }{n \P(|Z_1| > x_n)} 
&\le &\frac{\alpha}{\alpha-p} O( \delta^{-1} \,\epsilon^{-(\alpha-\kappa)}\, ( \mbox{$\sum_{|j| \ge s}$} |\varphi_j|^{\alpha-\kappa})^p).
\eeao 
We conclude that this term is negligible by letting first $n \to \infty$, and then $s \to \infty$.
\par 
We can follow similar steps as before to study the truncation from below terms $I_{l,t}$, $l = 0, \dots, [p]$. An application of Markov's inequality entails there exists $\kappa > 0$ such that
\beao 
\frac{\P(\mbox{$\sum_{t=1}^n \underline{ I_{l,t}}_\epsilon$} > \delta)}{n\P(|Z_1| > x_n) } &\le& \delta^{-1} \frac{\E\big[ |\underline{ \varphi_s Z_1/x_n}_\epsilon |^{p - l} \big]}{\P(|Z_1| > x_n)} \E\big[
\big( \mbox{$\sum_{\substack{|j| > s \\ j\not = s}}$} \underline{ |\varphi_j| Z_{t-j}/x_n}_\epsilon \big)^{l}\big]  \\
&=& \frac{\alpha}{\alpha-p+l}  O(\delta^{-1} |\varphi_s|^{\alpha-\kappa} \epsilon^{-(\alpha-\kappa)} ( \mbox{$\sum_{|j| \ge s}$} |\varphi_j|^{\alpha-\kappa})^l ),
\eeao 
as $n \to \infty$, where the last relation holds by Karamata's theorem and an application of Potter's bound. Hence, for $l = 0, \dots, [p],[p]+1$, we conclude letting first $n \to \infty$, and then $s \to \infty$. To sum up we have shown that 
\beao 
\lim_{s \to  \infty}\limsup_{n \to  \infty} \frac{\P( \sum_{t=1}^n \underline{I_{l,t}}_\epsilon > \delta ) }{ n \P(|\bfX_1| > x_n) } &=& 0.
\eeao 
\par 
We now turn to the terms relative to the truncation from above. In this case, the assumption $n/x_n^p \to 0$, entails $n\E[\overline{I_{l,t}}^\epsilon] \to 0$, as $n \to \infty$, for $l = 0, \dots,[p],[p]+1$. Therefore, to establish Equation~\eqref{eq:sum:terms2}, it suffices to check the following relation holds:
\beam\label{eq:bar:epsilon:chebychev} 
\lim_{s \to  \infty}\limsup_{n \to  \infty}\frac{\P(\mbox{$\sum_{t=1}^n \overline{ I_{l,t}}^\epsilon - \E[ \overline{ I_{l,t}}^\epsilon] $} > \delta)}{n\P(|\bfX_1| > x_n)} &=& 0. 
\eeam
We apply Chebychev's inequality, which together with the stationarity of the series $(Z_t)$, yields
\beao 
{ \P(\mbox{$\sum_{t=1}^n \overline{ I_{l,t}}^\epsilon - \E[ \overline{ I_{l,t}}^\epsilon] $} > \delta)}
&\le& 2 \, \delta^{-2} \, n \mbox{$\sum_{t=0}^n$} \Cov( \overline{ I_{l,0}}^\epsilon , \overline{ I_{l,t}}^\epsilon )
\eeao 
As in the arguments for the truncation from above, we start by showing that the term in \eqref{eq:bar:epsilon:chebychev} is negligible for $l =[p] +1$. This reasoning  can again be extended for $l = 0, \dots, [p]$. Computation of the covariances then yields
\beao 
\lefteqn{ \Cov( \overline{ I_{[p]+1,0}}^\epsilon , \overline{ I_{[p] +1 ,t}}^\epsilon ) } \\
&=& \mbox{$\sum_{ \substack{ i_1,\cdots, i_{[p] +1} \\ |i_j| \ge s, i_j \not = s} }$} 
\mbox{$\sum_{ \substack{ 
\ell_1,\cdots, \ell_{[p] +1} \\ |\ell_j| \ge  s, \ell_j \not = s} }$}
|\varphi_{i_1} \cdots \varphi_{i_{[p]} }| |\varphi_{i_{[p]+1}}|^{p-[p]} |\varphi_{\ell_1} \cdots \varphi_{\ell_{[p]} }| |\varphi_{\ell_{[p]+1}}|^{p-[p]}
\\
&& \times \Cov(  | \overline{Z_{-i_1}/x_n}^{\epsilon \varphi_{i_1}^{-1}} \cdots \overline{ Z_{-i_{[p] } }/x_n}^{\epsilon \varphi_{i_{[p]}}^{-1}} | |\overline{ Z_{-i_{[p]+1}}/x_n}^{{\epsilon \varphi_{i_{[p] +1}}^{-1}} }|^{p-[p]}  \\
    && \quad \qquad | \overline{Z_{t-\ell_1}/x_n}^{\epsilon \varphi_{\ell_1}^{-1}} \cdots \overline{ Z_{t-\ell_{[p] } }/x_n}^{\epsilon \varphi_{\ell_{[p]}}^{-1}} | |\overline{ Z_{t-\ell_{[p]+1}}/x_n}^{ _{\epsilon \varphi_{\ell_{[p] +1}}^{-1}} }|^{p-[p]}). 
\eeao 
Actually, all but a finite number of terms vanish in the previous double sum because the noise sequence $(Z_t)$ are independent random variables. More precisely,
\beao 
\lefteqn{ \Cov( \overline{ I_{[p]+1,0}}^\epsilon , \overline{ I_{[p] +1 ,t}}^\epsilon ) } \\
&=& \mbox{$\sum_{ \substack{ i_1,\cdots, i_{[p] +1} \\ |i_j| \ge s, i_j \not = s} }$} 
\mbox{$\sum_{ \substack{ 
\ell_1,\cdots, \ell_{[p] +1} \\ \ell_j \in \{i_1-t,\dots,i_{[p]+1}-t\}} }$}\\
&& \times 
|\varphi_{i_1} \dots \varphi_{i_{[p]} }| |\varphi_{i_{[p]+1}}|^{p-[p]} |\varphi_{\ell_1} \dots \varphi_{\ell_{[p]} }| |\varphi_{\ell_{[p]+1}}|^{p-[p]}
\\
&& \times \Cov(  | \overline{Z_{-i_1}/x_n}^{\epsilon \varphi_{i_1}^{-1}} \cdots \overline{ Z_{-i_{[p] } }/x_n}^{\epsilon \varphi_{i_{[p]}}^{-1}} | |\overline{ Z_{-i_{[p]+1}}/x_n}^{{\epsilon \varphi_{i_{[p] +1}}^{-1}} }|^{p-[p]}  \\
    && \quad \qquad | \overline{Z_{t-\ell_1}/x_n}^{\epsilon \varphi_{\ell_1}^{-1}} \cdots \overline{ Z_{t-\ell_{[p] } }/x_n}^{\epsilon \varphi_{\ell_{[p]}}^{-1}} | |\overline{ Z_{t-\ell_{[p]+1}}/x_n}^{ _{\epsilon \varphi_{\ell_{[p] +1}}^{-1}} }|^{p-[p]}). 
\eeao
Moreover, regarding the last covariance term, we notice that it is sufficient to bound the expectation of the product as
\beao 
\lefteqn{\Cov(  | \overline{Z_{-i_1}/x_n}^{\epsilon \varphi_{i_1}^{-1}} \cdots \overline{ Z_{-i_{[p] } }/x_n}^{\epsilon \varphi_{i_{[p]}}^{-1}} | |\overline{ Z_{-i_{[p]+1}}/x_n}^{{\epsilon \varphi_{i_{[p] +1}}^{-1}} }|^{p-[p]},}  \\
    && \quad \qquad | \overline{Z_{t-\ell_1}/x_n}^{\epsilon \varphi_{\ell_1}^{-1}} \cdots \overline{ Z_{t-\ell_{[p] } }/x_n}^{\epsilon \varphi_{\ell_{[p]}}^{-1}} | |\overline{ Z_{t-\ell_{[p]+1}}/x_n}^{ _{\epsilon \varphi_{\ell_{[p] +1}}^{-1}} }|^{p-[p]})\\
    &\le & 
\E[  | \overline{Z_{-i_1}/x_n}^{\epsilon \varphi_{i_1}^{-1}} \cdots \overline{ Z_{-i_{[p] } }/x_n}^{\epsilon \varphi_{i_{[p]}}^{-1}} | |\overline{ Z_{-i_{[p]+1}}/x_n}^{{\epsilon \varphi_{i_{[p] +1}}^{-1}} }|^{p-[p]}  \\
    && \quad \qquad \times | \overline{Z_{t-\ell_1}/x_n}^{\epsilon \varphi_{\ell_1}^{-1}} \cdots \overline{ Z_{t-\ell_{[p] } }/x_n}^{\epsilon \varphi_{\ell_{[p]}}^{-1}} | |\overline{ Z_{t-\ell_{[p]+1}}/x_n}^{ _{\epsilon \varphi_{\ell_{[p] +1}}^{-1}} }|^{p-[p]}].
\eeao 
Since $(Z_t)$ are iid random variable, the expectation term above can be written as the product of expectations as follows
\beao 
\lefteqn{|\varphi_{i_1} \dots \varphi_{i_{[p]} }| |\varphi_{i_{[p]+1}}|^{p-[p]} |\varphi_{\ell_1} \dots \varphi_{\ell_{[p]} }| |\varphi_{\ell_{[p]+1}}|^{p-[p]}}\\
&&\E(  | \overline{Z_{-i_1}/x_n}^{\epsilon \varphi_{i_1}^{-1}}| \cdots |\overline{ Z_{-i_{[p] } }/x_n}^{\epsilon \varphi_{i_{[p]}}^{-1}} | |\overline{ Z_{-i_{[p]+1}}/x_n}^{{\epsilon \varphi_{i_{[p] +1}}^{-1}} }|^{p-[p]}  \\
    && \quad \times | \overline{Z_{t-\ell_1}/x_n}^{\epsilon \varphi_{\ell_1}^{-1}}| \cdots |\overline{ Z_{t-\ell_{[p] } }/x_n}^{\epsilon \varphi_{\ell_{[p]}}^{-1}} | |\overline{ Z_{t-\ell_{[p]+1}}/x_n}^{ _{\epsilon \varphi_{\ell_{[p] +1}}^{-1}} }|^{p-[p]}) \\
    &=& 
\mbox{$\prod_{ 
\substack{ \gamma_1 + \cdots + \gamma_r  = p,\\  
\gamma_1^\prime + \cdots +\gamma^\prime_{r^\prime} = p }} $ } 
    |\varphi_{i_{\gamma_j} }|^{\gamma_j} |\varphi_{i_{\gamma_j}-t}|^{{\gamma^\prime_j}} \E[|\overline{Z_0/x_n}^{\epsilon \varphi_{i_{\gamma_j}}^{-1}}|^{\gamma_j} |\overline{Z_0/x_n}^{\epsilon \varphi_{i_{\gamma_j}-t}^{-1}}|^{\gamma_j'} ],
\eeao
where $
   \gamma_j,\gamma_j^\prime \in \{0,1,\dots,[p]\}$, $\gamma_r,  \gamma^\prime_{r^\prime} \in \{0,p-[p],p-[p]+1,\dots,p\}$ and $1\le r, r'\le [p]+1$. The product above is a factorization with respect to independent noise terms. We have also used the stationarity of $(Z_t)$. The new indices $\gamma_1,\dots,\gamma_r$ are defined recursively in terms of the sequence $(i_t)$. Similarly, we define $\gamma^\prime_1,\dots \gamma^\prime_{r^\prime}$ from $(\ell_t)$. To define $\gamma_1$, first we count the number of times the index $i_1$ appears in $\mathcal{I} = \{i_1,\dots,i_{[p]}\}$. If $i_1 \not = i_{[p]+1}$, we put $\gamma_1$ equal to this count, otherwise, we set $\gamma_1$ equal to this count plus $p-[p]$. Then, we look for the next index  different than $i_1$, say $i_j$, and set $\gamma_2$ as the number of repetitions of $i_j$ in $\mathcal{I}$ plus $p-[p]$ if $i_j \not = i_{[p]+1}$.  We continue in this way until the indices $i_r$ and $\gamma_r$ are defined as previously. We stop as we recognize that all the indices $i_{r},i_{r+1},\cdots,i_{[p]+1}$ have already been considered. Therefore, $\gamma_1+\cdots+\gamma_r = p$. In an identical fashion, we define $\gamma^\prime_1,\dots \gamma^\prime_{r^\prime}$ from $(\ell_t)$. Moreover, notice that for every  $\gamma\in\{\gamma_1,\ldots, \gamma_r\}$ and $\gamma'\in\{\gamma_1',\ldots, \gamma_r'\}$
\beao 
    \lefteqn{|\varphi_{i \gamma }|^\gamma |\varphi_{i_\gamma-t}|^{\gamma^\prime}\E[|\overline{Z_0/x_n}^{\epsilon \varphi_{i_\gamma}^{-1}}|^{\gamma} |\overline{Z_0/x_n}^{\epsilon \varphi_{i_{\gamma}-t}^{-1}}|^{\gamma'}  ] }\\
    &\le& (|\varphi_{i \gamma }|^{2\gamma} |\varphi_{i_\gamma-t}|^{2\gamma^\prime} \E[|\overline{Z_0/x_n}^{\epsilon \varphi_{i_\gamma}^{-1}}|^{2\gamma}]\E [|\overline{Z_0/x_n}^{\epsilon \varphi_{i_\gamma-t}^{-1}}|^{2{\gamma^\prime} }  ])^{1/2}\\
    &\le&  (|\varphi_{i \gamma }|^{2p} \E[|\overline{Z_0/x_n}^{\epsilon \varphi_{i_\gamma}^{-1}}|^{2p}])^{\gamma/2p} 
    \, (|\varphi_{i_\gamma-t}|^{2p}\E[ |\overline{Z_0/x_n}^{\epsilon \varphi_{i_\gamma-t}^{-1}}|^{2p}  ])^{\gamma^\prime/2p}.
\eeao 
The key property $\gamma_1 + \cdots + \gamma_r + \gamma_1^\prime + \cdots + \gamma_r^\prime = 2p$ yields 
\beao
\P(Z_0 > x_n)=\mbox{$\prod_{ 
\substack{ \gamma_1 + \cdots + \gamma_r  = p\\  
\gamma_1^\prime + \cdots +\gamma^\prime_{r^\prime} = p  }} $ }(\P(Z_0 > x_n))^{(\gamma+\gamma^\prime)/2p}\, .
\eeao 
In this case, we can apply Karamata's Theorem to each one of the expectation terms. 
Readily, 
\beao 
\lefteqn{ \mbox{$\sum_{t=1}^n$}
\Cov( \overline{ I_{[p]+1,0}}^\epsilon , \overline{ I_{[p] +1 ,t}}^\epsilon )/\P(Z_0 > x_n) } \\
&\le& \mbox{ $\sum_{t=1}^n$}\mbox{$\sum_{ \substack{ i_1,\cdots, i_{[p] +1} \\ |i_j| \ge s, i_j \not = s} }$ }
\mbox{$\sum_{ \substack{ 
\ell_1,\cdots, \ell_{[p] +1} \\ \ell_j \in \{i_1-t,\dots,i_{[p]+1}-t\}} }$} 
\mbox{$\prod_{ 
\substack{ \gamma_1 + \cdots + \gamma_r  = p\\  
\gamma_1^\prime + \cdots +\gamma^\prime_{r^\prime} = p  }} $ } 
        \\
&&   
    \frac{ (|\varphi_{i \gamma }|^{2p} \E[|\overline{Z_0/x_n}^{\epsilon \varphi_{i_\gamma}^{-1}}|^{2p}])^{\gamma/2p} 
    \, (|\varphi_{i_\gamma-t}|^{2p}\E[ |\overline{Z_0/x_n}^{\epsilon \varphi_{i_\gamma-t}^{-1}}|^{2p}  ])^{\gamma^\prime/2p} }{(\P(Z_0 > x_n))^{(\gamma+\gamma^\prime)/2p}} \\
 &\le& c\, \mbox{$\sum_{ \substack{ i_1,\cdots, i_{[p] +1} \\ |i_j| \ge s, i_j \not = s} }$ } 
\mbox{$\sum_{ \substack{ 
\ell_1,\cdots, \ell_{[p] +1} \\ \ell_j \in \{i_1-t,\dots,i_{[p]+1}-t\}} }$}\mbox{ $\sum_{t=1}^n$} \\
&&\mbox{$\prod_{ 
\substack{ \gamma_1 + \cdots + \gamma_r  = p\\  
\gamma_1^\prime + \cdots +\gamma^\prime_{r^\prime} = p   }} $ }  |\varphi_{i_\gamma}|^{\gamma(\alpha-\kappa)}|\varphi_{i_\gamma - t}|^{\gamma^\prime(\alpha-\kappa)}\\
&=& c\,\mbox{$\sum_{ \substack{ i_1,\cdots, i_{[p] +1} \\ |i_j| \ge s, i_j \not = s} }$ } 
\mbox{$\sum_{ \substack{ 
\ell_1,\cdots, \ell_{[p] +1} \\ \ell_j \in \{i_1-t,\dots,i_{[p]+1}-t\}} }$}\mbox{ $\sum_{t=1}^n$}  |\varphi_{i_j}|^{ (\alpha-\kappa)}|\varphi_{\ell_j - t}|^{ (\alpha-\kappa)} \\ 
&\le& c\, (p+1)(\mbox{$\sum_{|i| \ge s}$ } 
 |\varphi_{i}|^{ (\alpha-\kappa)} )^p ( \mbox{$\sum_{j \in \mathbb{Z}}$} |\varphi_{j}|^{ (\alpha-\kappa)})^p.
\eeao
where $c > 0$, is a constant of no interest. We conclude by letting $s \to \infty$ that \eqref{eq:bar:epsilon:chebychev} holds for $l = [p]+1$. Similarly, this arguments can be extended for $l =1,\dots, [p]$. Overall, this shows \eqref{eq:sum:terms2} holds, and this concludes the proof. \qed

\subsection{Proof of Theorem~\ref{eq:thm:linear:tcl}}\label{sec:proof:clt:linear}

We aim to apply Theorem~\ref{thm:main} for $p > \alpha/2$. First notice condition \eqref{cond:i)} yields $\|\varphi_t\|_p < \infty$.  Let $\kappa' > 0$, and define a sequence $(x_n)$ such that $x_n \in (n^{\kappa' + {1}/{(p \land \alpha)  }}, + \infty)$.  Proposition~\ref{lem:linear:process} implies conditions {\bf AC}, {\bf CS}$_p$ hold, and $n\P(|\bfX_0| > x_n) \to 0$, as $n \to \infty$. Fix $\kappa' > 0$, and $x_n = O( n^{\kappa' + {1}/{(p \land \alpha)  }})$, as $n \to \infty$.
\par 
{
Furthermore, since there exists $\epsilon > 0$ such that $\|(\varphi_j)\|_{\alpha - \epsilon} < \infty$, then condition {\bf M} is automatically satisfied by definition of $\bfQ$.
In addition, 
for the $\widehat \a-$cluster based estimators from Section~\ref{sec:examples:cluster} in \eqref{eq:estimator:Ei} \eqref{eq:estimator:c1}, and \eqref{eq:cluster:sizes} we need to check {\bf S} holds. 
For this we verify the conditions of Lemma~\ref{lem:verif:S}.
Actually, Equation \eqref{eq:moment:cond} has already been demonstrated in Proposition~\ref{lem:linear:process}; and it suffices to follow the lines of the proof of Equation~\eqref{eq:bar:epsilon:chebychev}.
We can therefore conclude that ${\bf S}$ holds
for the $\widehat \a-$cluster based estimators from Section~\ref{sec:examples:cluster} in \eqref{eq:estimator:Ei} \eqref{eq:estimator:c1}, and \eqref{eq:cluster:sizes}.
\par


}
\par 
Finally, to apply Theorem~\ref{thm:main} it suffices to verify the $\beta-$mixing conditions 
{\bf MX}$_\beta$ holds. 
Choose $(k_n)$ as in \eqref{eq:k:3}. Then, there exists $\epsilon,\epsilon^\prime > 0$, and a constant $c > 0$, such that
\beam\label{eq:k:lin}
m_n/k_n &=&  1/(c(p)b_n\P(|\bfX_0| > x_{b_n})) \\
& \le & c\,  x_{b_n}^{(\alpha + \epsilon)}/b_n  \;=\; c\, b_n^{-1+\tfrac{\alpha}{\alpha \land p } +\tfrac{\epsilon}{\alpha \land p } + \kappa' (\alpha + \epsilon) }\nonumber \\
&\le&  c\,  b_n^{ -1 + \tfrac{\alpha}{p\land \alpha} + \epsilon^\prime}. \nonumber
\eeam
This follows using Potter's bounds. Let $\ell_n = b_n^{(1-\epsilon)}$ such that $\ell_n/b_n \to 0$.  Finally, applying Proposition~\ref{prop:beta:mixing:linear}, we can find $\epsilon^\prime > 0$ such that the relation below holds
\beao 
m_n\beta_{\ell_n}/k_n &=& O( b_n^{ -\tfrac{(\rho-1)\alpha}{1+\alpha} + \tfrac{\alpha}{\alpha \land p} + \epsilon^\prime }  ), \quad n \to \infty.
\eeao 
Then, taking $\rho > 1 + \tfrac{1+\alpha}{\alpha \land p} + \epsilon$ yields  $m_n\beta_{\ell_n}/k_n \to 0$. In this argument, we can choose $\epsilon^\prime > 0$ to be arbitrarily small. Then, assuming \eqref{cond:i)} entails $m_n\beta_{\ell_n}/k_n \to 0$.
\par 
Moreover, let $\delta >0$ be as in \eqref{eq:cond:lind:1}. Since $\rho > \tfrac{2}{\delta}(1+\tfrac{1}{\alpha}) + 3 + \tfrac{2}{\alpha}$, equation~\eqref{eq:beta:linear} yields $\sum_{t=1}^{\infty}\beta_{t}^{\delta/(2+\delta)} < \infty$. In this case, there exist $\epsilon > 0$ such that
\beao \sum_{t=1}^{m_n} (m_n \beta_{tb_n}/k_n)^{\tfrac{\delta}{2+\delta}} &=& O\Big( b_n^{\big( -\tfrac{(\rho-1)\alpha}{1+\alpha} + \tfrac{\alpha}{\alpha \land p} + \epsilon \big)\tfrac{\delta}{(2+\delta)}}\Big).
\eeao 
Furthermore, for $p > \alpha/2$, notice $\rho > 3 + 2/\alpha > 1 + ({1+\alpha})/({\alpha \land p})$. Similarly as before, notice $\epsilon > 0$ can be made arbitrarily small. Putting everything together, we conclude that \eqref{eq:covariance:beta:mixing} holds.
This completes the proof that {\bf MX}$_\beta$ holds. 
\par 
Moreover, notice that, for all $\epsilon > 0$, 
\beao 
m_n \beta_{b_n} = O\left( n b_n^{-(\rho-1)\frac{\alpha-\epsilon}{1+\alpha-\epsilon}}\right), 
\eeao 
as $n \to \infty$. 
Furthermore, by our assumptions we have $\frac{1}{(\rho-1)}\big(1 + \frac{1}{\alpha - \epsilon} \big) <1,$ 
hence
 taking 
 \beao 
 b_n &\sim& n^{\frac{1}{(\rho-1)}\big(1 + \frac{1}{\alpha - \epsilon} \big)}, 
 \eeao 
 as $n \to \infty$,
 we obtain $m_n \beta_{b_n}\to 0$, as $n \to \infty$.
On the other hand, the sequence $(k_n)$ satisfies 
\beao
k_n \sim  c(p) n \P(|\bfX_0| > x_{b_n}) 
= o( n b_n^{-\frac{\alpha}{p\land \alpha} }), \quad n \to \infty,
\eeao
where in the last line we have used Potter's bounds. 
This implies that for all $\eta > 0$, we can take 
\[
k_n \sim n^{1-\frac{\alpha}{(p\land \alpha)(\rho-1)}\big(1 + \frac{1}{\alpha - \epsilon} \big)}, 
\]
where for $\epsilon$ sufficiently small, we can verify $\frac{\alpha}{(p\land \alpha)(\rho-1)}\big(1 + \frac{1}{\alpha-\epsilon} \big) < 1$. 
This concludes the proof of Theorem~\ref{eq:thm:linear:tcl}.
\par 
\qed

\subsection{Proof of Proposition~\ref{prop:sre:cs}}\label{sec:proof:sre:cs}

Let $(\bfX_t)$ be the stationary solution to the SRE \eqref{eq:sre:def} as in Example~\ref{example:sre}, satisfying ${\bf RV}_\alpha$, for $\alpha >0$. Then, $(\bfX_t)$ admits the causal representation in \eqref{eq:sre:representation}, where $((\bfA_t,\bfB_t))$ is a sequence of iid innovations. Then, backward computations yield
\beam\label{sre:representation}
\bfX_t &=& \Pi_{t} \bfX_0 + R_t, \qquad t \ge 1,
\eeam
where for $1 \le i \le t$
\beam\label{eq:Pi} \Pi_{i,t} := \bfA_{i} \cdots \bfA_t, \qquad R_t := \sum_{j=1}^{t} \Pi_{j+1,t} {\bfB}_j,
\eeam 
with the conventions: $\Pi_{1,t} = \Pi_t$, and $\Pi_{t+1,t} = Id$. Notice that the remaining term $R_t$ is measurable with respect to $\sigma\big( (\bfA_i,{\bfB}_i)_{1\le i \le t} \big)$, and is independent of the sigma-field $\sigma\big( (\bfX_t)_{t\le 0 } \big)$. 
\par 
Condition ${\bf AC}$ has been shown for Theorem 4.17 in \cite{mikosch:wintenberger:2013}. We focus on showing ${\bf CS}_p$ holds for $p \in( \alpha/2,\alpha)$.
\par 
To begin, note condition ${\bf CS}_p$ was borrowed from Equation (5.2) in \cite{buritica:mikosch:wintenberger:2021}. For $p \in (0,\alpha)$, and sequences $(x_n)$ such that $n/x_n^p \to 0$, as $n \to \infty$, we have $n\E[|\bfX_0/x_n|^p] \to 0$, thus our condition ${\bf CS}_p$ and  Equation (5.2) in \cite{buritica:mikosch:wintenberger:2021} coincide. More precisely, we show
\beam\label{eq:sre:11}
 \lim_{\epsilon \downarrow 0} \lim_{n \to \infty} \, \frac{\P\big( \big| \| \overline{\bfX_{[1,n]}/x_n}^\epsilon \|^p_p  - \E[\| \overline{\bfX_{[1,n]}/x_n}^\epsilon \|^p_p\big|  > \delta   \big)}{n \P(|\bfX_1| > x_n)}  &=& 0.
\eeam
For this reason, we focus on showing \eqref{eq:sre:11} holds.
Actually, we show below that, for $(x_n)$ as in Proposition~\ref{prop:sre:cs},  condition ${\bf CS}_p$ holds over uniform regions $\Lambda_n = (x_n,\infty)$ in the sense of \eqref{cs:sre:bound1}. Further, for the purposes of completeness, we show \eqref{cs:sre:bound1} holds generally for sequences $(x_n)$ such that $n\P(|\bfX_0| > x_n) \to 0$, as $n \to \infty$ in the setting of Example~\ref{example:sre}.
\par 
Let $p \in ( \alpha/2,\alpha)$, and consider a sequence $(x_n)$ satisfying the assumptions of Proposition~\ref{prop:sre:cs}. Consider the region $\Lambda_n = (x_n,\infty)$, and consider $x \in \Lambda_n$.  An application of Chebychev's inequality yields
\beam  \label{eq:cheby:ineq}
\P\big( \big| \| \overline{\bfX_{[1,n]}/x}^\epsilon \|^p_p  - \E[\| \overline{\bfX_{[1,n]}/x}^\epsilon \|^p_p\big|  > \delta   \big) &\le& \,2\,n\, \delta^{-2} \sum_{t=0}^n I_t\, ,
\eeam
such that we denote $I_t = \Cov( | \overline{\bfX_0/x}^\epsilon|^p,| \overline{\bfX_t/x}^\epsilon|^p)$. 
\par 
Let $(\Pi_t)$ and $(R_t)$ be as in \eqref{eq:Pi} such that $(\bfX_t)$ satisfies Equation~\eqref{eq:sre:representation}. We define a new Markov chain $(\bfX_t^\prime)_{t \ge 0}$ satisfying 
\beam \label{eq:sre:prime}
\bfX^\prime_t &:=& \Pi_t \bfX^\prime_0 + R_t\,,
\eeam 
with $\bfX^\prime_0$ independent of $(\bfX_t)$ and identically distributed as $\bfX_0$. We can see $(\bfX^\prime_t)$ as the solution of the SRE \eqref{eq:sre:def} for the sequence of innovations $\big((\bfA_t^{\prime}, {\bfB}_t^\prime ) \big)$ where $(\bfA_t^\prime,{\bfB}_t^\prime) = (\bfA_t, {\bfB}_t)$ for $t \le 0$ and $(\bfA_t^\prime,{\bfB}_t^\prime)_{t \ge  1}$ is an iid sequence independent of $(\bfA_t,{\bfB}_t)$, distributed as the generic element $(\bfA,\bfB)$. Then, following the notation in \eqref{eq:cheby:ineq}, we can rewrite $I_t$ as
\beao 
I_t &=& \E\big[|\overline{\bfX_0/x}^\epsilon|^p|\overline{\bfX_t/x}^\epsilon|^p\big] - \E\big[|\overline{\bfX_0/x}^\epsilon|^p |\overline{\bfX_t^\prime/x}^\epsilon|^p\big]  \\
&\le& \E\big[|\overline{\bfX_0/x}^\epsilon|^p\,\big( |\overline{\bfX_t/x}^\epsilon|^p -  |\overline{\bfX_t^\prime/x}^\epsilon|^p \big)_{+}\, \big] \\
&\le& \E\big[|\overline{\bfX_0/x}^\epsilon|^p\,\big( |{\bfX_t/x}|^p -  |{\bfX_t^\prime/x}|^p \big)_{+}\,\1(|\bfX_t/x| \le \epsilon) \1( |\bfX_t^\prime/x| \le \epsilon) \big] \\
&&+ \E\big[|\overline{\bfX_0/x_n}^\epsilon|^p\,|\overline{\bfX_t/x}^\epsilon|^p \1(| \bfX_t^\prime/x| > \epsilon) \big] \\
&=& I_{t,1} + I_{t,2}.
\eeao 
We show that $I_{t,1}$ is negligible letting first $n \to \infty$, and then $\epsilon \downarrow 0$. For this, we consider two cases. First, assume $p > 1$. Then, for the first term $I_{t,1}$, a Taylor decomposition yields 
\beao 
\lefteqn{I_{t,1}}\\ 
&\le& p\, \E\big[\, |\overline{\bfX_0/x}^\epsilon|^p |\overline{\bfX_t^\prime/x -\bfX_t/x}^{2\,\epsilon}| | \overline{\bfX_t^\prime/x}^\epsilon + \xi\big(\overline{\bfX_t^\prime/x -\bfX_t/x}^{2\epsilon} \big)|^{p-1} \, \big] \\
&=& p\, \E\big[\, |\overline{\bfX_0/x}^\epsilon|^p |\overline{\Pi_t \bfX_0^\prime/x -\Pi_t \bfX_0/x}^{2\,\epsilon}| \\
&& \qquad \times | \overline{\bfX_t^\prime/x}^\epsilon + \xi\big(\overline{\Pi_t \bfX_0^\prime/x -\Pi_t \bfX_0/x}^{2\epsilon} \big)|^{p-1} \, \big] ,
\eeao 
for some random variable $\xi \in (0,1)$ a.s. In the last equality, we have used the definition of $(\bfX^\prime)$ in \eqref{eq:sre:prime}. Moreover, we can bound $I_{t,1}$ by
\beao 
I_{t,1}&\le& p 2^{0\lor (p-1)} \, \E\big[\, |\overline{\bfX_0/x}^\epsilon|^p |\overline{\Pi_t \bfX_0^\prime/x -\Pi_t \bfX_0/x}^{2\,\epsilon}| | \overline{\bfX_t^\prime/x}^\epsilon|^{p-1}\big] \\
&& + \, p 2^{0\lor (p-1)} \, \E\big[\, |\overline{\bfX_0/x}^\epsilon|^p |\overline{\Pi_t \bfX_0^\prime/x -\Pi_t \bfX_0/x}^{2\,\epsilon}|^p\big] 
\eeao 
Now, an application of Jenssen's inequality, Potter's bounds, and Karamata's theorem, yield
\beao 
I_{t,1} &\le& c \, \E\big[\, |\overline{\bfX_0/x}^\epsilon|^{2p}]^{1/2} \E\big[ |\overline{\Pi_t \bfX_0^\prime/x}^{4\,\epsilon}|^{2p}\big]^{1/2}\\
&\le& c \, \big(\E\big[\, |\overline{\bfX_0/x}^\epsilon|^{2p}] \E\big [|\Pi_t|_{op}^{\alpha-\delta}\big] \P(|\bfX_0| > x) \big)^{1/2}. \\
&=& O\big( \big(\epsilon^{2p - \alpha} \E[|\Pi_t|_{op}^{\alpha-\delta}\big]\big)^{1/2} \P(|\bfX_0| > x) \, \big), \quad  x_n > x_0,
\eeao 
for constants $c>0,$ $x_0 >0$. Moreover, under the assumptions of Example~\ref{example:sre} we have $\E[|\Pi_t|_{op}^{\a - \delta}] < 1$ for $t$ sufficiently large. Thereby, we conclude \beam\label{eq:sre:I}
\lim_{\epsilon \downarrow 0}\limsup_{n \to  \infty} \sum_{t=1}^n I_{t,1}/\P(|\bfX_0| > x) &=& 0.
\eeam
We now come back to the case where $p < 1$. In this case we can use a subadditivity argument and we conclude by similar steps that relation \eqref{eq:sre:I} holds for all $p \in (\alpha/2,\alpha).$
\par 
Now, concerning the second term $I_{t,2}$ we have
\beao  
I_{t,2} &:=& \E\big[|\overline{\bfX_0/x}^\epsilon|^p\,|\overline{\bfX_t/x}^\epsilon|^p \1(| \bfX_t^\prime/x| > \epsilon) \big] \\
&\le& \E\big[|\overline{\bfX_0/x}^\epsilon|^p\,|\overline{\Pi_t \bfX_0/x}^\epsilon|^p \1(| \bfX_t^\prime/x| > \epsilon) \big] \\
&&+ \E\big[|\overline{\bfX_0/x}^\epsilon|^p\,|\overline{R_t/x}^{\epsilon}|^p \1(| \bfX_t^\prime/x| > \epsilon) \big]  \\
&&+ \E\big[|\overline{\bfX_0/x}^\epsilon|^p\,\1(|\Pi_t \bfX_0/x|  > \epsilon) \1(| \bfX_t^\prime/x| > \epsilon) \big] \\
&=& O( \E \big[|\overline{\bfX_0/x}^\epsilon|^p\big] \E\big[|\overline{R_t/x}^{
\epsilon}|^p \1(| \bfX_t^\prime/x| > \epsilon) \big])\,. 
\eeao 
Therefore we have,
\beao 
 \sum_{t=1}^{n} I_{t,2}/\P(|\bfX_0| > x) 
&=& O\big(n\E \big[|\overline{\bfX_0/x}^\epsilon|^p\big]\big).
\eeao 
Assume $(x_n)$ is such that there exists $\kappa' > 0$ satisfying $n/x_n^{p \land (\alpha-\kappa')}\to 0$, as $n \to \infty$. Then,
\beam\label{eq:sre:unif}
 \lim_{\epsilon \downarrow 0} \lim_{n \to \infty}\sup_{x \in \Lambda_n} \, \frac{\P\big( \big| \| \overline{\bfX_{[1,n]}/x}^\epsilon \|^p_p  - \E[\| \overline{\bfX_{[1,n]}/x}^\epsilon \|^p_p\big|  > \delta   \big)}{n \P(|\bfX_1| > x)}  &=& 0.
\eeam
Moreover, if $n/x_n^{p}\to 0$ then $\E[\| \overline{\bfX_{[1,n]}/x}^\epsilon \|^p_p] \to 0$ for $p \in (\alpha/2,\alpha)$. In this case,  ${\bf CS}_p$ holds uniformly over the region $\Lambda_n$.
\par 
On the other hand, note that we also have $I_t = O(\beta_t)$. Therefore, if we consider a sequence $(\ell_n)$ such that $\ell_n \to \infty$, $n\to \infty$, then we can have
\beam\label{cs:sre:bound1}
\lefteqn{\P\big( \big| \| \overline{\bfX_{[1,n]}/x_n}^\epsilon \|^p_p  - \E[\| \overline{\bfX_{[1,n]}/x_n}^\epsilon \|^p_p\big|  > \delta   \big)/n\P(|\bfX_0| > x_n)}\\
&\le& 2\,n\, \delta^{-2} (\sum_{t=0}^{\ell_n} I_t + \sum_{t= \ell_n+1}^n I_t)/n\P(|\bfX_0| > x_n) \nonumber  \\
&\le& O( \ell_n \E[ |\overline{\bfX_0/x_n}^\epsilon|^p] + \sum_{t= \ell_n+1}^n I_t/\P(|\bfX_0| > x_n)) \nonumber\\
&\le& O( \ell_n \E[ |\overline{\bfX_0/x_n}^\epsilon|^p] +  \sum_{t= \ell_n+1}^n \beta_t/\P(|\bfX_0| > x_n)\, ) \; = \; J_1 + J_2.\label{cs:sre:bound2}
\eeam
where in the last bound we use the covariance inequality for the $(\beta_t)$ mixing coefficients. 

Furthermore, the bound in \eqref{cs:sre:bound1} consists of two terms as $\eqref{cs:sre:bound1} \le J_1 + J_2$. If we want $J_1$ to go to zero as $n\to\infty$ we can choose $\ell_n := x_n^{p-\delta}$, for some $\delta > 0$. Now, for the second term $J_2$, we first note that it is null if $\ell_n>n$ by convention. Otherwise we recall that the mixing-coefficients $(\beta_t)$ have a geometric decaying rate. Thereby, there exists $\rho \in (0,1)$ such that we can bound the second term $J_2$ by 
\beao 
J_2 &=&  O\big(\mbox{$\sum_{t=\ell_n+1}^n$} \rho^t/\P(|\bfX_0| > x_n) \big) \\
 &=& O(\rho^{\ell_n}/\P(|\bfX_0| > x_n)) \\
 &\le& O(\rho^{\ell_n} x_n^{(\alpha + \delta) } ).
\eeao 
Therefore, $J_2 \to 0$ as $n \to \infty$ by plugging in the value we set for $\ell_n$. Overall, we conclude that for all sequences $(x_n)$ such that $n\P(|\bfX_n| > x_n) \to 0 $, then $\lim_{n \to \infty}\eqref{cs:sre:bound1}= 0$ and this shows \eqref{eq:sre:11}. Moreover, we also saw this convergence holds over uniform regions $\Lambda_n = (x_n,\infty)$, in the sense \eqref{eq:sre:unif}, if we assume in addition $n/x_n^{p\land(\alpha-\kappa')} \to 0$, as $n \to \infty$. Finally, this shows that ${\bf CS}_p$ holds and this concludes the proof of Proposition~\ref{prop:sre:cs}.

\subsection{Proof of Theorem~\ref{thm:sre:clt}}\label{proof:sre:clt}
Our goal it to verify that we can apply Theorem~\ref{thm:main} as we combine Proposition~\ref{prop:sre:cs} and \ref{prop:mixing}. First, notice for any $\kappa' >0$, if we consider a sequence $(x_n)$ such that $x_n \in ( n^{\kappa' + 1/(p\land \alpha)},+\infty)$, then conditions {\bf AC}, ${\bf CS}_p$ hold thanks to Proposition~\ref{prop:sre:cs}. Since $c(p) <\infty$ in \eqref{eq:constant:cp:1:tcl}, Proposition \ref{prop:existence:cluster:process} holds and the time series admits a $p-$cluster process $\bfQ^{(p)}$.  Fix $\kappa' >0$, and $x_n = O(n^{\kappa' + 1/(p\land\alpha)})$.
\par 
We focus now on the verification of the mixing condition 
${\bf MX}_\beta$ in Theorem~\ref{thm:main}. Applying Proposition~\ref{prop:mixing}, there exists $\rho \in (0,1)$ such that the coefficients $(\beta_t)$ satisfy $\beta_t = O(\rho^t)$. 
Moreover, if we choose $(k_n)$ according to \eqref{eq:k:3} as in the linear model case then there exists $\epsilon,\epsilon^\prime > 0$, and a constant $c > 0$, such that
\beao 
m_n/k_n  &\le&  c\,  b_n^{ -1 + \tfrac{\alpha}{p\land \alpha} + \epsilon^\prime},
\eeao 
and thus
$\sum_{t=1}^{m_n}(m_n\beta_{b_n}/k_n )^{\delta/(2+\delta)} = O( \rho^{b_n\delta/(2+\delta)}(b_n)^\epsilon),$ which goes to zero as $n \to \infty$. Moreover, for all $\eta \in (0,1)$, choosing $\ell_n = b_n^\eta $, we have $m_n\beta_{\ell_n}/k_n \to 0$, $b_n/\ell_n \to 0$, as $n \to 0$. Therefore, we have verified ${\bf MX}_\beta$ holds. 
Moreover, taking $(b_n)$ as 
$b_n \sim n^{1-\eta}$, for any, $\eta \in (0,1)$, implies $m_n \beta_{b_n} \to 0,$ as $n \to \infty$.
This verifies all assumptions of Theorem~\ref{thm:main}.

It remains to verify that {\bf M} and {\bf S} hold. 
To verify { \bf M} we check $\E[\|\bfQ\|_{\a-\vep}^{2\a }]<\infty$. We use the definition of $\bfQ$, a telescoping sum argument and the time-change formula to obtain
\beao
\lefteqn{\E\big[\|\bfQ\|_{\a-\vep}^{2\a }\big]}\\
&=& \E\Big[\dfrac{(\sum_{t\in\Z}|\bfTh_t|^{\a-\vep})^{{2\a }/(\a-\vep)}}{\|\bfTh\|_\a^{{2\a }}}\Big]\\
&=& \sum_{j\in\Z}\E\Big[\dfrac{(\sum_{t\ge j}|\bfTh_t|^{\a-\vep})^{{2\a }/(\a-\vep)}-(\sum_{t\ge j+1}|\bfTh_t|^{\a-\vep})^{{2\a }/(\a-\vep)}}{\|\bfTh\|_\a^{{2\a }}}\Big]\\
&=& \sum_{j\in\Z}\E\Big[|\bfTh_j|^\a\dfrac{(\sum_{t\ge 0}|\bfTh_t|^{\a-\vep})^{{2\a }/(\a-\vep)}-(\sum_{t\ge 1}|\bfTh_t|^{\a-\vep})^{{2\a }/(\a-\vep)}}{\|\bfTh\|_\a^{{2\a }}}\Big]\\
&=&\E\Big[\dfrac{(\sum_{t\ge 0}|\bfTh_t|^{\a-\vep})^{{2\a }/(\a-\vep)}-(\sum_{t\ge 1}|\bfTh_t|^{\a-\vep})^{{2\a }/(\a-\vep)}}{\|\bfTh\|_\a^{\alpha}}\Big]
\eeao
By convexity of the function $g(x)=x^{{2\a }/(\a-\vep)}$ we obtain
\beao 
\Big(1+\sum_{t\ge 1}|\bfTh_t|^{\a-\vep}\Big)^{{2\a }/(\a-\vep)}&-&\Big(\sum_{t\ge 1}|\bfTh_t|^{\a-\vep}\Big)^{{2\a }/(\a-\vep)}\\
&\le& \dfrac{2\a}{\a-\vep}\Big(\sum_{t\ge 1}|\bfTh_t|^{\a-\vep}\Big)^{{\a+\epsilon }/(\a-\vep)}\\
&\le& \dfrac{2\a}{\a-\vep} \sum_{t\ge 1}|\bfTh_t|^{{\a+\epsilon }}\,.
\eeao
Then, 
\beao
\E[\|\bfQ\|_{\a - \epsilon}^{2\a} ] &\le& \dfrac{2\a}{\a-\vep} \E\Big[\sum_{t\ge 1}\frac{|\bfTh_t|^{{\a+\epsilon}} }{\|\bfTh\|_\a^{\alpha}}\Big] \,\le \, \dfrac{2\a}{\a-\vep} \E\big[\sum_{t\ge 1}|\bfTh_t|^{\epsilon } \big].
\eeao
The latter expression is integrable because $\E[|\bfA|^{\kappa'}_{op}]<1$ and $\E[|\bfTh_t|^{\vep}]=\E[|\bfA_1\cdots\bfA_t|^{\vep}_{op}]\le\E[|\bfA|^{\kappa'}_{op}]^{t/\kappa'}$ for every $\vep<\kappa'$. Thus {\bf M} holds in Example~\ref{example:sre}.

\par 
Finally,
to verify {\bf S}
we rely on Lemma~\ref{lem:verif:S}.
Actually, Equation \eqref{eq:moment:cond} has already been demonstrated in Proposition~\ref{prop:sre:cs}; and it suffices to follow the lines of the proof of Equation~\eqref{eq:cheby:ineq}.
We can therefore conclude that ${\bf S}$ holds
for the $\widehat \a-$cluster based estimators from Section~\ref{sec:examples:cluster} in \eqref{eq:estimator:Ei} \eqref{eq:estimator:c1}, and \eqref{eq:cluster:sizes}.

\qed

\begin{lemma}\label{lem:verif:S}
Consider $(\bfX_t)$ to be a time series with values in $(\mathbb{R}^d, |\cdot|)$, and consider the $\widehat \a-$cluster based estimators from Section~\ref{sec:examples:cluster} in \eqref{eq:estimator:Ei} \eqref{eq:estimator:c1}, and \eqref{eq:cluster:sizes}. 
Assume there exists $\eta > 0$ such that $\alpha - \epsilon - \eta > \epsilon/2$ and 
\beam\label{eq:moment:cond}
\lim_{n \to \infty}\frac{\E[(\|\overline{ \bfX_{[1,n]}/x_n}^\delta\|^{\alpha-\epsilon}_{\alpha-\epsilon} 
-\E[\|\overline{ \bfX_{[1,n]}/x_n}^\delta\|^{\alpha-\epsilon}_{\alpha-\epsilon} ])^2]}{n\P(|\bfX_1| > x_n)} < \infty.
\eeam 
Then, condition ${\bf S}$ holds.
\end{lemma}
\begin{proof}{Proof of Lemma~\ref{lem:verif:S}}
For this we appeal to Remark~\ref{remark:check:S:ei}, which shows the functionals $\partial f_q/\partial q |_{q=\a}$ and $ \sup_{q^\prime \in [\alpha - \epsilon, \alpha + \epsilon]}\partial^2f_q/\partial q^2|_{q=q^\prime}$ are continuous functions in $\mathcal{G}_+(\ell^\alpha)$.
Moreover, we take the example  of the extremal index in \eqref{eq:estimator:Ei} with $f(\bfx) = \|\bfx\|^\a_{\infty}/\|\bfx\|^\a_\a$, 
as for \eqref{eq:estimator:c1}, and \eqref{eq:cluster:sizes} similar calculations yield the desired result. 
Then, to verify {\bf S} note that on the event $\{ \|\bfX_{[1,n]}\|_{q^\prime} > x_{b_n}\} $ we have, for $q^\prime > \alpha - \epsilon$,
\beao 
0 \quad < \quad  \lefteqn{ 
 \frac{\partial f_q(\bfX_{[1,n]} ) }{\partial q}\big|_{q=q^\prime} } \\
&=& 
\frac{\| \bfX_{[1,n]} \|_\infty^{q}}{\| \bfX_{[1,n]} \|_{q^\prime}^{q^\prime}} 
\sum_{t=1}^{n}
\frac{|\bfX_t|^{q^\prime}}{\| \bfX_{[1,n]} \|_{q^\prime}^{q^\prime}}
\log(\|\bfX_{[1,n]}\|_\infty/|\bfX_t |)\\
 && \qquad\le  \sum_{t=1}^{n}
\frac{|\bfX_t|^{\a-\epsilon}}{\| \bfX_{[1,n]} \|_{q^\prime}^{\a-\epsilon}} \\
&& \qquad\qquad  =
\frac{\|\bfX_{[1,n]}/x_n\|_{\a-\epsilon}^{\a-\epsilon}}{\| \bfX_{[1,n]}/x_n \|_{q^\prime}^{\a-\epsilon}} 
\quad \le \quad \|\bfX_{[1,n]}/x_n\|_{\a-\epsilon}^{\a-\epsilon}.
\eeao 
Similar calculations yield
\beam \label{eq:remain:log}
\lefteqn{ 
\Big| 
\frac{\partial^2 f_q(\bfX_{[1,n]} ) }{\partial q^2}\big|_{q=q^\prime} \big| } \nonumber\\
& = & 
\frac{\| \bfX_{[1,n]} \|_\infty^{q}}{\| \bfX_{[1,n]} \|_{q^\prime}^{q^\prime}} 
\sum_{t=1}^{n}\sum_{j=1}^{n}
\frac{|\bfX_t|^{q^\prime}}{\| \bfX_{[1,n]} \|_{q^\prime}^{q^\prime}}
\frac{|\bfX_j|^{q^\prime}}{\| \bfX_{[1,n]} \|_{q^\prime}^{q^\prime}}\log(\|\bfX_{[1,n]}\|_\infty/|\bfX_j |)|\log(| \bfX_j |/| \bfX_t |)| \nonumber\\
&\le& 
\sum_{t=1}^{n}\sum_{j=1}^{n}
\frac{|\bfX_t|^{q^\prime}}{\| \bfX_{[1,n]} \|_{q^\prime}^{q^\prime}}
\frac{|\bfX_j|^{\alpha-\epsilon}}{\| \bfX_{[1,n]} \|_{q^\prime}^{\alpha-\epsilon}}|\log(| \bfX_j |/| \bfX_t |)| \nonumber \\
&\le& \|\bfX_{[1,n]}/x_n\|_{\a-\epsilon}^{\a-\epsilon} \nonumber  \\
&&- \sum_{t=1}^{n}\sum_{j=1}^{n}
\frac{|\bfX_t|^{q^\prime}}{\| \bfX_{[1,n]} \|_{q^\prime}^{q^\prime}}
\frac{|\bfX_j|^{\alpha-\epsilon}}{\| \bfX_{[1,n]} \|_{q^\prime}^{\alpha-\epsilon}}
\log(| \bfX_t |/| \bfX_j |) \1(|\bfX_j|/|\bfX_t| < 1). \nonumber\\
&=& \|\bfX_{[1,n]}/x_n\|_{\a-\epsilon}^{\a-\epsilon}  + R\nonumber\\
&&
\eeam 
To study \eqref{eq:remain:log} , we handle separately the two remaining terms. 
Regarding the remaining term $R$ in \eqref{eq:remain:log}, note,
\beao 
\lefteqn{
0 \quad < \quad R}\\
&=& \sum_{t=1}^{n}\sum_{j=1}^{n}
\frac{|\bfX_t|^{q^\prime}}{\| \bfX_{[1,n]} \|_{q^\prime}^{q^\prime}}
\frac{|\bfX_j|^{\alpha-\epsilon}}{\| \bfX_{[1,n]} \|_{q^\prime}^{\alpha-\epsilon}}
\log(| \bfX_t |/| \bfX_j |) \1(|\bfX_t|/|\bfX_j| > 1)\\
&\le& 
\sum_{t=1}^{n}\sum_{j=1}^{n}
\frac{|\bfX_t|^{q^\prime}}{\| \bfX_{[1,n]} \|_{q^\prime}^{q^\prime}}
\frac{|\bfX_j|^{\alpha-\epsilon}}{\| \bfX_{[1,n]} \|_{q^\prime}^{\alpha-\epsilon}}
\left(\frac{| \bfX_t |}{| \bfX_j |}\right)^\epsilon \1(|\bfX_j|/|\bfX_t| > 1)\\
&=&
\frac{\| \bfX_{[1,n]} \|_{\alpha-2\epsilon}^{\alpha-2\epsilon}}{\| \bfX_{[1,n]} \|_{q^\prime}^{\alpha-\epsilon}} \sum_{t=1}^{n}
\frac{|\bfX_t|^{q^\prime}}{\| \bfX_{[1,n]} \|_{q^\prime}^{q^\prime}}
| \bfX_t |^\epsilon.
\eeao 
Then, by an application of Hölder's inequality, we obtain
\beao 
R &\le&
\frac{\| \bfX_{[1,n]} \|_{\alpha-2\epsilon}^{\alpha-2\epsilon}}{\| \bfX_{[1,n]} \|^{\alpha-\epsilon}_{q^\prime}} 
\left(\sum_{j=1}^{n}
\frac{|\bfX_t|^{pq^\prime}}{\| \bfX_{[1,n]} \|_{q^\prime}^{pq^\prime}}\right)^{1/p} \left(\sum_{j=1}^{n}|\bfX_t|^{p^\prime\epsilon}\right)^{1/p^\prime}. \\
&=&
\frac{\| \bfX_{[1,n]} \|_{\alpha-2\epsilon}^{\alpha-2\epsilon}}{\| \bfX_{[1,n]} \|^{\alpha-\epsilon}_{q^\prime}} 
\frac{\| \bfX_{[1,n]} \|_{pq^\prime}^{q^\prime}}{\| \bfX_{[1,n]} \|^{q^\prime}_{q^\prime}}  \|\bfX_{[1,n]} \|_{p^\prime \epsilon}^{\epsilon} \\
&\le& \frac{\| \bfX_{[1,n]} \|_{\alpha-2\epsilon}^{\alpha-2\epsilon}}{\| \bfX_{[1,n]} \|^{\alpha-\epsilon}_{q^\prime}} 
\|\bfX_{[1,n]} \|_{p^\prime \epsilon}^{\epsilon}
\eeao 
for $p > 1$, and $1/p + 1/p^\prime = 1$. 
Then, we see it is enough to show that for $\eta < \epsilon < 1$,
\beao 
\lim_{n \to \infty}
\frac{
\E\big[\big(\|\bfX_{[1,n]}/x_n\|_{\alpha-2\epsilon}^{\alpha - 2\epsilon}\|\bfX_{[1,n]} /x_n\|_{\epsilon}^{\epsilon}\big)^{1+\eta}\1(\|\bfX_{[1,n]}\|_{q^\prime} > x_n)\big]}{\P(\|\bfX_{[1,n]}\|_{q^\prime} > x_n)} < \infty. 
\eeao 
Moreover, by an application of Hölder's inequality,
\beao \label{eq:remainder}
\lefteqn{
\E\big[
\big(\|\bfX_{[1,n]}/x_n\|_{\alpha-2\epsilon}^{\alpha - 2\epsilon}\|\bfX_{[1,n]} /x_n\|_{\epsilon}^{\epsilon}\big)^{1+\eta}\1(\|\bfX_{[1,n]}\|_{q^\prime} > x_n)
\big] }\\
&\le& 
\E\big[\big(\|\bfX_{[1,n]}/x_n\|_{\alpha-\epsilon}^{\alpha - \epsilon}\big)^{1+\eta}\big]\\
&\le& 
\E\big[\|\bfX_{[1,n]}/x_n\|_{\alpha-\epsilon+\eta}^{\alpha - \epsilon+\eta}
\big],
\eeao 
where the last inequality follow by subadditivity when $\eta < 1$.
Therefore, it is enough to show there exists $\epsilon, \eta > 0$ such that $\alpha - \epsilon - \eta > \epsilon/2$,
\beam \label{eq:alpha:1:term} 
\lim_{n \to \infty}\frac{
\E[\|
\bfX_{[1,n]}/x_n\|_{\alpha-\epsilon}^{\alpha-\epsilon}\1(\|\bfX_{[1,n]}\|_{\alpha-\epsilon-\eta } >x_n) ]}{\P(\|\bfX_{[1,n]}\|_{\alpha-\epsilon} >x_n)} < \infty.
\eeam 
We verify that for all $\epsilon, \eta > 0$ such that $\alpha - \epsilon - \eta > \epsilon/2$, and \eqref{eq:alpha:1:term} holds and an application of Lemma~\ref{lem:verif:S} yields the desired result.
Indeed, notice for all $\delta > 0$,
\beam \label{eq:term:norm:alpha}
\lefteqn{ 
\frac{\E[\|\bfX_{[1,n]}/x_n\|_{\alpha-\epsilon}^{\alpha-\epsilon}\1(\|\bfX_{[1,n]}\|_{\alpha-\epsilon - \eta } >x_n) ]}{\P(\|\bfX_{[1,n]}\|_{\alpha-\epsilon} >x_n)} }\nonumber \\
&=&
\frac{\E[\|\overline{\bfX_{[1,n]}/x_n}^\delta\|_{\alpha-\epsilon}^{\alpha-\epsilon}\1(\|\bfX_{[1,n]}\|_{\alpha-\epsilon- \eta} >x_n) ]}{\P(\|\bfX_{[1,n]}\|_{\alpha-\epsilon} >x_n)}  \nonumber \\
& +&
\frac{\E[\|\underline{ \bfX_{[1,n]}/x_n}_\delta\|^{\alpha-\epsilon}_{\alpha-\epsilon}\1(\|\bfX_{[1,n]}\|_{\alpha-\epsilon- \eta} >x_n) ]}{\P(\|\bfX_{[1,n]}\|_{\alpha-\epsilon} >x_n)} \nonumber  \\
&=& I + II,
\eeam 
and here we recall the notation $\overline{\bfX_{[1,n]}/x_n}^\delta = (\overline{\bfX_1/x_n}^\delta, \dots, \overline{\bfX_{n}/x_n}^\delta)$, and $\underline{\bfX_{[1,n]}/x_n}_\delta = (\underline{\bfX_1/x_n}_\delta, \dots, \underline{\bfX_{n}/x_n}_\delta).$
Next, we treat $I$ and $II$ separately.
For $II$ it is easy to see 
\beao 
I &\le& n \frac{\E[ |\underline{\bfX_1/x_n}_\delta|^{\alpha - \epsilon} ]}{\P(\|\bfX_{[1,n]}\|_{\alpha-\epsilon} >x_n)}  \quad =\quad \frac{n\P(|\bfX_1| > x_n)}{\P(\|\bfX_{[1,n]}\|_{\alpha - \epsilon} > x_n )} \frac{\E[ |\underline{\bfX_1/x_n}_\delta|^{\alpha - \epsilon} ]}{\P(|\bfX_1| > x_n)}. 
\eeao 
Then, by an application of Karamata's theorem we see $\lim_{n \to \infty} II < \infty $.
Now to treat $I$, recall $\E[\|\underline{ \bfX_{[1,n]}/x_n}_\delta\|^{\alpha-\epsilon}_{\alpha-\epsilon}] \to 0$, as $n \to \infty$ since $\E[|\bfX_1|^{\alpha-\epsilon}] < \infty$ and $n/x_n^{\alpha-\epsilon} \to 0$, as $n \to \infty$.
Hence, by Equation~\eqref{eq:moment:cond}, 
we conclude $\lim_{n \to \infty}(I + II )< \infty$, following the notation in \eqref{eq:term:norm:alpha}, and this completes the proof. 
\end{proof}

}}


\begin{thebibliography}{99}

\bibitem{basrak:davis:mikosch:2002}
{\sc Basrak, B., Davis, R.A. and Mikosch, T.}\ (2002)
Regular variation of GARCH processes.
{\em Stochastic Process. Appl. } {\bf 99}, 95--115. 

\bibitem{basrak:davis:mikosch:2002:2}
{\sc Basrak, B., Davis, R.A. and Mikosch, T.}\ (2002)
A characterization of multivariate regular variation. 
{\em Ann. Appl. Probab. } {\bf 12}, 908--920.

\bibitem{basrak:planinic:soulier:2018}
{\sc Basrak, B., Planinic, H., and Soulier, P.}\ (2018). An invariance principle for sums and record times of regularly varying stationary sequences. {\em Probability. Theory Related Fields.} {\bf 127}, 869--914.


\bibitem{basrak:segers:2009}
{\sc Basrak, B. and Segers, J.}\ (2009)
Regularly varying multivariate time series.
{\em Stochastic Process. Appl.} {\bf 119}, 1055--1080.



\bibitem{bingham:goldie:teugels:1987}
{\sc Bingham, N.H., Goldie, C.M. and Teugels, J.L.}\ (1987) 
{\em Regular variation.}
Cambridge University Press, Cambridge.

\bibitem{billingsley:2013}
{\sc Billingsley, P.}\ (2013)
{\em Convergence of probability measures.}
John Wiley \& Sons.


\bibitem{bradley:1988} 
{\sc Bradley, R.C.}\ (1988)
A central limit theorem for stationary $\rho-$mixing sequences with infinite variance. {\em Ann. Probab.} {\em 16}, 313--332.

\bibitem{bradley:2005}
{\sc Bradley, R.C.}\ (2005)
Basic properties of strong mixing conditions. A survey and some open questions. 
{\em Probab. Surv.} {\em 2}, 107--144.

\bibitem{breiman:1965}
{\sc Breiman, L.}\ (1965)
On some limit theorems similar to the arc-sin law.
{\em Theory Probab. Appl.} {\em 10}, 323--331.

\bibitem{brockwell:davis:2016}
{\sc Brockwell, P.J. and Davis, R.A.}\ (2016)
{\em Introduction to time series and forecasting.} Springer.

 

\bibitem{buraczewski:damek:mikosch:2016} 
{\sc Buraczewski, D., Damek, E., and Mikosch, T.}\ (2016). {\em Stochastic models with power-law tails. The Equation $X= AX+ B$.}
Springer Ser. Oper. Res. Financ. Eng., Springer, Cham, 10, 978-3.

\bibitem{buritica:mikosch:wintenberger:2021}
{\sc Buritic\'a, G., Mikosch, T. and Wintenberger, O.}\ (2023)
Large deviations of $\ell^p-$blocks of regularly varying time series and applications to cluster inference 
{\em Stochastic Process. Appl. }
{\bf 161}, 68--101.

\bibitem{buritica:meyer:mikosch:wintenberger:2021}
{\sc Buritic\'a, G. Mikosch, T. Meyer, N. and Wintenberger, O.}\ (2021)
Some variations on the extremal index. 
{\em Zap. Nauchn. Semin. POMI. Volume 501, Probability and Statistics.} {\bf 30}, 52—77. To be translated in J.Math.Sci. (Springer). 
\bibitem{cissokho:kulik:2021}
{\sc Cissokho, Y. and Kulik, R.} \ (2021)
Estimation of cluster functionals for regularly varying time series: sliding block estimators.
{\em Electron. J. Stat.} {\bf 15}, 2777--2831.

\bibitem{kulik:cissokho:2022}
{\sc Cissokho, Y. and Kulik, R. }
\ (2022)
Estimation of cluster functionals for regularly varying time series: Runs estimators.
{\em Electron. J. Stat.} {\bf 16},
3561--3607.

\bibitem{cline:1983}
{\sc Cline, D. B.}\ (1983)
{\em Estimation and linear prediction for regression, autoregression and ARMA with infinite variance data.} PhD Diss. Colorado State University.





\bibitem{dedecker:doukhan:lang:rafael:louhichi:prieur:2007}
{\sc Dedecker, J., Doukhan, P., Lang, G., Rafael, L. R. J., Louhichi, S. and Prieur, C.}\ (2007)
{\em  Weak dependence: With examples and applications.}
Springer, New York.

\bibitem{drees:rootzen:2010}
{\sc Drees, H. and  Rootzén, H.}\ (2010)
Limit theorems for empirical processes of cluster functionals. {\em Ann. Statist. } {\bf 38}, 2145--2186.

\bibitem{drees:janssen:neblung:2021}
{\sc  Drees, H. Janssen, A. and Neblung, S.}\ (2021)
Cluster based inference for extremes of time series.
{\em Stochastic Process. Appl.}.
{\bf 142}, 1--33.


\bibitem{drees:neblung:2021}
{\sc Drees, H. and Neblung, S.}\ (2021)
Asymptotics for sliding block estimators of rare events.
{\em Bernoulli} {\bf 27}, 1239--1269.


\bibitem{embrechts:kluppelberg:mikosch:1997}
{\sc Embrechts, P. Klüppelberg, C. and Mikosch, T.}\ (1997)
{\em Modelling extremal events for insurance and finance.}
Springer-Verlab, Berlin Heidelberg New York
\bibitem{ferro:2003}
{\sc Ferro, C.A.T.}\ (2003)
Statistical methods for clusters of extreme values.
PhD-thesis, Lancaster Univ.
\bibitem{goldie:1991}
{\sc Goldie, C.M.}\ (1991)
Implicit renewal theory and tails of solutions of random equations.
{\em Ann. Appl. Probab. } {\em 1}, 126--166.


\bibitem{haan:resnick:rootzen:vries:1989}
{\sc Haan, L. de, Resnick, S.I., Rootz\'en, H. and Vries, de C.G.}\ (1989)
Extremal behaviour of solutions to stochastic difference equation with applications to ARCH processes.
{\em Stochastic Process. Appl..} {\bf 32}, 213--224.

\bibitem{hsing:1991}
{\sc Hsing, T.}\ (1991)
On tail index estimation using dependent data
{\em Ann. Statist. }
{\bf 1}, 1547--1569.

\bibitem{hsing:1993}
{\sc Hsing, T.}\ (1993)
On some estimates based on sample behavior near high level excursions.
{\em Probab. Theory Related Fields} {\bf 95}, 331--356.


\bibitem{hult:samorodnitsky:2008}
{\sc Hult, H. and Samorodnitsky, G.}\ (2008)
Tail probabilities for infinite series of regularly varying random vectors.
{\em Bernoulli} {\bf 14},
838--864.

\bibitem{hult:samorodnistky:2010}
{\sc Hult, H. and Samorodnitsky, G.}\ (2010) Large deviations for point processes based on stationary sequences with heavy tails. {\em J. Appl. Probab.} {\bf 47}, 1--40.







\bibitem{janssen:segers:2014}
{\sc Janssen, A. and Segers, J.}\ (2014)
Markov Tail Chains. 
{\em J. Appl. Probab. }
{\bf 51}, 1133--1153.
\bibitem{kesten:1973}
{\sc Kesten, H.}\ (1973)
Random difference equations and reneval theory for products of random matrices.
{\em Acta Math.} {\bf 131}, 207--2048.




\bibitem{kulik:soulier:2020}
{\sc Kulik, R. and Soulier, P.}\ (2020)
{\em Heavy-Tailed Time Series}.
Springer, New York.
\bibitem{leadbetter:1983}
{\sc Leadbetter, M.R.}\ (1983)
Extremes and local dependence in stationary sequences.
{\em Probab. Theory Related Fields} {\bf 65}, 291--306.

\bibitem{leadbetter:lindgren:rootzen:1983}
{\sc Leadbetter, M.R., Lindgren, G., and Rootz\'en, H.}\ (1983)
{\em Extremes and related properties of random sequences and processes.} Springer, Berlin. 
\bibitem{meyn:tweedie:1993}
{\sc Meyn, S.P. and Tweedie R.L.}\ (1993)
{\em Markov Chains and Stochastic Stability}.
{Springer, London. }

\bibitem{mikosch:samorodnitsky:2000}
{\sc Mikosch, T. and Samorodnitsky, G.}\ (2000).
The supremum of a negative drift random walk with dependent heavy-tailed steps.
{\em Ann. Appl. Probab. } {\em 10}, 1025--1064.

\bibitem{mikosch:wintenberger:2013}
{\sc Mikosch, T. and Wintenberger, O.}\ (2013)
Precise large deviations for dependent regularly varying sequences. 
{\em Probab. Theory Related Fields} {\bf 156}, 851--887.
\bibitem{mikosch:wintenberger:2014}
{\sc Mikosch, T. and Wintenberger, O.}\ (2014)
The cluster index of regularly varying sequences with
applications to limit theory for functions of multivariate Markov chains.
{\em Probab. Th. Rel. Fields} {\bf 159}, 157--196.


\bibitem{pham:tran:1985}
{\sc Pham, T.D. and Tran, L.T.}\
(1985)
Some mixing properties of time series models.
{\em Stochastic Process. Appl.} {\bf 19}, 297--303.

\bibitem{rio:2017}
{\sc Rio, E.}\ (2017)
{\em Asymptotic theory of weakly dependent random processes.}
Springer, New York.


\bibitem{resnick:2007}
{\sc Resnick, S.I.}\ (2007)
{\em Heavy-tail phenomena: probabilistic and statistical modeling. } 
Springer Science \& Business Media.



\bibitem{robert:2009}
{\sc Robert, C. Y.}\ (2009)
Inference for the limiting cluster size distribution of extreme values.
{\em Ann. Statist. }
{\bf 37}, 271--310.










\bibitem{vandervaart:2000}
{\sc van der Vaart, A. W.}\ (2000)
{\it Asymptotic statistics.}
Cambridge University Press, Cambridge. 


\bibitem{vandervaart:wellner:1996}
{\sc van der Vaart, A. W.  and Wellner, J. A. }\ (1996)
{\it Weak convergence and empirical processes.}
Springer, New York. 

 
\end{thebibliography}
\end{document}